\documentclass[a4paper,11pt]{article}

\usepackage[utf8x]{inputenc}
\usepackage{lmodern}
\usepackage[T1]{fontenc}

\usepackage[frenchb]{babel}
\usepackage[colorlinks]{hyperref}
\hypersetup{linkcolor=blue, citecolor=red}

\usepackage{bbm}
\usepackage[bbgreekl]{mathbbol}
\usepackage{nccrules}
\usepackage[nottoc]{tocbibind}
\usepackage[intlimits,leqno]{amsmath}
\usepackage{amsthm}
\usepackage{amssymb}
\usepackage{mathrsfs}
\usepackage{stmaryrd}
\usepackage{xspace}
\usepackage[all]{xy}
\usepackage{paralist}

\usepackage{geometry}
\newcommand{\Z}{\ensuremath{\mathbb{Z}}}
\newcommand{\Q}{\ensuremath{\mathbb{Q}}}
\newcommand{\R}{\ensuremath{\mathbb{R}}}
\newcommand{\C}{\ensuremath{\mathbb{C}}}

\newcommand{\A}{\ensuremath{\mathbb{A}}}

\newcommand{\topwedge}{\ensuremath{\bigwedge^{\mathrm{max}}}}
\newcommand{\rank}{\ensuremath{\mathrm{rg}\xspace}}
\newcommand{\Aut}{\ensuremath{\mathrm{Aut}}}
\newcommand{\Tr}{\ensuremath{\mathrm{tr}\xspace}}

\newcommand{\Resprod}{\ensuremath{{\prod}'}}

\newcommand{\dd}{\ensuremath{\,\mathrm{d}}}

\newcommand{\angles}[1]{\ensuremath{\langle #1 \rangle}}
\newcommand{\mes}{\ensuremath{\mathrm{mes}}}
\newcommand{\sgn}{\ensuremath{\mathrm{sgn}}}
\newcommand{\Stab}{\ensuremath{\mathrm{Stab}}}

\newcommand{\identity}{\ensuremath{\mathrm{id}}}

\newcommand{\Hom}{\ensuremath{\mathrm{Hom}}}
\newcommand{\End}{\ensuremath{\mathrm{End}}}
\newcommand{\rightiso}{\ensuremath{\stackrel{\sim}{\rightarrow}}}

\newcommand{\Ker}{\ensuremath{\mathrm{ker}\xspace}}
\newcommand{\Coker}{\ensuremath{\mathrm{coker}\xspace}}

\newcommand{\Lie}{\ensuremath{\mathrm{Lie}\xspace}}
\newcommand{\Ad}{\ensuremath{\mathrm{Ad}\xspace}}
\newcommand{\ad}{\ensuremath{\mathrm{ad}\xspace}}

\newcommand{\Spec}{\ensuremath{\mathrm{Spec}\xspace}}

\newcommand{\Gm}{\ensuremath{\mathbb{G}_\mathrm{m}}}

\newcommand{\Supp}{\ensuremath{\mathrm{Supp}}}

\newcommand{\GL}{\ensuremath{\mathrm{GL}}}
\newcommand{\SO}{\ensuremath{\mathrm{SO}}}
\newcommand{\Or}{\ensuremath{\mathrm{O}}}

\newcommand{\U}{\ensuremath{\mathrm{U}}}
\newcommand{\SU}{\ensuremath{\mathrm{SU}}}
\newcommand{\PGL}{\ensuremath{\mathrm{PGL}}}

\newcommand{\SL}{\ensuremath{\mathrm{SL}}}
\newcommand{\Sp}{\ensuremath{\mathrm{Sp}}}

\newcommand{\Mp}{\ensuremath{\widetilde{\mathrm{Sp}}}}

\theoremstyle{plain}
\newtheorem{proposition}{Proposition}[subsection]
\newtheorem{lemma}[proposition]{Lemme}
\newtheorem{theorem}[proposition]{Théorème}
\newtheorem{corollary}[proposition]{Corollaire}

\theoremstyle{definition}
\newtheorem{definition}[proposition]{Définition}
\newtheorem{definition-theorem}[proposition]{Définition-Théorème}
\newtheorem{definition-proposition}[proposition]{Définition-Proposition}

\newtheorem{example}[proposition]{Exemple}

\theoremstyle{remark}
\newtheorem{remark}[proposition]{Remarque}

\newcommand{\bmu}{\ensuremath{\bbmu}}

\newcommand{\noyau}{\ensuremath{\boldsymbol{\varepsilon}}} 
\newcommand{\CVP}{\ensuremath{\stackrel{\mathrm{VP}}{\longleftrightarrow}}} 
\newcommand{\rev}{\ensuremath{\mathbf{p}}} 
\newcommand{\asp}{\ensuremath{\dashrule[.7ex]{2 2 2 2}{.4}}} 

\newcommand{\Wittequiv}{\ensuremath{\stackrel{\text{Witt}}{\sim}}}

\title{La formule des traces stable pour le groupe métaplectique: les termes elliptiques}
\author{Wen-Wei Li}
\date{}

\begin{document}


\maketitle

\begin{abstract}
  Soient $F$ un corps de nombres et $\tilde{G}$ le revêtement métaplectique de Weil de $G(\A_F) = \Sp(2n,\A_F)$. Nous stabilisons les termes elliptiques semi-simples de la partie spécifique de la formule des trace pour $\tilde{G}$ à l'aide du transfert des fonctions test et du lemme fondamental établis dans un article antérieur. Nous obtenons ainsi une expression en termes des intégrales orbitales stables le long de certaines classes $F$-elliptiques dans $\SO(2n'+1) \times \SO(2n''+1)$, ce que nous appelons équi-singulières, pour les paires $(n', n'')$ avec $n'+n''=n$. Une formule simple des coefficients est aussi obtenue.
\end{abstract}

\begin{flushleft}
  \small MSC classification (2010): \textbf{11F72}, 11F27, 11F70.
\end{flushleft}

\tableofcontents

\section{Introduction}
Soient $F$ un corps de nombres et $\A = \A_F$ son anneau d'adèles. La formule des traces d'Arthur-Selberg est un outil indispensable pour l'étude des représentations automorphes de $G(\A)$, où $G$ est un $F$-groupe réductif connexe. Au-delà des groupes anisotropes, de $\GL(n)$, de $\PGL(n)$ ou de leurs formes intérieures, on bute immédiatement sur le problème de \emph{stabilité} pour utiliser la formule des traces. Expliquons-le. Pour un corps $F$ et un $F$-groupe $G$ quelconques, deux éléments dans $G(F)$ sont dits stablement conjugués s'ils sont conjugués dans $G(\bar{F})$. Le but de la stabilisation est d'exprimer le côté géométrique de la formule des traces en termes des intégrales orbitales locales stables, associées à certains $F$-groupes quasi-déployés plus petits que $G$. La théorie de l'endoscopie est inventée par Langlands pour l'accomplir. Le point clé est de transférer les intégrales orbitales semi-simples régulières sur $G$ vers des intégrales orbitales stables sur les groupes dits endoscopiques,
 qui sont associés aux données endoscopiques de $G$. On renvoie à \cite{CHLN11} pour une introduction excellente.

Cette théorie a eu un grand succès dès lors. Elle donne aussi des cas abordables du principe de fonctorialité de Langlands. Signalons qu'en admettant le transfert et le \guillemotleft lemme fondamental\guillemotright\;pour les intégrales orbitales, la stabilisation des termes elliptiques semi-simples réguliers de la formule des traces est achevée par Langlands \cite{Lan83}; quant aux termes elliptiques, c'est dû à Kottwitz \cite{Ko86}, puis à Labesse \cite{Lab04} dans le cadre tordu. Remarquons aussi qu'une stabilisation complète est obtenue par Arthur \cite{Ar02}, mais les facteurs de transfert étendus \cite[p.225 et 228]{Ar02} dans ses théorèmes globaux ne sont pas explicités.

Pour nous, le problème est de la généraliser aux revêtements des groupes réductifs connexes, disons de la forme
$$ 1 \to \bmu_m \to \tilde{G} \xrightarrow{\rev} G(\A) \to 1 $$
en tant qu'une extension centrale des groupes localement compacts, où $m \in \Z_{\geq 1}$ et $\bmu_m$ est le groupe des racines $m$-ièmes de l'unité dans $\C^\times$. Il suffit de regarder les représentations $\pi$ de $\tilde{G}$ telles que $\pi(\noyau\tilde{x}) = \noyau\pi(\tilde{x})$ pour tout $\noyau \in \bmu_m$. De telles représentations sont dites spécifiques. Du côté géométrique, on se limite donc aux intégrales orbitales des fonctions test $f$ anti-spécifiques au sens que $f(\noyau\tilde{x}) = \noyau^{-1} f(\tilde{x})$ pour tout $\noyau \in \bmu_m$.

En fait, ce que nous étudierons dans cet article est le revêtement métaplectique $\rev: \Mp(W) \to \Sp(W, \A)$ de Weil \cite{Weil64}, où $W$ est un $F$-espace symplectique de dimension $2n$, ainsi que son analogue pour un corps local de caractéristique nulle. À proprement parler, il faut aussi fixer un caractère additif non-trivial $\psi = \prod_v \psi_v: \A/F \to \C^\times$. Ces revêtements interviennent de façon naturelle dans divers problèmes en arithmétique et en analyse harmonique, eg. ceux reliés aux fonctions $\vartheta$ ou à la correspondance de Howe. On prend garde que notre revêtement métaplectique est une extension de $\Sp(W)$ par $\bmu_8$, au lieu de celle par $\bmu_2$ ou $\C^\times$ dans la littérature.

La formule des traces invariante d'Arthur et les fondements de l'analyse harmonique pour les revêtements sont traités dans \cite{Li11a, Li12b}. Le cas du revêtement $\tilde{G} = \Mp(W)$ est relativement simple. Pour nous, le problème est d'établir une théorie convenable de l'endoscopie pour $\tilde{G}$, prouver le transfert et le lemme fondamental des intégrales orbitales, et puis stabiliser la partie elliptique de la formule des traces. Pour l'endoscopie, il manque deux ingrédients cruciaux pour les revêtements: le groupe dual de Langlands, et la conjugaison stable.

Pour $\Mp(W)$, l'attaque s'est initiée sur le front local. Dans cet article, $\SO(2n+1)$ signifie toujours sa forme déployée. C'est bien connu des spécialistes que $\Mp(W)$ et $\SO(2n+1)$ doivent avoir le même groupe dual $\Sp(2n, \C)$, avec l'action galoisienne triviale. Adams \cite{Ad98} définit la conjugaison stable pour $\Mp(W)$ au cas $F=\R$, et obtient les identités endoscopiques de caractères pour le groupe endoscopique principal $\SO(2n+1)$. Ensuite, Renard \cite{Re98, Re99} obtient le transfert des intégrales orbitales pour les autres groupes endoscopiques elliptiques. Le cas non-archimédien avec $n=1$ est traité dans la thèse de Schultz \cite{Sc98} à l'aide de la restriction de $\widetilde{GL}(2)$ et de la correspondance de Flicker, mais il ne traite pas les intégrales orbitales.

Un formalisme unifié pour tout $F$ de caractéristique nulle est proposé dans \cite{Li11}. Les données endoscopiques sont en bijection avec les paires $(n', n'') \in (\Z_{\geq 0})^2$ avec $n'+n''=n$; le groupe endoscopique associé est
$$ H = H_{n',n''} := \SO(2n'+1) \times \SO(2n''+1). $$

On a également une correspondance entre classes semi-simples stables entre $H(F)$ et $G(F)$, notée $\gamma \leftrightarrow \delta$. Les facteurs de transfert $\Delta$ sont définis à l'aide du caractère de la représentation de Weil $\omega_\psi$. Le transfert et le lemme fondamental des intégrales orbitales sont aussi obtenus en général. Il doit être clair que les résultats de \textit{loc. cit.} conduisent aisément à la stabilisation des termes elliptiques semi-simples réguliers de la formule des traces de $\Mp(W)$. Cependant, les termes elliptiques contiennent la distribution $f \mapsto f(1)$, dont l'intérêt en analyse harmonique est évident.

Par ailleurs, Hiraga et Ikeda ont aussi considéré la stabilisation de la formule des traces pour le revêtement à $m$ feuillets ($m \in \Z_{\geq 1}$) de $\SL(2)$ en supposant que $F$ contient les $m$-ièmes racines de l'unité (non-publié).

Cela conclut nos remarques historiques. Passons à la stabilisation des termes elliptiques pour $\Mp(W)$. Le but de cet article est le suivant (Théorème \ref{prop:FT-stable})
$$ T^{\tilde{G}}_{\mathrm{ell}}(f) = \sum_{\substack{(n',n'') \\ H := H_{n',n''}}} \iota(\tilde{G}, H) ST^H_\text{équi,ell}(f^H) $$
pour toute fonction test $f \in C^\infty_c(\tilde{G})$ anti-spécifique, où
\begin{itemize}
  \item les mesures de Tamagawa sont utilisées partout pour définir les intégrales orbitales, et $\tau(\cdots)$ désigne le nombre de Tamagawa;
  \item $T^{\tilde{G}}_\text{ell}(f) := \sum_{\delta: \text{ell}} \tau(G_\delta) J_{\tilde{G}}(\gamma, f)$ est la partie elliptique spécifique de la formule des traces pour $\tilde{G}$, ici $G_\delta$ est le commutant connexe de $\delta$;
  \item $f^H \in C^\infty_c(H(\A))$ est un transfert de $f$;
  \item $ST^H_\text{équi,ell}(f^H) = \tau(H) \sum_\gamma S_H(\gamma, f^H)$ est le morceau équi-singulier elliptique de la distribution stable $ST^H_\text{ell}$ sur $H(\A)$ qui fait partie du côté géométrique de la formule des traces stable de $H$;
  \item le coefficient $\iota(\tilde{G}, H)$ est égal à $\tau(H)^{-1}$, avec la même hypothèse que les mesures de Tamagawa sont mises partout; plus précisément, lorsque $n \geq 1$
    $$ \iota(\tilde{G}, H) = \tau(H)^{-1} = \begin{cases}
	\frac{1}{4}, & \text{ si } n',n'' \geq 1, \\
	\frac{1}{2}, & \text{ sinon,} 
      \end{cases} $$
\end{itemize}
Pour la notion d'équi-singularité, on renvoie à la Définition \ref{def:equi-sing}. Remarquons que cela s'appelle $(G,H)$-régularité dans \cite{Ko86}; notre terminologie est empruntée à \cite{LS2}.

Expliquons brièvement la motivation. Pour les groupes réductifs connexes (algébriques), la stabilisation est une étape nécessaire pour attaquer la conjecture de fonctorialité de Langlands via la comparaison de formules des traces; on renvoie à l'introduction de \cite{Ar02} pour une vue d'ensemble. On peut espérer une telle stratégie pour les groupes métaplectiques, ce que l'on reporte aux travaux ultérieurs.

Selon Labesse \cite{Lab04}, la stabilisation des termes elliptiques se divise en deux étapes.
\begin{enumerate}
  \item \emph{Pré-stabilisation}. C'est une procédure globale de nature cohomologique, d'où viennent les coefficients $\iota(\tilde{G}, H)$.
  \item \emph{Transfert équi-singulier}. C'est un problème local: voir \eqref{eqn:transfert-equi-pre}.
\end{enumerate}

Beaucoup de confusions sont dues à l'asymétrie entre $n'$ et $n''$ dans les données endoscopiques elliptiques; notons qu'il n'en est pas le cas pour l'endoscopie de $\SO(2n+1)$. Le principe proposé dans \cite{Li12} dit que l'on suppose non seulement que le groupe dual $\widehat{\tilde{G}}$ est $\Sp(2n, \C)$, mais on doit \guillemotleft prétendre\guillemotright\, que $Z_{\widehat{\tilde{G}}} = \{1\}$! Par exemple, si l'on le compare avec le formalisme formel de l'endoscopie dans \cite{LS1}, alors ce n'est pas seulement la classe $s Z_{\widehat{\tilde{G}}}$ qui compte, mais bien $s$. Il résulte que
\begin{itemize}
  \item les données endoscopiques $(n', n'')$ et $(n'', n')$ sont effectivement différentes;
  \item il n'y a pas d'automorphismes extérieurs pour toute donnée endoscopique elliptique $(n', n'')$, donc la formule pour $\iota(\tilde{G}, H)$ est plus simple que celle pour $\SO(2n+1)$.
\end{itemize}

Signalons qu'il existe toujours un rapport entre $(n', n'')$ et $(n'', n')$; voir \cite[Proposition 5.16]{Li11}. Nous l'utiliserons dans la démonstration du transfert équi-singulier pour simplifier les arguments.

Bien sûr, la meilleure justification de ces principes sera une démonstration complète. Pour cela, nous donnerons des définitions et preuves détaillées et indépendantes pour la pré-stabilisation. Notre approche est modelée sur celle de Labesse \cite{Lab01, Lab04}, y compris son usage systématique de la cohomologie galoisienne abélianisée. Mais notre cas est beaucoup plus simple.

La pré-stabilisation permet d'écrire $T^{\tilde{G}}_\text{ell}(f)$ en termes des $\kappa$-intégrales orbitales $J^\kappa(\delta, f)$. L'étape suivante est d'associer à chaque $(\delta, \kappa)$ une paire convenable $(n',n'',\gamma)$, ce qu'assure le Lemme \ref{prop:bijection}. Ici encore, le principe d'asymétrie est manifeste. Ensuite, on peut appliquer le transfert équi-singulier avec des propriétés du facteur de transfert, pour transformer $J^\kappa(\delta, f)$ en une intégrale orbitale stable $S_H(\gamma, f^H)$. Cela achève la stabilisation.

Reste donc à discuter le transfert équi-singulier. Supposons que $F$ est local et $\Mp(W)$ est le groupe métaplectique local. On le fera encore en deux étapes.
\begin{enumerate}
  \item Définir le facteur de transfert $\Delta(\gamma, \tilde{\delta})$ pour $\gamma$ et $\delta$ en correspondance équi-singulière, puis montrer des propriétés nécessaires, eg. la propriété du cocycle, la formule du produit, etc. Ceci est accompli dans le Théorème \ref{prop:facteur}. L'idée est de se ramener au cas semi-simple régulier, où les facteurs de transfert sont déjà définis dans \cite{Li11}, par passage à la limite. Une esquisse de cette méthode se trouve déjà dans \cite[\S 2.4]{LS2}.

  \item Établir le transfert équi-singulier (Théorème \ref{prop:transfert-equi}): supposons que tous les commutants connexes $H_\gamma$, $G_\delta$ tels que $\gamma \leftrightarrow \delta$ est équi-singulier sont munis de mesures de Haar convenables (à préciser plus tard), alors
  \begin{gather}\label{eqn:transfert-equi-pre}
    \sum_{\substack{\delta \in \Gamma_{\mathrm{ss}}(G) \\ \delta \leftrightarrow \gamma}} \Delta(\gamma, \tilde{\delta}) e(G_\delta) J_{\tilde{G}}(\tilde{\delta}, f) = |2|^{-t}_F \cdot S_H(\gamma, f^H)
  \end{gather}
  où $2t$ est la somme des multiplicités des valeurs propres $\pm 1$ de $\delta$, pour n'importe quel $\delta \leftrightarrow \gamma$. À part le facteur $|2|^{-t}_F$, ceci se trouve également dans \cite{Ko86} et \cite[\S 2.4]{LS2}.
\end{enumerate}

Comment démontre-t-on cette égalité? Qu'est-ce que la raison d'être du facteur $|2|^{-t}_F$? Vient à présent l'une des différences essentielles entre l'endoscopie standard et celle pour $\Mp(W)$.

Regardons d'abord le cas de l'endoscopie standard. La notion d'équi-singularité (ou $(G,H)$-régularité) est conçue de sorte que si $\gamma \in H(F)_\text{ss}$ et $\delta \in G(F)_\text{ss}$ sont en correspondance équi-singulière, alors $H_\gamma$ et $G_\delta$ sont reliés par torsion intérieure. Un argument standard de descente nous ramène, pour l'essentiel, au cas où $\gamma=1$, $\delta=1$ et $H$ est la forme quasi-déployée de $G$. Les mesures de Haar sont choisies de façon compatible avec la torsion intérieure. Pour déduire le raccordement des intégrales orbitales en $1$ de celui des points semi-simples réguliers, on emploie les outils suivants.
\begin{itemize}
  \item La formule de limite de Harish-Chandra \cite[\S 17]{HC75} si $F$ est archimédien. Des opérateurs différentiels et signes interviennent dans ladite formule. Cependant, il s'avère que cela se comporte bien avec la torsion intérieure \cite[Lemma 2.4A]{LS2}.
  \item La formule de Rogawski \cite{Ro81} pour le germe de Shalika de degré $0$ si $F$ est non-archimédien. Ce résultat est plus facile à utiliser. Toutefois il faut prendre garde aux choix de mesures de Haar. Cf. \cite{Ko88}.
\end{itemize}

Considérons le cas métaplectique. Pour avoir la bijection cruciale du Lemme \ref{prop:bijection}, il faut permettre que, pour $\gamma \leftrightarrow \delta$ équi-singulier avec $H_\epsilon$ quasi-déployé, les commutants ont les décompositions
\begin{align*}
  H_\gamma & = U_H \times \SO(2a+1) \times \SO(2b+1), \\
  G_\delta & = U_G \times \Sp(2a) \times \Sp(2b),
\end{align*}
telles que les parties dites unitaires $U_H$ et $U_G$ sont reliées par torsion intérieure. En particulier, il faut aussi comparer les intégrales orbitales entre groupes de la forme $\Sp(2n)$, $\SO(2n+1)$. Ceci forme le \guillemotleft noyau dur\guillemotright\;du transfert équi-singulier. On a donc besoin de
\begin{itemize}
  \item choisir des mesures de Haar de façon canonique sur $\Sp(2n)$ et toutes les formes intérieures de $\SO(2n+1)$ sur tout corps local $F$ de caractéristique nulle, y compris $\C$;
  \item montrer que la mesure de Tamagawa se décompose en les mesures choisies, à une constante inoffensive près;
  \item comparer les formules de Harish-Chandra ou de Rogawski avec ces mesures.
\end{itemize}

Dans cet article, nous préférons la recette de Gross \cite{Gr97} pour définir les mesures canoniques locales sur chaque $F$-groupe réductif connexe. On a aussi des équations fonctionnelles pour ces mesures, ainsi qu'une décomposition de la mesure de Tamagawa en termes des mesures canoniques locales. Ces équations s'expriment en termes des invariants cohomologiques et des motifs d'Artin-Tate associés par Gross au groupe en question. Au fond, tous nos arguments reposent sur le fait que les groupes de type $\mathbf{B}_n$ et $\mathbf{C}_n$ ont le même motif d'Artin-Tate $\Q(-1) \oplus \Q(-3) \oplus \cdots \oplus \Q(1-2n)$. On peut penser que le facteur $|2|^{-t}_F$ y apparaît à cause de la différence des longueurs des racines.

Remarquons que cette situation s'inscrit dans l'endoscopie non standard pour les systèmes de racines de type $\mathbf{B}_n$ et $\mathbf{C}_n$, ce qui est déjà utilisé dans \cite{Li11}. L'endoscopie non standard est un ingrédient crucial dans \cite{Wa08} pour l'étude de l'endoscopie tordue via la descente. Il nous paraît que de tels calculs avec la formule de Harish-Chandra/Rogawski n'aient jamais été faits au cadre non standard.

\subsection*{Organisation de cet article}
Dans le \S\ref{sec:revue}, nous passons en revue les résultats de \cite{Li11} sur l'endoscopie pour $\Mp(W)$. Nous rappelons le transfert et le lemme fondamental qui constituent le point de départ de ce travail. De plus, nous corrigeons une erreur mineure de \cite{Li11, Li12} dans la Remarque \ref{rem:erratum}.

Le \S\ref{sec:coho-mes} est une ratatouille de rappels sur les accessoires cohomologiques de Labesse, les mesures sur ces objets, les mesures d'Euler-Poincaré de Serre, les mesures canoniques locales de Gross ainsi que leurs équations fonctionnelles; pour cela, on a aussi besoin d'introduire le motif associé à un groupe réductif connexe. Nous établissons aussi certains résultats utiles qui ne se trouvent pas dans la littérature standard.

Dans le \S\ref{sec:transfert-es}, le formalisme de base du transfert équi-singulier est mis en place. La notion de conjugaison stable est généralisée à tout élément semi-simple de $\Mp(W)$, et le facteur de transfert équi-singulier est étudié. Le Théorème \ref{prop:transfert-equi} y est énoncé.  Nous effectuons ensuite la première étape de la preuve, à savoir la descente semi-simple.

Avant d'entamer les preuves du transfert, dans le \S\ref{sec:stabilisation} nous montrons comment le formalisme de Labesse \cite{Lab04} peut s'adapter au cas métaplectique. En supposant le Théorème \ref{prop:transfert-equi}, nous stabilisons la partie elliptique de la formule des traces pour $\Mp(W)$ et donnons les formules explicites pour les coefficients $\iota(\tilde{G}, H)$. Là se trouve le résultat principal de cet article, le Théorème \ref{prop:FT-stable}.

Dans les \S\ref{sec:reel}, \S\ref{sec:cplx} et \S\ref{sec:nonarch}, nous prouvons le Théorème \ref{prop:transfert-equi} pour tout corps local $F$ de caractéristique nulle. Les outils nécessaires sont aussi rappelés.

\subsection*{Remerciements}
L'auteur tient à remercier Wee Teck Gan, Atsushi Ichino et Huajun Lü pour des conseils utiles. Il remercie Jean-Loup Waldspurger qui a donné une longue liste de commentaires. L'auteur est aussi reconnaissant envers le rapporteur anonyme pour sa lecture minutieuse et ses remarques pertinentes du manuscrit.

\subsection*{Conventions}
\paragraph*{Algèbre} Soient $E$ une extension galoisienne d'un corps $F$, le groupe de Galois pour $E/F$ est désigné par $\Gamma_{E/F}$. Si une clôture séparable de $F$ est fixée, le groupe de Galois absolu est désigné par $\Gamma_F$. Si $F$ est un corps local (archimédien ou non-archimédien), la valeur absolue normalisée de $F$ est désignée par $|\cdot|_F$. Si $F$ est un corps global, les places de $F$ sont toujours désignées par le symbole $v$; le complété de $F$ en une place $v$ est noté $F_v$, et son anneau d'adèles est noté $\A = \A_F := \Resprod_v F_v$. Un caractère de $F$ (si $F$ est local) ou $\A/F$ (si $F$ est global) signifie toujours un caractère additif continu de ce groupe.

Les formes quadratiques, symplectiques et hermitiennes sont toujours non-dégénérées dans cet article. Soit $k$ un anneau commutatif avec $1$. L'algèbre symétrique associée à un $k$-module est désignée par $\text{Sym}(\cdot)$. La puissance extérieure de degré maximal d'un $k$-module libre de type fini est désignée par $\topwedge$.

\paragraph*{Schémas} Soient $k$ un anneau commutatif avec $1$ et $A$ une $k$-algèbre. L'ensemble des $A$-points d'un $k$-schéma $X$ est noté par $X(A)$. On ne considère pas les schémas sur des bases plus générales. Si $A=\C$, on identifie $X$ et $X(\C)$. Si $k$ est un corps local, on munit $X(k)$ de la topologie induite par celle de $k$.

Soit $k'/k$ une extension d'anneaux commutatifs avec $1$. Le foncteur de restriction des scalaires (à la Weil) correspondant est noté par $\text{Res}_{k'/k}$, quoiqu'on l'omettra souvent dans les notations. D'autre part, le changement de base est noté par $\cdot \times_k k'$. Si $X$ est un $\R$-schéma, on écrit $X_\C := X \times_\R \C$.

\paragraph*{Groupes algébriques} Soit $G$ un $k$-schéma en groupes. L'algèbre de Lie de $G$ est notée de façon gothique $\mathfrak{g} := \text{Lie}(G)$. L'action adjointe de $G$ sur lui-même ou sur $\mathfrak{g}$ est notée par $\Ad$. Le centre de $G$ est noté $Z_G$, et $G_\text{AD} := G/Z_G$ est le groupe adjoint. La composante connexe neutre de $G$ est notée $G^\circ$ si $k$ est un corps; on utilisera la généralisation usuelle de $G^\circ$ lorsque $G$ est un schéma en groupes raisonnable sur un anneau à valuation discrète $k$ (eg. ceux dans la théorie de Bruhat-Tits). Les normalisateurs et centralisateurs par rapport à $G$ sont désignés par $N_G(\cdot)$ et $Z_G(\cdot)$. De telles notations s'appliquent aussi aux groupes abstraits.

Soit $F$ un corps. Si $\delta \in G(F)$, on note le commutant $G^\delta := Z_G(\delta)$ et le commutant connexe $G_\delta := Z_G(\delta)^\circ$. Supposons que $G$ est un $F$-groupe réductif connexe. On note $G_\text{SC} \to G_\text{der}$ le revêtement simplement connexe du groupe dérivé de $G$. L'ouvert de Zariski dans $G$ des éléments semi-simples réguliers est noté $G_\text{reg}$. Plus généralement, si $\mathcal{U}$ est un sous-ensemble de $G$ (resp. de $G(k)$), on note $\mathcal{U}_\text{reg} := \mathcal{U} \cap G_\text{reg}$ (resp. $\mathcal{U} \cap G_\text{reg}(k)$). Un élément semi-simple régulier $\delta$ est dit fortement régulier si $G^\delta = G_\delta$, et on désigne l'ouvert dense dans $G$ des éléments semi-simples fortement réguliers par $G_\text{freg}$. L'ensemble des éléments semi-simples dans $G(F)$ est noté $G(F)_\text{ss}$. Un élément semi-simple $\delta$ est dit $F$-elliptique si $Z^\circ_{G_\delta}$ et $Z^\circ_G$ ont la même partie déployée.

On considéra des revêtements $\rev: \tilde{G} \twoheadrightarrow G(F)$ lorsque $F$ est un corps local. C'est une extension centrale des groupes localement compacts avec $\Ker(\rev) \subset \C^\times$ fini. On définit $\tilde{G}_\text{reg}$, $\tilde{G}_\text{ss}$, etc. en prenant les images réciproques des versions pour $G(F)$. Les éléments dans le revêtement seront toujours dotés d'un $\sim$, eg. $\tilde{x} \in \Mp(W)$, et leurs images par $\rev$ seront notés par $x \in \Sp(W)$, etc.

Si $F$ est un corps local et $G$ est un $F$-groupe réductif connexe, les représentations de $G(F)$ sont toujours supposées dans la bonne catégorie, eg. lisses de longueur finie. \textit{Idem} pour les revêtements.

\paragraph*{Classes de conjugaison} Dans cet article, deux éléments $\delta_1, \delta_2 \in G(F)$ sont dits stablement conjugués s'ils sont conjugués dans $G(\bar{F})$ où $\bar{F}$ signifie une clôture algébrique de $F$.

L'ensemble des classes de conjugaison dans $G(F)$ est noté $\Gamma(G)$. Il admet les sous-ensembles $\Gamma_\text{ss}(G)$, $\Gamma_\text{reg}(G)$, $\Gamma_\text{ell}(G)$ des éléments semi-simples, semi-simples réguliers, $F$-elliptiques, etc. Pour les classes de conjugaison stables, on a les variantes $\Sigma(G)$, $\Sigma_\text{ss}(G)$, $\Sigma_\text{reg}(G)$, etc. On écrit $\Ad(g)x = gxg^{-1}$.

Signalons que la notion de conjugaison stable ici est différente que celle dans \cite{Ko82}, lorsque le commutant est non-connexe. Il vaut peut-être mieux l'appeler conjugaison géométrique. Or il n'y a pas de différence pour les classes en correspondance équi-singulière considérées dans cet article, d'après (i) de Proposition \ref{prop:equi-sing-commutants}.

\paragraph*{Compléments} On désigne par $\Gm = \GL(1)$ le groupe multiplicatif. Soit $F$ un corps de caractéristique autre que $2$. Les notations usuelles $\GL$, $\SO$, $\Sp$, $\U$, $\SU$, etc. sont employées pour désigner les $F$-groupes linéaires généraux, orthogonaux spéciaux, symplectiques, unitaires, ainsi que les groupe dérivés de groupes unitaires, etc. Plus généralement, soit $L$ une $F$-algèbre commutative étale à involution $\sigma$, dont $L^\sharp$ est la sous-algèbre des points fixes, la notation $U_{L/L^\sharp}(V, h)$ sera utilisée pour désigner le groupe unitaire associé à une forme $(L, \sigma)$-hermitienne ou anti-hermitienne $h: V \times V \to L$. C'est un $L^\sharp$-groupe réductif connexe. On les regarde comme des $F$-groupes au moyen de la restriction des scalaires, dont la notation $\text{Res}_{L^\sharp/F}$ sera systématiquement omise.

Nous travaillerons aussi avec les complexes et leurs cohomologies. Dans cet article, le terme à droite $A$ d'un complexe borné $[\cdots \to A]$ est toujours placé en degré zéro. Tel est aussi la convention dans \cite{Lab01,Lab04}.

Nous utiliserons aussi le $\delta$ de Kronecker: on pose $\delta_{x,y}=1$ si $x=y$, et $\delta_{x,y}=0$ sinon.

\section{Revue de l'endoscopie}\label{sec:revue}
Cette section donne un survol rapide des contenus de \cite[\S\S 2-5]{Li11}. Pour les détails sur le revêtement métaplectique et la représentation de Weil, on renvoie le lecteur à \cite{Wa88}.

\subsection{Groupes métaplectiques}\label{sec:groupes}
Soit $F$ un corps de caractéristique autre que $2$. Un $F$-espace symplectique est une paire $(W, \angles{\cdot|\cdot})$ où $W$ est un $F$-espace vectoriel de dimension finie et $\angles{\cdot|\cdot}: W \times W \to F$ est une forme bilinéaire non-dégénérée anti-symétrique. On note $\Sp(W) := \Aut(W, \angles{\cdot|\cdot}) \subset \GL(W)$ le groupe symplectique correspondant, vu comme un $F$-groupe réductif connexe. Afin d'alléger les notations, on identifie parfois $\Sp(W)$ au groupe $\Sp(W,F)$ de ses $F$-points lorsqu'il n'y a aucune confusion à craindre. Puisque les espaces symplectiques sont classifiés par leurs dimensions, on écrit parfois $\Sp(2n)$ au lieu de $\Sp(W)$ si $\dim_F W = 2n$.

\subsubsection{Le cas local}
Supposons que $F$ est un corps local. Fixons un caractère unitaire $\psi: F \to \C^\times$, qui est non-trivial. À un $F$-espace symplectique $(W,\angles{\cdot|\cdot})$ on construit le revêtement métaplectique (de Weil) à huit feuillets (cf. \cite[\S 2.2]{Li11})
$$ 1 \to \bmu_8 \to \Mp(W) \xrightarrow{\rev} \Sp(W) \to 1 $$
qui est une extension centrale des groupes localement compacts. Ici on désigne, pour tout $m \in \Z$,
$$ \bmu_m := \{ \noyau \in \C^\times : \noyau^m = 1 \}. $$

Ce revêtement est scindé si et seulement si $F = \C$. Signalons que le revêtement peut être réduit à une extension centrale de $\Sp(W)$ par $\bmu_2$; tel est la convention usuelle dans la littérature. Il s'avérera que notre revêtement à huit feuillets est beaucoup plus commode pour la stabilisation.

Le groupe $\Mp(W)$ porte une représentation importante $\omega_\psi$, qui s'appelle la représentation de Weil. Elle se décompose en les morceaux pairs et impairs
$$ \omega_\psi = \omega^+_\psi \oplus \omega^-_\psi $$
qui correspondent aux fonctions paires et impaires dans le modèle de Schrödinger, respectivement.

De plus, $\omega^\pm_\psi$ sont des représentations unitaires irréductibles, dites spécifiques au sens que $\omega^\pm_\psi(\noyau)=\noyau\cdot\identity$ pour tout $\noyau \in \bmu_8$. On désigne leurs caractères par
$$ \Theta^\pm_\psi := \Tr(\omega^\pm_\psi) $$
en tant que des distributions spécifiques sur $\Mp(W)$. À proprement parler, cela dépend du choix d'une mesure de Haar sur $\Mp(W)$, ce que nous préciserons plus loin. On pose $\Theta_\psi := \Theta^+_\psi + \Theta^-_\psi$, c'est le caractère de la représentation de Weil.

En utilisant le modèle de Schrödinger pour $\omega_\psi$, on voit qu'il existe une image réciproque canonique de $-1 \in \Sp(W)$ dans le revêtement à huit feuillets $\Mp(W)$, notée toujours par $-1$. Elle est caractérisée par
$$ \omega^\pm_\psi(-1) = \pm \identity. $$
Par abus de notation, on note $-\tilde{x} := (-1) \cdot \tilde{x}$ pour tout $\tilde{x} \in \Mp(W)$.

\begin{definition}
  On pose
  $$ C^\infty_{c,\asp}(\tilde{G}) := \left\{ f \in C^\infty_c(\tilde{G}) : \forall \noyau \in \bmu_8, \;  f(\noyau\tilde{x})=\noyau^{-1}f(\tilde{x}) \right\}. $$
  Les fonctions dans $C^\infty_{c,\asp}(\tilde{G})$ sont dites anti-spécifiques. Ce sont les fonctions test que nous voulons transférer aux groupes endoscopiques elliptiques.

  Signalons une notion liée. Une fonction $\phi: \tilde{G} \to \C$ est dite spécifique si $\phi(\noyau\tilde{x}) = \noyau\phi(\tilde{x})$. Cette notion se généralise aux distributions sur $\tilde{G}$.
\end{definition}

\begin{remark}\label{rem:bonte}
  Notons une propriété importante de $\tilde{G}$. Soient $\tilde{x}, \tilde{y} \in \tilde{G}$ avec images $x, y \in G(F)$. Alors $\tilde{x}, \tilde{y}$ commutent si et seulement si $x, y$ commutent. En particulier, le commutant de $\tilde{x}$ dans $\tilde{G}$ est l'image réciproque de $G^x(F)$. Autrement dit, tout élément est bon au sens de \cite[Définition 2.6.1]{Li14a}.
\end{remark}

\subsubsection{Le cas non-ramifié}
Supposons que $F$ est non-archimédien de caractéristique résiduelle autre que $2$. Notons $\mathfrak{o}_F$ l'anneau des entiers dans $F$ et $\mathfrak{p}_F$ son idéal maximal. Soit $L \subset W$ un $\mathfrak{o}_F$-réseau et notons $L^* := \{w \in W : \angles{w,L} \subset \mathfrak{o}_F \}$. Si $L$ est auto-dual au sens que $L=L^*$, alors
$$ K = K_L := \Stab_{\Sp(W)}(L) $$
est un sous-groupe hyperspécial de $\Sp(W)$. Nous ne considérons que les sous-groupes hyperspéciaux ainsi obtenus.

Supposons de plus que $\psi$ est de conducteur $\mathfrak{o}_F$, à savoir que $\psi|_{\mathfrak{o}_F} \equiv 1$ mais $\psi|_{\mathfrak{p}^{-1}_F} \not\equiv 1$. Alors le modèle latticiel pour $\omega_\psi$ fournit un scindage $K \hookrightarrow \Mp(W)$ du revêtement métaplectique au-dessus de $K$.

On note $f_K$ l'unique fonction dans $C^\infty_{c,\asp}(\Mp(W))$ telle que $\Supp(f_K) \subset \rev^{-1}(K)$ et $f_K|_K \equiv 1$. C'est l'unité de l'algèbre de Hecke sphérique anti-spécifique de $\Mp(W)$, qui est définie par rapport au produit de convolution pour la mesure de Haar sur $\Mp(W)$ avec $\mes(\rev^{-1}(K))=1$.

\subsubsection{Le cas global}
Supposons maintenant que $F$ est un corps global. Notons $\A = \A_F$ son anneau d'adèles et fixons un caractère unitaire non-trivial $\psi: \A/F \to \C^\times$. Il admet une décomposition $\psi = \prod_v \psi_v$ en des caractères locaux. Le revêtement métaplectique adélique est une extension centrale de groupes localement compacts munie d'un scindage
$$\xymatrix{
  1 \ar[r] & \bmu_8 \ar[r] & \Mp(W,\A) \ar[r]^{\rev} & \Sp(W,\A) \ar[r] & 1 \\
  & & & \Sp(W,F) \ar[lu]^{\exists} \ar@{^{(}->}[u] &
}$$
Le scindage ici $\Sp(W,F) \hookrightarrow \Mp(W,\A)$ est forcément unique car $\Sp(W,F)$ est parfait. Nous regardons $\Sp(W,F)$ comme un sous-groupe discret de $\Mp(W,\A)$.

Notons toujours $\mathfrak{o}_F$ l'anneau des entiers dans $F$. On prend un $\mathfrak{o}_F$-modèle $(L,\angles{\cdot|\cdot})$ de $(W,\angles{\cdot|\cdot})$. Pour toute place $v$ on note $(W_v, \angles{\cdot|\cdot})$ le complété de $(W,\angles{\cdot|\cdot})$ en $v$. Alors en presque toute place finie $v$, le complété $L_v$ est auto-duale et $\psi_v$ est de conducteur $\mathfrak{o}_{F_v}$. Cela permet de décrire $\Mp(W,\A)$ comme suit.
$$ \Mp(W,\A) = \left( \Resprod_v \Mp(W_v) \right)/\mathbf{N} \xrightarrow{\rev} \Resprod_v \Sp(W_v) = \Sp(W,\A) $$
où les produits restreints sont pris par rapport aux sous-groupes hyperspéciaux $K_{L_v}$ et
$$ \mathbf{N} := \left\{(\noyau_v)_v \in \bigoplus_v \bmu_8 : \prod_v \noyau_v = 1 \right\}. $$
Donc on peut regarder $\Mp(W_v)$ comme la fibre de $\Mp(W,\A) \twoheadrightarrow \Sp(W,\A)$ au-dessus de $\Sp(W_v)$. Cette construction ne dépend pas du choix du réseau $L \subset W$.

Soit $f \in C^\infty_c(\Mp(W,\A))$. On regarde $f$ comme une fonction sur $\Resprod_v \Mp(W_v)$, ce qui permet de parler des fonctions factorisables
\begin{gather}\label{eqn:factorisation-f}
  f = \prod_v f_v, \quad f_v \in C^\infty_c(\Mp(W_v))
\end{gather}
bien que $\Mp(W,\A)$ lui-même n'est pas un produit restreint. Définissons le sous-espace anti-spécifique $C^\infty_{c,\asp}(\Mp(W,\A))$ comme dans le cas local, alors si $f = \prod_v f_v$ appartient à $C^\infty_{c,\asp}(\Mp(W,\A))$, on a $f_v \in C^\infty_{c,\asp}(\Mp(W_v))$ pour toute $v$. La réciproque est aussi vraie. De plus, pour presque toute $v$ finie, on a $f_v = f_{K_{L_v}}$ sous la même hypothèse.

\subsection{Données endoscopiques elliptiques}\label{sec:donnees-endo}
Dans ce qui suit, $F$ désigne un corps local ou global de caractéristique nulle. On fixe un $F$-espace symplectique $(W, \angles{\cdot|\cdot})$ avec $\dim_F W = 2n$. Notons $G := \Sp(W)$ et $\tilde{G} := \Mp(W)$ dans le cas local (resp. $\tilde{G} := \Mp(W,\A)$ dans le cas global).

Rappelons qu'étant donné $k \in \Z_{\geq 0}$, le groupe $\SO(2k+1)$ signifie toujours sa forme déployée. Plus précisément, c'est le groupe $\SO(V,q)$ associé au $F$-espace quadratique $(V,q)$ où $V$ est un $F$-espace vectoriel avec une base
$$ e_k, \ldots, e_1, e_0, e_{-1}, \ldots e_{-k}, $$
et $q: V \times V \to F$ est la forme bilinéaire symétrique telle que
$$ q(e_i|e_{-j}) = \delta_{i,j}, \quad -k \leq i,j \leq k $$
où $\delta_{i,j}$ signifie le $\delta$ de Kronecker.

\begin{definition}
  Une donnée endoscopique elliptique pour le revêtement métaplectique $\tilde{G}$ est une paire $(n',n'') \in (\Z_{\geq 0})^2$ satisfaisant à $n'+n''=n$. Le groupe endoscopique correspondant à $(n',n'')$ est
  $$ H = H_{n',n''} := \SO(2n'+1) \times \SO(2n''+1). $$
\end{definition}

Cette description est analogue au cas de l'endoscopie de $\SO(2n+1)$. Toutefois il faut sous-ligner deux différences importantes.
\begin{enumerate}
  \item Il n'y a pas symétrie entre $(n',n'')$ et $(n'',n')$.
  \item Il n'y a pas d'automorphismes extérieurs (voir \cite[\S 7.2]{Ko84}) d'une donnée endoscopique elliptique; la signification sera claire dans notre étude de stabilisation.
\end{enumerate}

Les données endoscopiques en général sont définies dans \cite[\S 3]{Li12}, mais nous n'en avons pas besoin.

Fixons une donnée endoscopique elliptique $(n',n'')$. La correspondance entre les classes de conjugaison stables semi-simples est donnée en termes de valeurs propres. Précisons.

\begin{definition}
  Soient $\gamma = (\gamma', \gamma'') \in H(F)$ et $\delta \in G(F)$ des éléments semi-simples. On dit que $\gamma$ et $\delta$ se correspondent, noté par $\gamma \leftrightarrow \delta$, si $\gamma'$ a pour valeurs propres (en tant qu'un élément de $\GL(2n'+1, \bar{F})$)
  $$ a'_{n'}, \ldots, a'_1, 1, (a'_1)^{-1} \ldots (a'_{n'})^{-1} $$
  et celles de $\gamma''$ sont
  $$ a''_{n''}, \ldots, a''_1, 1, (a''_1)^{-1} \ldots (a''_{n''})^{-1}, $$
  tandis que les valeurs propres de $\delta$ (en tant qu'un élément de $\GL(W,\bar{F})$) sont
  $$ a'_{n'}, \ldots, a'_1, (a'_1)^{-1} \ldots (a'_{n'})^{-1}, -a''_{n''}, \ldots, -a''_1, -(a''_1)^{-1} \ldots -(a''_{n''})^{-1}. $$

  C'est clair que cette correspondance ne dépend que de la classe de conjugaison stable. Notons que des éléments de $G(F)$ sont stablement conjugués au sens de \cite{Lab04} si et seulement s'ils sont conjugués par un élément de $G(\bar{F})$ parce que $G$ est simplement connexe, mais pour $H$ la définition dans \textit{loc.\ cit.} est plus forte car $H^\gamma$ n'est pas connexe en général. De plus, cela induit une application bien définie entres les classes semi-simples stables
  $$ \mu: H(F)_{\text{ss}}/\text{conj. stable} \longrightarrow G(F)_{\text{ss}}/\text{conj. stable} $$
  à fibres finies. On peut le vérifier soit à l'aide de la connexité simple de $G$ et la description dans \cite[\S 5.1]{Li11}, soit à l'aide du paramétrage des classes semi-simples à rappeler ci-dessous.

  On dit qu'un élément semi-simple $\gamma$ de $H$ est $G$-régulier s'il correspond à un élément semi-simple régulier de $G$, ce qui entraîne que $\gamma$ est fortement régulier. La notion de $G$-régularité est de nature géométrique. Les éléments semi-simples $G$-réguliers dans $H$ forment un ouvert dense pour la topologie de Zariski, noté par $H_{G-\text{reg}}$.
\end{definition}

\subsection{Paramétrage des classes semi-simples}\label{sec:parametrage}
Rappelons maintenant le paramétrage des classes semi-simples proposé dans \cite[\S 3]{Li11}. Un exposé détaillé se trouve dans \cite[I.3]{LiPhD}. Dans ce qui suit, $F$ désigne un corps quelconque de caractéristique autre que $2$.

\subsubsection{Le cas des groupes symplectiques}
Soit $(W, \angles{\cdot|\cdot})$ un $F$-espace symplectique de dimension $2n$. Les classes de conjugaison semi-simples dans $G := \Sp(W)$ sont en bijection avec les classes d'isomorphismes des données
$$ K/K^\sharp, x, (W_K, h_K), (W_+, \angles{\cdot|\cdot}_+), (W_-, \angles{\cdot|\cdot}_-) $$
où
\begin{itemize}
  \item $K$ est une $F$-algèbre étale de dimension finie, munie d'une involution $\tau: K \to K$;
  \item $K^\sharp$ est la sous-algèbre étale de $K$ fixée par $\tau$, remarquons que la paire $(K, K^\sharp)$ détermine $\tau$;
  \item $x \in K^\times$ satisfait à $\tau(x)=x^{-1}$ et $K = F[x]$;
  \item $W_K$ est un $K$-module fidèle de type fini, et $h_K: W_K \times W_K \to K$ est une forme anti-hermitienne par rapport à l'involution $\tau$;
  \item $(W_\pm, \angles{\cdot|\cdot}_\pm)$ sont des $F$-espaces symplectiques.
\end{itemize}
La seule condition est
\begin{gather*}
  \dim_F W_K + \dim_F W_+ + \dim_F W_- = \dim_F W.
\end{gather*}

Le commutant $G^x$ d'un élément $x \in G(F)$ ainsi paramétré est isomorphe à
$$ U(W_K, h_K) \times \Sp(W_+) \times \Sp(W_-), $$
en particulier on a $G^x = G_x$. Ici, on a par définition
$$ U(W_K, h_K) = \prod_{i \in I^*} U_{K_i/K^\sharp_i}(W_i, h_i) \times \prod_{i \notin I^*} \GL_{K^\sharp_i}(n_i) $$
si l'on décompose la $F$-algèbre étale
\begin{gather}\label{eqn:decomp-K}
  K = \prod_{i \in I} K_i
\end{gather}
de sorte que pour chaque indice $i \in I$, la sous-algèbre $K^\sharp_i := K^\sharp \cap K_i$ est un corps, et on pose
$$ I^* := \{i \in I: K_i \text{ est un corps}\}.$$

On décompose $(W_K, h_K) = \prod_{i \in I} (W_i, h_i)$ en conséquence et définit les groupes unitaires $U_{K_i/K^\sharp_i}(W_i, h_i)$; pour $i \notin I^*$, ce que l'on obtient est un groupe linéaire général $\GL_{K^\sharp_i}(n_i)$. Enfin, on les regarde comme des $F$-groupes réductifs connexes en prenant la restriction des scalaires $K^\sharp_i/F$.

On dit donc que les données $(K/K^\sharp, x, (W_K, h_K))$ est la \guillemotleft partie unitaire\guillemotright\,de ce paramètre. Les valeurs propres (avec multiplicités) de $x$ sont
$$ \left\{\sigma(x)^{(\dim_{K_i} W_i \text{ fois})} : \sigma \in \Hom_{F-\text{alg}}(K, \bar{F}) \right\} \sqcup \{+1\}^{(\dim_F W_+ \text{ fois})} \sqcup \{-1\}^{(\dim_F W_- \text{ fois})}, $$
où $i = i(\sigma) \in I$ est l'indice tel que $\sigma$ se factorise par le facteur direct $K_i$ de $K$, et $\dim_{K_i} W_i := \rank_{K_i} W_i$ pour $i \notin I^*$.

Pour paramétrer les classes de conjugaison stables semi-simples dans $G$, il suffit d'oublier la forme anti-hermitienne $h_K$ dans les données.

\subsubsection{Le cas des groupes orthogonaux}
Soit $H := \SO(V,q)$, où $(V,q)$ est un $F$-espace quadratique. Pour notre étude ($\dim V$ est impaire), il suffit de considérer les classes de conjugaison semi-simples par $\Or(V,q)$. Celles-ci sont en bijection avec les classes d'isomorphismes des données
$$ K/K^\sharp, x, (V_K, h_K), (V_+, q_+), (V_-, q_-) $$
où
\begin{itemize}
  \item les données $K$, $K^\sharp$, $x$ sont comme dans le cas de $\Sp(2n)$;
  \item $V_K$ est un $K$-module fidèle de type fini;
  \item $h_K: V_K \times V_K \to K$ est une forme hermitienne par rapport à l'involution $\tau$ associée à $(K, K^\sharp)$;
  \item $(V_\pm, q_\pm)$ sont des $F$-espaces quadratiques.
\end{itemize}
Les conditions sont
\begin{gather*}
  \dim_F V_- \equiv 0 \mod 2, \\
  (V,q) \simeq (\Tr_{K/F})_* (V_K, h_K) \oplus (V_+, q_+) \oplus (V_-, q_-) \quad (\text{somme directe orthogonale})
\end{gather*}
où $(\Tr_{K/F})_* (V_K, h_K)$ est la forme bilinéaire $\Tr_{K/F} \circ h_K: V_K \times V_K \to F$, qui est une $F$-forme quadratique. La première condition est imposée pour que la classe de conjugaison ainsi obtenue soit incluse dans $\SO(V,q)$. D'autre part, cela entraîne $\dim_F V \equiv \dim_F V_+ \mod 2$.

Le commutant $H^x$ d'un élément $x \in H(F)_\text{ss}$ ainsi paramétré est l'intersection de $H$ avec le sous-groupe suivant de $\Or(V,q)$:
$$ U(W_K, h_K) \times \Or(V_+, q_+) \times \Or(V_-, q_-), $$
avec les mêmes conventions que dans le cas des groupes symplectiques.

Les valeurs propres (avec multiplicités) de $x$ sont
$$ \left\{ \sigma(x)^{(\dim_{K_i} V_i \text{ fois})} : \sigma \in \Hom_{F-\text{alg}}(K, \bar{F}) \right\} \sqcup \{+1\}^{(\dim_F V_+ \text{ fois})} \sqcup \{-1\}^{(\dim_F V_- \text{ fois})}, $$
où $\dim_{K_i} V_i$ a le même sens qu'au cas symplectique.

Deux éléments semi-simples dans $H(F)$ sont conjugués par les $\bar{F}$-points de $\Or(V,q)$ si et seulement si leurs paramètres sont isomorphes après oubli des formes $h_K$, $q_+$ et $q_-$ dans les données. Cette relation de conjugaison est plus grossière que la conjugaison stable dans $H$, mais cela suffit pour nous.

\subsubsection{Correspondance des classes}\label{sec:correspondance}
Revenons à l'endoscopie pour le groupe métaplectique. Soit $F$ un corps local ou global de caractéristique nulle. Fixons un $F$-espace symplectique $(W,\angles{\cdot|\cdot})$ de dimension $2n$ et posons
\begin{align*}
  G & := \Sp(W), \\
  H = H_{n',n''} & := \SO(2n'+1) \times \SO(2n''+1) 
\end{align*}
pour la donnée endoscopique elliptique correspondant à une paire $(n',n'')$, $n'+n''=n$.

On utilise la forme quadratique choisie dans \S\ref{sec:donnees-endo} pour définir les groupes orthogonaux spéciaux $\SO(2n'+1)$ et $\SO(2n''+1)$.

Vu le paramétrage des classes semi-simples et la description de la correspondance des classes en termes de valeurs propres, on voit que deux éléments $\delta \in G(F)_\text{ss}$ et $\gamma = (\gamma', \gamma'') \in H(F)_\text{ss}$ se correspondent si et seulement si
\begin{itemize}
  \item la classe de $\gamma'$ a pour paramètres
    $$ K'/{K'}^\sharp, a', (V_{K'}, h_{K'}), (V'_+, q'_+), (V'_-, q'_-); $$
  \item la classe de $\gamma''$ a pour paramètres
    $$ K''/{K''}^\sharp, a'', (V_{K''}, h_{K''}), (V''_+, q''_+), (V''_-, q''_-); $$
  \item la classe de $\delta$ a pour paramètres
    $$ K/K^\sharp, (a',-a''), (W_K, h_K), (W_\pm, \angles{\cdot|\cdot}_\pm) $$
    où $K$ est la sous-algèbre de $K' \times K''$ engendrée par $(a', -a'')$, $K^\sharp := K \cap ({K'}^\sharp \times {K''}^\sharp)$, $W_K \simeq V_{K'} \oplus V_{K''}$ en tant que des $K$-modules, et
    \begin{align*}
      \dim_F W_+ + 1 &= \dim_F V'_+ + \dim_F V''_-, \\
      \dim_F W_- + 1 &= \dim_F V'_- + \dim_F V''_+.
    \end{align*}
\end{itemize}

D'après la description des commutants, si $\gamma$ est $G$-régulier, alors $\gamma$ est fortement régulier et on a un isomorphisme $H^\gamma = H_\gamma \simeq G_\delta$ entre $F$-tores.

\begin{remark}\label{rem:erratum}
  On avait mal expliqué la correspondance des classes semi-simples singulières dans \cite[\S 5.1 et \S 7.2]{Li11} ou \cite[\S 6.1]{Li12}; par exemple, on ne peut pas supposer que $K = K' \times K''$ en général!

  Regardons la partie unitaire. On peut toujours décomposer $(K, K^\sharp)$ en un produit des paires $(L, L^\sharp)$ telles que $L^\sharp$ est un corps. Pour $L$ fixé, on pose $u := a|_L$ et on a $L=F[u]$. En indexant les termes par les classes d'isomorphisme des telles paires $(L,u)$, on écrit la partie unitaire du paramètre de $\delta$ comme
  $$ \prod_{L,u} (L/L^\sharp, u, (W_u, h_u)) $$
  avec $W_u$ non nuls. Les parties unitaires des paramètres de $\gamma'$, $\gamma''$ se décomposent en conséquence
  $$ \prod_{L,u} (L/L^\sharp, u, (V'_u, h'_u)), \quad \prod_{L,u} (L/L^\sharp, -u, (V''_u, h''_u)), $$
  où $(L,u)$ parcourt le même ensemble d'indices, mais on permet que pour tout $(L,u)$, au plus l'un de $V'_u$ et $V''_u$ est nul. En effet, cela est clair d'après le raccordement des valeurs propres. Pour la même raison, on doit exiger que
  $$ W_u \simeq V'_u \oplus V''_u \quad \text{en tant que des $L$-modules } $$
  pour tout $(L,u)$. C'est compatible avec la description précédente.

  Tel est le point de départ de \cite[\S 7]{Li11} et \cite[\S 6]{Li12}, donc tous les arguments de ceux-ci demeurent valables.
\end{remark}

\subsection{Transfert}
Fixons les objets suivants dans cette sous-section.
\begin{itemize}
  \item $F$: un corps local de caractéristique nulle.
  \item $\psi: F \to \C^\times$: un caractère unitaire non-trivial.
  \item $(W, \angles{\cdot|\cdot})$: un $F$-espace symplectique de dimension $2n$.
  \item $G := \Sp(W)$.
  \item $\tilde{G} := \Mp(W) \xrightarrow{\rev} G(F)$ le revêtement métaplectique à huit feuillets.
  \item $(n',n'')$: une donnée endoscopique elliptique de $\tilde{G}$, et $H := \SO(2n'+1) \times \SO(2n''+1)$ le groupe endoscopique correspondant.
\end{itemize}

\subsubsection{Facteur de transfert géométrique}
Soient $\gamma = (\gamma', \gamma'') \in H_{G-\text{reg}}(F)$, $\delta \in G_\text{reg}(F)$ et $\tilde{\delta} \in \rev^{-1}(\delta)$. Si $\gamma \leftrightarrow \delta$, le facteur de transfert géométrique s'écrit comme
$$ \Delta(\gamma, \tilde{\delta}) = \Delta'(\tilde{\delta}') \Delta''(\tilde{\delta}'') \Delta_0(\delta', \delta''). $$

Expliquons. L'ensemble des valeurs propres de $\delta$ se décompose en deux parties: l'une provient de $\gamma'$ et l'autre de $\gamma''$. Puisque $\delta \in \Sp(W)$, cela induit la décomposition orthogonale $W = W' \oplus W''$ par rapport à $\angles{\cdot|\cdot}$. Il y a un diagramme commutatif naturel
$$\xymatrix{
  \Mp(W') \times \Mp(W'') \ar[r] \ar@{->>}[d]_{(\rev', \rev'')} & \Mp(W) \ar@{->>}[d]^{\rev} \\
  \Sp(W') \times \Sp(W'') \ar@{^{(}->}[r] & \Sp(W).
}$$

Il existe $(\delta', \delta'') \in \Sp(W') \times \Sp(W'')$ qui s'envoie sur $\delta$. On peut prendre donc $(\tilde{\delta}', \tilde{\delta}'') \in \Mp(W') \times \Mp(W'')$ qui s'envoie sur $\tilde{\delta}$; cette paire est unique 
au plongement anti-diagonal de $\bmu_8$ près. Fixons une telle paire. Alors
\begin{align*}
  \Delta'(\tilde{\delta}') & := \frac{{\Theta'}^+_\psi - {\Theta'}^-_\psi}{|{\Theta'}^+_\psi - {\Theta'}^-_\psi|}(\tilde{\delta}'), \\
  \Delta''(\tilde{\delta}'') & := \frac{{\Theta''}^+_\psi + {\Theta''}^-_\psi}{|{\Theta''}^+_\psi + {\Theta''}^-_\psi|}(\tilde{\delta}''),
\end{align*}
où ${\Theta'}^\pm_\psi$ (resp. ${\Theta''}^\pm_\psi$) sont définis par rapport à $W'$ (resp. $W''$). Pour justifier ces définitions, on applique les faits suivants à $W'$ et $W''$.
\begin{enumerate}
  \item La distribution $\Theta^+_\psi \mp \Theta^-_\psi$ est lisse sur $\tilde{G}_\text{reg} := \rev^{-1}(G_\text{reg}(F))$. En fait, c'est crucial de noter que son lieu singulier est exactement l'image réciproque de $\{x \in \Sp(W): \det(x \pm 1) = 0\}$.
  \item Hors son lieu singulier, $\Theta^+_\psi \mp \Theta^-_\psi$ n'est jamais nulle.
\end{enumerate}
Ces faits sont des conséquences de la formule de Maktouf \cite[\S 4]{Li11} pour $\Theta^+_\psi \mp \Theta^-_\psi$. On voit que le produit $\Delta'\Delta''$ ne dépend pas du choix de $(\tilde{\delta}', \tilde{\delta}'')$.

Pour définir $\Delta_0$, on utilise le paramétrage de $\gamma'$, $\gamma''$ et $\delta$ dans \S\ref{sec:parametrage}. Alors on a
\begin{gather*}
  \Delta_0(\delta', \delta'') = \sgn_{K''/{K''}^\sharp}(P_{a'}(a'')(-a'')^{-n'} \det(\delta'+1))
\end{gather*}
où $P_{a'}$ est le polynôme caractéristique sur $F$ de $a' \in K'$,
$$ \sgn_{K''/{K''}^\sharp} := \prod_{i \in I^*} \sgn_{K''_i/{K''}^\sharp_i}: ({K''}^\sharp)^\times \to \{\pm 1\} $$
selon la décomposition de l'algèbre étale $K''$ comme dans \eqref{eqn:decomp-K}, et $\sgn_{K''_i/{K''}^\sharp_i}$ est le caractère quadratique de $({K''}^\sharp_i)^\times$ fourni par la théorie du corps de classes.

C'est clair que $\Delta_0(\delta', \delta'')$ est déterminé par la classe de conjugaison stable de $\gamma$.

D'après la formule pour $\Theta_\psi$ dans \cite[Théorème 4.2]{Li11}, on voit que $\Delta(\gamma, \tilde{\delta}) \in \bmu_8$. Notre facteur de transfert géométrique est privé du terme $\Delta_{IV}$ dans \cite{LS1}.

\subsubsection{Transfert géométrique}
Choisissons les mesures de Haar sur $G(F)$ et $H(F)$. On en déduit une mesure de Haar sur le revêtement $\tilde{G}$ en suivant la recette générale dans \cite{Li14a}, i.e. on exige que
\begin{gather}\label{eqn:convention-mes}
  \mes(\rev^{-1}(E)) = \mes(E)
\end{gather}
pour tout sous-ensemble mesurable $E \subset G(F)$.

Définissons les intégrales orbitales anti-spécifiques sur $\tilde{G}$:
$$ J_{\tilde{G}}(\tilde{\delta}, f), \quad \tilde{\delta} \in \tilde{G}_\text{reg}, \; f \in C^\infty_{c,\asp}(\tilde{G}), $$
et les intégrales orbitales stables sur $H$:
$$ S_H(\gamma, f^H), \quad \gamma \in H_{\text{freg}}(F), \; f^H \in C^\infty_c(H(F)). $$

Comme dans les travaux d'Arthur, nos intégrales orbitales sont normalisées par le discriminant de Weyl. Ce n'est pas la peine de répéter les définitions ici, car nous allons en proposer une généralisation aux classes semi-simples singulières dans \S\ref{sec:int-orb}. En ce moment, il suffit de savoir qu'elles dépendent du choix d'une mesure de Haar sur $G_\delta(F)$ et $H_\gamma(F)$, respectivement.

Les résultats principaux de \cite[\S 5.5]{Li11} s'énoncent comme suit.
\begin{theorem}\label{prop:transfert}
  Pour tout $f \in C^\infty_{c,\asp}(\tilde{G})$, il existe $f^H \in C^\infty_c(H(F))$ tel que pour tout $\gamma \in H_{G-\mathrm{reg}}(F)$, on a
  $$ \sum_{\substack{\delta \in \Gamma_{\mathrm{reg}}(G) \\ \delta \leftrightarrow \gamma}} \Delta(\gamma, \tilde{\delta}) J_{\tilde{G}}(\tilde{\delta}, f) = S_H(\gamma, f^H) $$
  où $\Gamma_{\mathrm{reg}}(G)$ désigne l'ensemble des classes de conjugaison semi-simples régulières dans $G(F)$, et pour chaque choix du représentant $\delta \in G(F)$, $\tilde{\delta}$ est un élément quelconque dans $\rev^{-1}(\delta)$. On exige que les mesures de Haar sur $H_\gamma(F)$ et $G_\delta(F)$ se concordent par l'isomorphisme $H_\gamma \simeq G_\delta$ entre $F$-tores.
\end{theorem}

Le produit $\Delta(\gamma, \tilde{\delta}) J_{\tilde{G}}(\tilde{\delta}, f)$ dans la somme ne dépend pas du choix de $\tilde{\delta}$. On dit que $f^H$ est un transfert de $f$. Cette notion dépend du choix de mesures de Haar sur $G(F)$ et $H(F)$.

\begin{remark}
  Le transfert $f \mapsto f^H$ deviendra une application bien définie si l'on le considère $f$ et $f^H$ comme éléments dans les espaces
  \begin{itemize}
    \item $\mathcal{I}_{\asp}(G)$: l'espace des intégrales orbitales normalisées anti-spécifiques $J_{\tilde{G}}(\cdot, f)$, en tant que des fonctions $\Gamma_\text{reg}(\tilde{G}) \to \C$;
    \item $S\mathcal{I}(H)$: l'espace des intégrales orbitales normalisées stables $S_H(\cdot, f^H)$, en tant que des fonctions $\Sigma_\text{reg}(H) \to \C$,
  \end{itemize}
  respectivement. On en donnera un aperçu dans la Définition \ref{def:I}.
\end{remark}

Plaçons-nous maintenant dans la situation non-ramifiée (voir \S\ref{sec:groupes}). Alors les $F$-groupes réductifs connexes $G$ et $H$ sont tous non-ramifiés. On a défini $f_K \in C^\infty_{c,\asp}(\tilde{G})$, l'unité de l'algèbre de Hecke sphérique anti-spécifique de $\tilde{G}$. Rappelons aussi que la mesure non-ramifiée d'un groupe non-ramifié est la mesure de Haar telle qu'un (donc tout) sous-groupe hyperspécial a pour masse totale égale à $1$. Il faut prendre garde que pour les revêtements de $G(F)$, on utilise la convention \eqref{eqn:convention-mes}.

\begin{theorem}
  Dans la situation non-ramifiée, prenons les mesures non-ramifiées sur $G(F)$ et $H(F)$. Si $f = f_K$, alors $f^H := \mathbbm{1}_{K^H}$ est un transfert de $f$, où $K^H \subset H(F)$ est un sous-groupe hyperspécial quelconque, et $\mathbbm{1}_{K^H}$ désigne sa fonction caractéristique.
\end{theorem}

C'est le \guillemotleft lemme fondamental\guillemotright\,pour l'unité de l'algèbre de Hecke sphérique anti-spécifique.

Supposons maintenant que $F$ est un corps de nombres et considérons les objets $\rev: \tilde{G} \to G(\A)$, $(n',n'')$, $H$ au cadre global. Alors on peut définir un transfert $f^H \in C^\infty_c(H(\A))$ d'une fonction anti-spécifique $f = \prod_v f_v \in C^\infty_{c,\asp}(\tilde{G})$ (cf. la convention pour \eqref{eqn:factorisation-f}) en rassemblant les transfert locaux, disons $f^H := \prod_v f^H_v$, ce qui est bien défini grâce au lemme fondamental pour l'unité.

\section{Cohomologie et mesures}\label{sec:coho-mes}
\subsection{Cohomologie abélianisée}\label{sec:coho-abel}
Reprenons le formalisme cohomologique dans \cite{Lab99, Lab01, Lab04}.

\subsubsection{Cohomologie galoisienne}
Soit $F$ un corps. On désigne les groupes de cohomologie galoisienne d'un $F$-groupe diagonalisable $C$ par
$$ H^i(F, C) := \varinjlim_{K/F} H^i(\Gamma_{K/F}, C(K)), \quad i \in \Z $$
où $K$ parcourt les extensions galoisiennes finies de $F$. Plus généralement, soit
$$ C^\bullet = [\cdots \to C^{-1} \to C^0 \to \cdots] $$
un complexe borné de $F$-groupes diagonalisables, on désigne l'hypercohomologie galoisienne de $C^\bullet$ par $H^i(F, C^\bullet)$, pour tout $i \in \Z$.

C'est aussi important d'étendre ce langage au cas non-commutatif. Soit $G$ un $F$-groupe algébrique, on peut toujours définir le groupe $H^0(F, G) = G(F)$ et l'ensemble pointé $H^1(F,G)$. Soit $f: I \to G$ un homomorphisme entre des $F$-groupes algébriques, alors on peut considérer le petit complexe $[I \to G]$ (en degrés $-1$ et $0$) et définir
\begin{align*}
  H^{-1}(F, I \to G) & := \Ker(f)(F), \\
  H^0(F, I \to G) & := (f(I) \backslash G)(F)
\end{align*}
en tant qu'un ensemble pointé. Ici $f(I) \backslash G$ est vu comme un faisceau pour le site étale sur $\Spec(F)$. On a donc l'inclusion $f(I(F)) \backslash G(F) \subset H^0(F, I \to G)$. La suite exacte associée est
\begin{gather}\label{eqn:suite-I-G}
  H^{-1}(F, I \to G) \to H^0(F, I)  \to H^0(F, G) \to H^0(F, I \to G) \to H^1(F, I) \to H^1(F,G)
\end{gather}
au sens d'ensembles pointés. Comme ce que la notation suggère, $[I \to G]$ est regardé comme le cône d'application de $f$.

\subsubsection{Cohomologie adélique}
Maintenant $F$ est un corps global et $\A = \A_F$ est l'anneau d'adèles. Soit $C^\bullet$ un complexe borné de $F$-groupes diagonalisables. La cohomologie adélique de $C^\bullet$ est définie dans \cite[\S 2.1]{Lab01} au moyen de produit restreint. Pour presque toute place finie $v$ de $F$ et tout $i \in \Z$, on peut définir
$$ H^i(\Gamma_{F^\text{nr}_v/F_v}, C^\bullet(\mathfrak{o}^\text{nr}_v)) $$
où $F^\text{nr}_v$ est une extension non ramifiée maximale de $F_v$, dont $\mathfrak{o}^\text{nr}_v$ est l'anneau d'entiers. Il y a une application canonique de ce groupe sur $H^i(F_v, C^\bullet)$, dont l'image est notée par $H^i(\mathfrak{o}_v, C^\bullet)$. On a
$$ H^i(\A, C^\bullet) := \Resprod_v H^i(F_v, C^\bullet), \quad i \in \Z $$
où le produit restreint est pris par rapport aux sous-groupes $H^i(\mathfrak{o}_v, C^\bullet)$. Ces groupes de cohomologie sont invariants par quasi-isomorphismes entre complexes.

Cette approche s'adapte au cas non-commutatif et permet de définir $H^0(\A, G)$, $H^1(\A, G)$, $H^{-1}(\A, I \to G)$ et $H^0(\A, I \to G)$; pour étudier le dernier ensemble pointé, le résultat suivant sera utile.

\begin{lemma}\label{prop:I-G-adelique}
  Soit $f: I \to G$ un homomorphisme entre des $F$-groupes réductifs connexes. Si $v$ est une place finie de $F$ telle que $I$ et $G$ admettent des $\mathfrak{o}_v$-modèles lisses connexes sur $\mathfrak{o}_v$ pour lesquels $f$ est défini sur $\mathfrak{o}_v$, alors
  $$ H^0(\mathfrak{o}_v, I \to G) = f(I(\mathfrak{o}_v)) \backslash G(\mathfrak{o}_v). $$
  La condition sur $v$ est satisfaite pour presque toute place de $F$.
\end{lemma}
\begin{proof}
  Puisque $H^0(\mathfrak{o}_v, I \to G)$ est défini comme l'image de $(f(I) \backslash G)(\mathfrak{o}_v)$ dans $(f(I) \backslash G)(F_v)$, ceci est essentiellement une conséquence du théorème de Lang.
\end{proof}

Supposons que $T^\bullet$ est un complexe borné de $F$-tores. Alors une autre définition équivalente pour la cohomologie galoisienne adélique de $T^\bullet$ est donnée dans \cite[p.26]{Lab99}
$$ H^i(\A, T^\bullet) = \varinjlim_{K/F} H^i(\Gamma_{K/F}, T^\bullet(\A_K)), \quad i \in \Z $$
où $K/F$ parcourt les extensions galoisiennes finies de $F$. Cela fournit un moyen pour étudier $H^i(\A, C^\bullet)$, disons on choisit un complexe borné de $F$-tores $T^\bullet$ quasi-isomorphe à $C^\bullet$, et on calcule $H^i(\A, T^\bullet)$ par la recette ci-dessus.

Définissons maintenant les groupes $H^i(\A/F, C^\bullet)$ pour un complexe borné $C^\bullet$ de $F$-groupes diagonalisables. D'après ce que l'on vient d'observer, on peut supposer que $C^\bullet$ est un complexe de $F$-tores. C'est plus rapide de le définir en termes du cône d'application $[C^\bullet(K) \to C^\bullet(\A_K)]$:
$$ H^i(\A/F, C^\bullet) := \varinjlim_{K/F} H^i(\Gamma_{K/F}, C^\bullet(K) \to C^\bullet(\A_K)), \quad i \in \Z $$
où $K/F$ parcourt les extensions galoisiennes finies de $F$.

La suite exacte longue associée au triangle distingué
$$ C^\bullet(K) \to C^\bullet(\A_K) \to [C^\bullet(K) \to C^\bullet(\A_K)] \xrightarrow{+1} $$
permettent de définir l'application de localisation $H^i(F, C^\bullet) \to H^i(\A, C^\bullet)$. Son noyau est noté par $\Ker^i(F, C^\bullet)$. Ces groupes font partie de la suite exacte suivante.
\begin{multline}\label{eqn:A/F-suite}
  1 \to \Ker^i(F, C^\bullet) \to H^i(F, C^\bullet) \to H^i(\A_F, C^\bullet) \to H^i(\A/F, C^\bullet) \to \Ker^{i+1}(F, C^\bullet) \to 1.
\end{multline}

Pour un $F$-groupe $G$,  on peut aussi définir l'ensemble pointé $\Ker^1(F, G) := \Ker(H^1(F, G) \to H^1(\A, G))$.

\subsubsection{Abélianisation}
Soient $F$ un corps parfait et $G$ un $F$-groupe réductif connexe. On note $G_\text{SC} \to G$ le revêtement simplement connexe du groupe dérivé $G_\text{der}$. Suivant Borovoi, on définit un petit complexe de groupes diagonalisables
$$ G_\text{ab} := [Z_{G,\text{sc}} \to Z_G], \quad \text{en degrés } -1,0 $$
où $Z_{G,\text{sc}}$ est l'image réciproque de $Z_G$ dans $G_\text{SC}$. Si l'on préfère les complexes de $F$-tores, on peut aussi utiliser l'isomorphisme dans la catégorie dérivée bornée
$$ G_\text{ab} \simeq [T_\text{sc} \to T] $$
induit par $Z_G \hookrightarrow T$, où $T$ est un $F$-tore maximal quelconque de $G$, et $T_\text{sc}$ est son image réciproque dans $G_\text{SC}$. La cohomologie abélianisée de $G$ est, par définition
$$ H^i_\text{ab}(F, G) := H^i(F, G_\text{ab}), \quad i \in \Z $$
le terme à droite étant l'hypercohomologie galoisienne d'un petit complexe de $F$-groupes diagonalisables, ou de $F$-tores. Une conséquence immédiate est que $H^i_\text{ab}(F, G) = H^i_\text{ab}(F, G')$ si $G'$ est une forme intérieure de $G$.

On dispose de l'application d'abélianisation
$$ \text{ab}^i_G: H^i(F,G) \to H^i_\text{ab}(F,G), \quad i \leq 1, $$
qui est fonctorielle en $G$. Pour $F$ local ou global, $\text{ab}^1_G$ est toujours surjective; elle est un isomorphisme si $F$ est non-archimédien. L'application $\text{ab}^0_G$ est surjective pour $F$ non-archimédien. Voir \cite[Proposition 1.6.7]{Lab99}.

C'est donc loisible de définir les groupes $H^i_\text{ab}(\A, G)$, $H^i_\text{ab}(\A/F, G)$, $\Ker^i_{\text{ab}}(F, G)$ pour $F$ global, en utilisant le complexe $G_\text{ab}$. Enfin, pour un homomorphisme $I \to G$ entre $F$-groupes réductifs connexes, on définit $H^i_\text{ab}(\star, I \to G) := H^i(\star, I_\text{ab} \to G_\text{ab})$ pour $\star \in \{F, \A, \A/F\}$, où $[I_\text{ab} \to G_\text{ab}]$ est le cône d'application.

\begin{example}\label{ex:coho-U}
  Soient $F$ un corps local ou global, de caractéristique autre que $2$, et $E/F$ une extension quadratique de corps. Notons
  $$ \text{Res}^1_{E/F} \Gm := \Ker\left[ N_{E/F}: \text{Res}_{E/F} \mathbb{G}_m \to \mathbb{G}_m \right],$$
  vu comme un $F$-tore, où $N_{E/F}$ est l'application norme et $\text{Res}_{E/F}$ est la restriction des scalaires.

  Soit $(V,h)$ un espace hermitien par rapport à $E/F$, i.e. $V$ est un $E$-espace vectoriel de dimension finie et $h: V \times V \to E$ est une forme hermitienne. Le groupe unitaire associé est noté $U := \U_{E/F}(V,h)$. On a les isomorphismes canoniques suivants
  \begin{itemize}
    \item si $F$ est local, alors $H^1_{\mathrm{ab}}(F, U) \simeq H^1(F, \text{Res}^1_{E/F} \Gm) = \{\pm 1\}$;
    \item si $F$ est global, alors $H^1_{\mathrm{ab}}(\A/F, U) \simeq H^1(\A/F, \text{Res}^1_{E/F} \Gm) = \{\pm 1\}$.
  \end{itemize}

  En effet, puisque $U_\text{der} = \SU_{E/F}(V,h)$ est simplement connexe, le complexe $U_\text{ab}$ est quasi-isomorphe au co-centre $U/U_\text{der}$ placé en degré $0$. Or le déterminant induit $U/U_\text{der} \rightiso \text{Res}^1_{E/F} \Gm$. Pour $F$ local, on a $H^1(F, \text{Res}^1_{E/F} \Gm) = F^\times/N_{E/F}(E^\times)$ qui est isomorphe à $\{\pm 1\}$. Le cas global est pareil.
\end{example}

\begin{lemma}\label{prop:coho-U-mult}
  Soit $E/F$ comme dans l'Exemple \ref{ex:coho-U}. On se donne
  \begin{itemize}
    \item une $F$-algèbre étale $L^\sharp$ et la $E$-algèbre étale $L := L^\sharp \otimes_F E$ munie de l'involution provenant de $E/F$;
    \item un $L$-module fidèle $V$ avec une forme $L/L^\sharp$-hermitienne $h_L: V \times V \to L$.
  \end{itemize}
  On note $h := \Tr_{L/E} \circ h_L$. Alors l'inclusion naturelle
  $$ \mathrm{Res}_{L^\sharp/F} \U_{L/L^\sharp}(V, h_L) \hookrightarrow \U_{E/F}(V, h) $$
  induit l'application de la forme
  $$ \{\pm 1\}^r \xrightarrow{\mathrm{produit}} \{\pm 1\} $$
  pour $H^1_{\mathrm{ab}}(F, \cdot)$ (le cas local) ou $H^1_{\mathrm{ab}}(\A/F, \cdot)$ (le cas global).
\end{lemma}
\begin{proof}
  Traitons le cas $F$ local. On observe d'abord la commutativité du diagramme suivant
  $$\xymatrix{
    \U_{L/L^\sharp}(V, h_L) \ar[r]^{\det_L} \ar@{^{(}->}[d] & \text{Res}^1_{L/L^\sharp} \Gm \ar[d]^{N_{L/E}} \\
    \U_{E/F}(V, h) \ar[r]_{\det_E} & \text{Res}^1_{E/F} \Gm
  }$$
  dont les lignes donnent les co-centres en question.

  Ensuite, on observe que l'homomorphisme induit par $N_{L/E}$
  $$ H^1(F, \text{Res}^1_{L/L^\sharp} \Gm) \to H^1(F, \text{Res}^1_{E/F} \Gm) $$
  se traduit en l'homomorphisme $N_{L^\sharp/F}: L^{\sharp \times}/N_{L/L^\sharp}(L^\times) \to F^\times/N_{E/F}(E^\times)$. Pour conclure, on décompose $L^\sharp$ pour se ramener au cas où $L^\sharp$ est un corps; ledit homomorphisme $N_{L^\sharp/F}$ se décompose en conséquence en un produit. Reste à vérifier la commutativité du diagramme
  $$\xymatrix{
    L^{\sharp \times}/N_{L/L^\sharp}(L^\times) \ar[dd]_{N_{L^\sharp/F}} \ar[rd]^{\sgn_{L/L^\sharp}} & \\
    & \{ \pm 1\} \\
    F^\times/N_{E/F}(E^\times) \ar[ru]_{\sgn_{E/F}}
  }$$
  où $\sgn_{L/L^\sharp}$ et $\sgn_{E/F}$ sont les caractères quadratiques correspondant à ces extensions de degré $2$; on a $\sgn_{L/L^\sharp} = 1$ si $L = L^\sharp \times L^\sharp$. Un tel résultat est standard dans la théorie du corps de classes.
\end{proof}

\subsubsection{Un glossaire}
Comme dans \cite{Lab01, Lab04}, on aura besoin des objets suivants lors de la stabilisation. Soit $F$ un corps local ou global. Soient $G$ un $F$-groupe réductif connexe et $I$ un sous-groupe réductif de $G$ (en fait, dans \cite[\S 3.3]{Lab01} Labesse permet des sous-groupes plus généraux dits admissibles \cite[Définition 1.8.2]{Lab99}).

\begin{definition}\label{def:glossaire}
  On définit
  \begin{align*}
    \mathfrak{D}(I, G; F) & := \Ker[H^1(F, I) \to H^1(F, G)], \\
    \mathfrak{E}(I, G; F) & := \Ker[H^1_\text{ab}(F, I) \to H^1_\text{ab}(F, G)].
  \end{align*}

  Si $F$ est global, on définit
  \begin{align*}
    \mathfrak{D}(I, G; \A) & := \Ker[H^1(\A, I) \to H^1(\A, G)], \\
    \mathfrak{E}(I, G; \A) & := \Ker[H^1_\text{ab}(\A, I) \to H^1_\text{ab}(\A, G)], \\
    \mathfrak{E}(I, G; \A/F) & := \Coker[H^0_\text{ab}(\A, G) \to H^0_\text{ab}(\A/F, I \to G)], \\
    \mathfrak{R}(I, G; \A) & := \text{le dual de } H^0_\text{ab}(\A, I \to G), \\
    \mathfrak{R}(I, G; F) & := \text{le dual de } H^0_\text{ab}(\A/F, I \to G), \\
    \mathfrak{R}(I, G; F)_1 & := \text{le dual de } \mathfrak{E}(I, G; \A/F).
  \end{align*}
  Ici, le mot dual signifie le dual de Pontryagin.
\end{definition}

Observons que les applications d'abélianisation induisent $\mathfrak{D}(I, G; F) \to \mathfrak{E}(I, G; F)$ et $\mathfrak{D}(I, G; \A) \to \mathfrak{E}(I, G; \A)$ (pour $F$ global).  Dans le cas $F$ global, il y a un homomorphisme $\mathfrak{R}(I, G; F) \to \mathfrak{R}(I, G; \A)$ qui est dual à $H^0_\text{ab}(\A, I \to G) \to H^0_\text{ab}(\A/F, I \to G)$ (voir \eqref{eqn:A/F-suite}). D'autre part, $\mathfrak{R}(I, G; F)_1$ est inclus dans $\mathfrak{R}(I, G; F)$ car $\mathfrak{E}(I, G; \A/F)$ est un quotient de $H^0_\text{ab}(\A/F, I \to G)$.

Par ailleurs, $G(F)$ opère à droite sur $H^0(F, I \to G) = (I \backslash G)(F)$ par multiplication. Vu la suite exacte \eqref{eqn:suite-I-G}, le quotient $H^0(F, I \to G)/G(F)$ est isomorphe à $\mathfrak{D}(I, G; F)$ en tant que des ensembles pointés.

\subsection{Mesures de Tamagawa en cohomologie}\label{sec:mesure-coho}
Dans cette sous-section, $F$ est un corps de nombres.

\begin{enumerate}[A)]
  \item \emph{Mesures de Tamagawa usuelles}. Soit $G$ un $F$-groupe réductif connexe. C'est bien connu qu'il existe a une mesure de Haar canonique $\mu^G_\text{Tama}$ sur $G(\A)$, la mesure de Tamagawa (voir eg. \cite[Appendix 2]{Weil82}). Au cas $Z^\circ_G$ anisotrope, la définition sera rappelée dans la sous-section suivante. Comme $G(F)$ est un sous-groupe discret de $G(\A)$, on en déduit une mesure canonique sur $G(F) \backslash G(\A)$.

  La mesure de Tamagawa est ensuite généralisée par Labesse \cite[\S 1.2]{Lab01} aux groupes dits réductifs quasi-connexes, y compris les groupes diagonalisables.

  \item \emph{Mesures de Tamagawa en cohomologie}. Soit $C$ un groupe diagonalisable. Alors $H^0(\A, C) = C(\A)$ est déjà muni de la mesure de Tamagawa. Le groupe $H^1(\A, C) = \Resprod_v H^1(F_v, C)$ est muni de la mesure qui est le produit des mesures locales donnant masse
  $$ |(C/C^\circ)(F_v)|^{-1} |H^1(F_v, C)| $$
  à $H^1(F_v, C)$. Pour $i \geq 2$, on munit $H^i(\A, C)$ de la mesure discrète.

  Considérons un petit complexe $E := [C^{-1} \to C^0]$ de groupes diagonalisables. Un dévissage standard (voir \cite[\S 2.3]{Lab01}) permet de définir la mesure dite de Tamagawa sur $H^i(\A, E)$ pour tout $i \in \Z$, qui est discrète pour $i \geq 2$. On munit $H^i(F, E)$ et $\Ker^i(F, E)$ des mesures discrètes, alors la suite exacte \eqref{eqn:A/F-suite} définit la mesure de Tamagawa sur $H^i(\A/F, E)$ pour tout $i \in \Z$.

  \item \emph{Cohomologie abélianisée}. La recette précédente permet aussi de définir les mesures de Tamagawa sur $H^i_\text{ab}(\A, G)$ et $H^i_\text{ab}(\A/F, G)$, pour un $F$-groupe réductif connexe $G$ et tout $i \in \Z$. Notons que la mesure sur $H^i_\text{ab}(\A, G)$ est discrète si $i \geq 2$.

  \item \emph{L'ensemble $H^0(\A, I \to G)$}. Soit $I$ un sous-groupe réductif connexe de $G$. On a défini l'ensemble $\mathfrak{D}(I, G; \A)$ et son avatar abélianisé $\mathfrak{E}(I, G; \A)$ dans la Définition \ref{def:glossaire}. On peut regarder $\mathfrak{E}(I, G; \A)$ comme un sous-groupe ouvert de $H^1_\text{ab}(\A, I)$, d'où une mesure induite. On a une application canonique $\mathfrak{D}(I, G; \A) \to \mathfrak{E}(I, G; \A)$, dont les composantes en les places finies sont des isomorphismes. Les composantes archimédiennes sont des applications entre des ensembles finis pointés. Donc $\mathfrak{D}(I, G; \A)$ est un revêtement fini d'un sous-ensemble ouvert de $\mathfrak{E}(I, G; \A)$. On munit $\mathfrak{D}(I, G; \A)$ de l'unique mesure telle que $\mathfrak{D}(I, G; \A) \to \mathfrak{E}(I, G; \A)$ préserve localement les mesures.

  Il existe une application surjective $H^0(\A, I \backslash G) \to \mathfrak{D}(I, G; \A)$ provenant de \eqref{eqn:suite-I-G} en chaque place. Soient $\eta \in \mathfrak{D}(I, G; \A)$ et $x \in H^0(\A, I \backslash G)$ tels que $x \mapsto \eta$, alors la fibre de ladite application en $\eta$ est $I_\eta(\A) \backslash G(\A)$ où $I_\eta := x^{-1} I x$ est un schéma en groupes défini sur $\A$. Les composantes locales de $I_\eta$ sont des formes intérieures de $I$, ce qui permet de transférer la mesure de Tamagawa de $I(\A)$ à $I_\eta(\A)$.

  Maintenant on peut définir la mesure de Tamagawa sur $H^0(\A, I \to G)$ comme suit. Pour intégrer une fonction sur $H^0(\A, I \to G)$, on intègre d'abord le long de la fibre $I_\eta(\A) \backslash G(\A)$ munie de la mesure quotient. Puis on intègre sur $\mathfrak{D}(I, G; \A)$ avec la mesure décrite plus haut.

  Une recette analogue pour le cas local sera utilisée dans \S\ref{sec:int-orb}.
\end{enumerate}

\subsection{Mesures canoniques locales}\label{sec:mesures-locales}
\subsubsection{Mesure d'Euler-Poincaré}
Soient $F$ un corps local et $G$ un $F$-groupe réductif connexe. La mesure d'Euler-Poincaré $\mu^G_\text{EP}$ définie par Serre \cite[\S 3]{Se71} est une mesure signée invariante sur $G(F)$ qui est positive, nulle ou négative. On sait que $\mu^G_\text{EP} \neq 0$ si et seulement s'il existe un $F$-tore maximal anisotrope dans $G$. Dans le cas non-archimédien, cela équivaut à ce que $Z^\circ_G$ est anisotrope. Dans le cas archimédien, cela équivaut à l'existence des représentations de carré intégrable (i.e. les séries discrètes); d'après \cite[Lemma 7.3]{Gr97}, cela équivaut encore à l'existence d'une forme intérieure compacte notée $G^c$. Remarquons que $\mu^G_\text{EP} = 0$ si $F=\C$ et $G \neq \{1\}$.

On note $\rank_F(G)$ le $F$-rang de $G$. On définit
\begin{gather}\label{eqn:q}
  q(G) := \begin{cases}
            \rank_F(G_\text{der}), & F \text{ est non-archimédien}, \\
            \frac{1}{2} \dim G_\text{der}(F)/K_0, & F \text{ est archimédien}
          \end{cases}
\end{gather}
où $K_0$ est un sous-groupe compact maximal de $G_\text{der}$. Supposons l'existence d'un $F$-tore maximal anisotrope dans $G$, alors $q(G) \in \Z$ et le signe de $\mu^G_\text{EP}$ est $(-1)^{q(G)}$.

On aura besoin du signe de Kottwitz \cite{Ko83} défini comme suit. Notons $G_\text{qd}$ la forme intérieure quasi-déployée de $G$, alors $q(G) - q(G_\text{qd}) \in \Z$ et on peut poser
\begin{gather}\label{eqn:signe-Kottwitz}
  e(G) := (-1)^{q(G) - q(G_\text{qd})}.
\end{gather}
On a $e(G_1 \times G_2) = e(G_1) e(G_2)$, et $e(G)=1$ si $G$ est quasi-déployé.

Comme la définition dans \cite[p.133]{Se71} ne fait intervenir que le groupe localement compact $G(F)$, la mesure d'Euler-Poincaré est compatible à la restriction des scalaires. D'autre part, la mesure d'Euler-Poincaré est compatible au produit direct: on a $\mu^{G_1 \times G_2}_\text{EP} = \mu^{G_1}_\text{EP} \otimes \mu^{G_2}_\text{EP}$.

Le résultat suivant donne une caractérisation commode de $\mu^G_\text{EP}$.

\begin{lemma}\label{prop:deg-EP}
  Supposons que $G$ est un $F$-groupe réductif connexe tel que $\mu^G_{\mathrm{EP}} \neq 0$. Désignons par $G^c$ la forme intérieure compacte de $G$.
  \begin{itemize}
    \item Si $F=\R$, soit $\pi_0$ une représentation de carré intégrable de $G(\R)$ dont le caractère infinitésimal est égal, au moyen de l'isomorphisme de Harish-Chandra, à celui de la représentation triviale de $G^c(\R)$. Alors la mesure de Haar
      $$ (-1)^{q(G)} |\mathfrak{D}(T, G; \R)|^{-1} \mu^G_\mathrm{EP} $$
      satisfait à
      $$ d(\pi_0) = 1. $$
    \item Si $F$ est non-archimédien, notons $\pi_0$ la représentation de Steinberg de $G(F)$. Alors la mesure de Haar
      $$ (-1)^{q(G)} \mu^G_\mathrm{EP} $$
      satisfait à
      $$ d(\pi_0) = 1. $$
  \end{itemize}
  Ici, $d(\cdots)$ signifie le degré formel.
\end{lemma}
\begin{proof}
  C'est essentiellement dû au fait que les degrés formels et lesdites mesures de Haar sont conservées par torsion intérieure (voir \cite[p.2.23]{Sh83} et \cite[Theorem 1]{Ko88}), qui nous permet de nous ramener au cas trivial $G=G^c$. Voir aussi \cite[p.120]{Lab99} avec \cite[\S 4]{CL11}.
\end{proof}

\subsubsection{Mesure canonique locale de Gross}
Soit $G$ un $F$-groupe réductif connexe sur un corps local $F$. Une mesure de Haar canonique $\mu^G$ sur $G(F)$ est donnée par Gross \cite{Gr97} (aucune condition sur le groupe $G(F)$). Passons en revue ses constructions.

La mesure canonique $\mu^G$ est associée à une forme
$$ \omega^G \in \topwedge \mathfrak{g}^*, \quad \omega^G \neq 0, $$
vue comme une forme invariante sur $G(F)$. Plus précisément, c'est la densité invariante $|\omega^G|$ qui détermine une mesure de Haar $\mu^G$ sur $G(F)$.

Commençons par le cas archimédien. Quitte à effectuer une restriction de scalaires, on peut supposer $F=\R$. Supposons d'abord que $G$ est anisotrope. La forme invariante $\omega^G$ induit une forme invariante $\omega^{\mathfrak{g}}$ sur $\mathfrak{g}$. Macdonald \cite{Mc80} a indiqué un $\Z$-modèle explicite $\mathfrak{g}_\Z$ de $\mathfrak{g}$, et la densité invariante $|\omega^G|$ est caractérisée par
$$ \int_{\mathfrak{g}/\mathfrak{g}_\Z} |\omega^{\mathfrak{g}}| = 2^{|\Sigma|}, $$
où $\Sigma$ est un système de racines positives pour $\mathfrak{g}$.

En général, notons $G^c$ la forme compacte de $G$ (pas forcément intérieure). On choisit un isomorphisme $\Psi: G \times_\R \C \rightiso G^c \times_\R \C$. Prenons un sous-groupe compact maximal $K$ de $G$. On vérifie que la forme
\begin{gather}\label{eqn:omega-transfert}
  \omega^{G^c} \cdot i^{\dim G/K} \in \topwedge \mathfrak{g}^* \otimes \C
\end{gather}
se transfère au moyen de $\Psi^*$ à une forme $\omega^G \in \topwedge \mathfrak{g}^*$, $\omega^G \neq 0$. Cela définit $\mu^G$. Nous reviendrons sur ce point dans la Remarque \ref{rem:forme-compacte}.

Quant au cas non-archimédien, on peut supposer que $G$ est quasi-déployé, car on peut transférer les formes $\omega^G$ par torsion intérieure. Dans \cite[\S 4]{Gr97}, on a fait un choix d'un sommet spécial $x$ dans l'immeuble de Bruhat-Tits associé à $G$, canonique à $G_\text{AD}(F)$-conjugaison près. D'où un $\mathfrak{o}_F$-schéma en groupes lisse $\underline{G_x}^\circ$ à fibre générale égale à $G$. On prend pour $\omega^G$ une forme à bonne réduction par rapport au modèle $\underline{G_x}^\circ$. Alors $|\omega^G|$ est bien définie. Signalons que $x$ est hyperspécial si $G$ est non-ramifié.

Pour la commodité du lecteur, voici une récapitulation des propriétés de base.
\begin{proposition}\label{prop:mes-can-prop}
  Les mesures canoniques locales $\mu^G$ satisfont aux propriétés suivantes.
  \begin{enumerate}[i)]
    \item Les mesures $\mu^G$ sont conservées par torsion intérieure.
    \item On a $\mu^{G_1 \times G_2} = \mu^{G_1} \otimes \mu^{G_2}$.
    \item Soit $A$ un $F$-tore déployé central dans $G$, alors la suite exacte courte
      $$ 1 \to A(F) \to G(F) \to (G/A)(F) \to 1 $$
      respecte les mesures canoniques pour $A$, $G$ et $G/A$.
   \end{enumerate}
\end{proposition}
\begin{proof}
  Les premières deux assertions sont des conséquences immédiates de la construction et se trouvent déjà dans \cite{Gr97}. Traitons la dernière assertion.

  Considérons d'abord le cas non-archimédien. Vu la définition des schémas en groupes lisses $\underline{G_x}^\circ$, etc. dans la théorie de Bruhat-Tits, on se ramène aussitôt au cas $G=T$ est un $F$-tore. La mesure canonique est définie par une forme $\omega^T$ de degré maximal et à bonne réduction par rapport à $\underline{T}^\circ$, où $\underline{T} := \underline{T}^{\text{tf,NR}}$ désigne le modèle de Néron-Raynaud de $T$ de type fini \cite[(3.1)]{CY01}. \textit{Idem} pour $A$ et $T/A$. Prouvons que la suite
  $$ 1 \to \underline{A} \to \underline{T} \to \underline{A/T} \to 1$$
  est exacte; l'assertion en découlera en appliquant d'abord le foncteur exact $\underline{T} \mapsto \Lie(\underline{T}) = \Lie(\underline{T}^\circ)$ (voir \cite[II. Lemme 5.2.1 et $\text{VI}_\text{B}$. Remarque 3.2]{SGA3-1}), puis le foncteur $\topwedge$ pour des $\mathfrak{o}_F$-modules libres de type fini. En effet, cela est \cite[Lemma 11.2]{CY01} si $A$ est la partie déployée de $T$ (tel est le cas pour notre application dans \S\ref{sec:preuve-nonarch}), mais on vérifie que les arguments marchent en général.

  Pour le cas archimédien, on se ramène de façon similaire au cas $G=T$, ce qui est élémentaire.
\end{proof}

On va étudier les liens entre ces mesures locales ainsi que des conséquences arithmétiques dans la sous-section suivante.

\subsection{Motifs et les équations fonctionnelles}\label{sec:motifs}
\subsubsection{Motifs d'Artin-Tate}
Pour un corps $F$, la catégorie des motifs d'Artin sur $F$ est la $\otimes$-catégorie (sur $\Q$) des représentations continues de $\Gamma_F$ de dimension finie à coefficients dans $\Q$. On obtient la $\otimes$-catégorie des motifs d'Artin-Tate sur $F$ en y rajoutant formellement un \guillemotleft objet de Tate\guillemotright\,inversible $\Q(1)$ ainsi que toutes ses puissances tensorielles. Voici quelques opérations sur ces trucs.

\begin{itemize}
  \item \emph{Torsion à la Tate}. On écrit $\Q(k) := \Q(1)^{\otimes k}$ avec $\Q(0) := \Q$ (la représentation triviale de $\Gamma_F$), et
    $$ M(k) := M \otimes \Q(k), \quad k \in \Z $$
    pour tout motif d'Artin-Tate $M$.
  \item \emph{Contragrédient}. Le foncteur contragrédient $V \mapsto V^\vee$ pour les représentations se généralise aux motifs d'Artin-Tate, en posant $\Q(k)^\vee := \Q(-k)$.
  \item \emph{Localisation}. Soient $F$ un corps global et $v$ une place de $F$, alors on a un foncteur $M \mapsto M_v$ correspondant à l'inclusion $\Gamma_{F_v} \subset \Gamma_F$ et la restriction des représentations.
\end{itemize}

Supposons que $F$ est un corps local. À un motif d'Artin-Tate est associée la fonction $L$ locale $L(s,M)$, qui est une fonction méromorphe en $s \in \C$, et on pose
$$ L(M) := L(0,M) \in \C \sqcup \{\infty\}. $$
Si $M$ est un motif d'Artin, alors $L(M,s)$ est la fonction $L$ d'Artin familière. En général, on a
\begin{align*}
  L(s, M_1 \oplus M_2) & = L(s, M_1) L(s, M_2), \\
  L(s, M(k)) & = L(s+k, M), \quad k \in \Z
\end{align*}
pour les motifs d'Artin-Tate.

Supposons maintenant que $F$ est un corps global. On définit la fonction $L$ partielle $L^S(s, M)$ pour tout ensemble fini $S$ de places de $F$ contenant les places archimédiennes, et la fonction $L$ complétée $\Lambda(s, M) := \prod_{v \in S} L(s, M_v) \cdot L^S(s, M)$. Elles se réduisent aux fonctions $L$ d'Artin lorsque $M$ est un motif d'Artin, et elles satisfont aux propriétés ci-dessus. On pose
$$
  \Lambda(M) := \Lambda(0,M), \quad L^S(M) := L^S(0,M).
$$

L'équation fonctionnelle d'Artin s'écrit comme suit
\begin{gather}
  \Lambda(M) = \varepsilon(M) \Lambda(M^\vee(1))
\end{gather}
où $\varepsilon(M)$ est le facteur $\varepsilon$ associé à $M$. Pour les motifs d'Artin-Tate, $\varepsilon(M)$ est positif et peut s'exprimer en termes du discriminant $d_F$ de $F/\Q$ et des conducteurs d'Artin; le \guillemotleft nombre de racines\guillemotright\, dans l'équation fonctionnelle d'Artin disparaît pour nos motifs, d'après le théorème de Fröhlich et Queyrut \cite[Theorem 1]{FQ73}.

Soit $G$ un $F$-groupe réductif connexe. Gross a associé un motif d'Artin-Tate sur $F$ à $G$, ce que l'on note par $M_G$. Sa construction précise ne nous concerne pas ici. Il suffit de savoir les propriétés suivantes \cite[Lemma 2.1]{Gr97}.

\begin{enumerate}[i)]
  \item $M_G$ ne dépend que de la classe d'isogénie de $G$.
  \item Si $G'$ est une forme intérieure de $G$, alors $M_{G'} = M_G$.
  \item $M_{G_1 \times G_2} = M_{G_1} \oplus M_{G_2}$.
\end{enumerate}

Heureusement, l'exemple crucial pour nous est déjà calculé dans \cite[p.290]{Gr97}.
\begin{example}\label{ex:motif}
  Les groupes déployés $\SO(2n+1)$ et $\Sp(2n)$ partagent le même motif
  $$ M = \Q(-1) \oplus \Q(-3) \oplus \cdots \oplus \Q(1-2n). $$
\end{example}

\begin{remark}
  Lorsque $F$ est un corps local et $M$ est le motif associé à un $F$-groupe réductif connexe, on a $L(M) \neq 0$ et $L(M^\vee(1)) \neq 0, \infty$. Ce fait est implicite dans \cite[\S5 et \S 7]{Gr97}.
\end{remark}

\subsubsection{Équations fonctionnelles}
Supposons d'abord que $F$ est un corps local.

\begin{theorem}\label{prop:eq-fonc-locale}
  Soit $G$ un $F$-groupe réductif connexe. La mesure d'Euler-Poincaré $\mu^G_{\mathrm{EP}}$ et la mesure canonique de Gross $\mu^G$ sont reliées comme suit.
  \begin{enumerate}[i)]
    \item On a $\mu^G_{\mathrm{EP}} \neq 0$ si et seulement si $L(M_G) \neq \infty$.
    \item Supposons que $L(M_G) \neq \infty$, alors on a
    $$ L(M_G) \cdot \frac{|H^1(F,T)| e(G)}{|\mathfrak{D}(T, G; F)|} \cdot \mu^G_{\mathrm{EP}} = L(M^\vee_G(1)) \cdot \mu^G, $$
    où $T$ est un $F$-tore maximal anisotrope quelconque dans $G$.
  \end{enumerate}
  On renvoie à \S\ref{sec:mesures-locales} et Définition \ref{def:glossaire} pour les définitions des objets $e(G)$, $\mathfrak{D}(T, G; F)$, etc.
\end{theorem}
\begin{proof}
  C'est une combinaison de \cite[Proposition 5.3, Lemma 7.3]{Gr97} et \cite[Theorem 8.1]{Gr97}.
\end{proof}

Supposons dès maintenant que $F$ est un corps de nombres, et $G$ un $F$-groupe réductif connexe comme d'habitude. Le résultat suivant est bien connu, cf. \cite[Proposition 9.4]{Gr97}.

\begin{proposition}\label{prop:anisotrope-crit}
  Les conditions suivantes sont équivalentes
  \begin{enumerate}[i)]
    \item $Z^\circ_G$ est anisotrope;
    \item $\mes(G(F) \backslash G(\A))$ est finie;
    \item $L^S(M^\vee_G(1)) \neq \infty$ pour un ensemble fini de places $S$ contenant les places archimédiennes;
    \item $\Lambda(M^\vee_G(1)) \neq \infty$.
  \end{enumerate}
\end{proposition}

Supposons donc que les conditions ci-dessus sont satisfaites par $G$. On observe que $\Lambda(M^\vee_G(1)) \neq 0$ (voir la démonstration de \cite[Proposition 9.4]{Gr97}). C'est donc loisible de définir une mesure de Haar canonique sur $G(\A)$ par
\begin{gather}\label{eqn:Gross-global}
  \mu^G := \Lambda(M^\vee_G(1))^{-1} \cdot \prod_v L(M^\vee_{G,v}(1)) \mu^{G_v}.
\end{gather}

Les facteurs $L$ locaux dans \eqref{eqn:Gross-global} sont des facteurs de convergence au sens de \cite[\S 2.3]{Weil82}, pour l'essentiel. En effet, $L(M^\vee_{G, v}(1)) \mu^{G_v}$ est la mesure non-ramifiée lorsque $G_v := G \times_F F_v$ est un groupe non-ramifié d'après \cite[Proposition 4.7]{Gr97}. Comparons-le avec la mesure de Tamagawa sur $G(\A)$.

Choisissons un élément rationnel $\omega \in \topwedge \mathfrak{g}^*$ quelconque avec $\omega \neq 0$. Pour toute place $v$, on en déduit une densité invariante $|\omega|_v$, ce que l'on regarde comme une mesure de Haar sur $G(F_v)$. Alors pour presque toute place finie $v$, on a $|\omega|_v = \mu^{G_v}$ par la construction de la mesure canonique de Gross. On pose
$$ d_F := \text{ le discriminant de } F/\Q. $$

Sous les conditions dans la Proposition \ref{prop:anisotrope-crit}, la mesure de Tamagawa est définie par le produit convergent
\begin{gather}\label{eqn:Tamagawa}
  \mu^G_\text{Tama} = |d_F|^{-\frac{\dim G}{2}} \Lambda(M^\vee_G(1))^{-1} \cdot \prod_v L(M^\vee_{G, v}(1)) |\omega|_v
\end{gather}

\begin{proposition}\label{prop:epsilon-M}
  Supposons que l'une des conditions dans la Proposition \ref{prop:anisotrope-crit} soit satisfaite, alors on a l'égalité dans l'espace des mesures invariantes
  $$ \varepsilon(M_G) \cdot \mu^G_{\mathrm{Tama}} = \mu^G. $$
\end{proposition}
\begin{proof}
  Cela résulte en comparant \eqref{eqn:Gross-global} et \eqref{eqn:Tamagawa} à l'aide de la formule \cite[Theorem 11.5]{Gr97} suivante
  $$ \prod_v \dfrac{\mu^{G_v}}{|\omega|_v} = \varepsilon(M_G) |d_F|^{-\frac{\dim G}{2}} $$
  dont le terme à gauche est un produit fini.
\end{proof}

Ainsi, on obtient une décomposition canonique de $\mu^G_\text{Tama}$ en des composantes locales explicites.

\section{Transfert équi-singulier: les énoncés}\label{sec:transfert-es}
\subsection{Conjugaison stable sur $\tilde{G}$}
Suivant Adams \cite{Ad98}, on a défini dans \cite[\S 5.2]{Li11} la notion de conjugaison stable pour les éléments semi-simples réguliers dans le revêtement métaplectique. Il faut la généraliser aux éléments semi-simples quelconques.

Dans cette sous-section, $F$ désigne un corps local ou global de caractéristique nulle. On fixe $(W, \angles{\cdot|\cdot})$ un $F$-espace symplectique de dimension $2n$. On note $G := \Sp(W)$ comme d'habitude.

\subsubsection{Le cas local}
Soient $F$ un corps local de caractéristique nulle et $\psi: F \to \C^\times$ un caractère unitaire non-trivial. À ces données est associé le revêtement métaplectique local $\rev: \tilde{G} \to G(F)$ avec $\Ker(\rev) = \bmu_8$.

Rappelons que dans \cite[\S 5.2]{Li11}, deux éléments semi-simples réguliers $\tilde{\delta}, \tilde{\delta}' \in \tilde{G}$ sont dits stablement conjugués (ou géométriquement conjugués) si:
\begin{itemize}
  \item leurs images $\delta, \delta' \in G(F)$ sont stablement conjugués, et
  \item $\Theta^+_\psi - \Theta^-_\psi$ prend la même valeur en $\tilde{\delta}$ et $\tilde{\delta}'$.
\end{itemize}
Nous allons le généraliser à tous les éléments semi-simples dans $\tilde{G}$.

Soient $\delta, \delta' \in G(F)_\text{ss}$ avec paramètres
\begin{gather*}
  K/K^\sharp, x, (W_K, h_K), (W_\pm, \angles{\cdot|\cdot}_\pm), \quad \text{et} \\
  K'/K'^\sharp, x', (W'_{K'}, h'_{K'}), (W'_\pm, \angles{\cdot|\cdot}'_\pm)
\end{gather*}
respectivement (voir \S\ref{sec:parametrage}). Soient $\tilde{\delta} \in \rev^{-1}(\delta)$ et $\tilde{\delta}' \in \rev^{-1}(\delta')$.

\begin{definition}\label{def:conj-stable-local}
  On dit que $\tilde{\delta}$ et $\tilde{\delta}'$ sont stablement conjugués (ou géométriquement conjugués) si les deux conditions suivantes sont satisfaites.
  \begin{enumerate}
    \item On exige que $\delta$ et $\delta'$ soient stablement conjugués dans $G$.
    \item Décomposons l'espace symplectique $W$ en une somme orthogonale
      $$ W = W_K \oplus W_+ \oplus W_- $$
      selon les valeurs propres de $\delta$: les espaces $W_\pm$ correspondent aux valeurs propres $\pm 1$, et $W_K$ correspond aux autres valeurs propres (disons celles provenant de $x$). On peut donc décomposer $\tilde{\delta}$ de manière canonique en écrivant
      \begin{align*}
        \Mp(W_K) \times \Mp(W_+) \times \Mp(W_-) & \longrightarrow \Mp(W) \\
        (\tilde{\delta}_K, +1, -1) & \longmapsto \tilde{\delta}.
      \end{align*}

      De même, on a une analogue $(\tilde{\delta}'_{K'}, +1, -1) \mapsto \tilde{\delta}'$ par rapport à une décomposition $W = W'_{K'} \oplus W'_+ \oplus W'_-$. En supposant $\delta$ et $\delta'$ stablement conjugués, c'est loisible de supposer que $K=K'$, $K^\sharp = K'^\sharp$, $x=x'$, $W'_{K'} = W_K$ et $(W_\pm, \angles{\cdot|\cdot}_\pm) = (W'_\pm, \angles{\cdot|\cdot}'_\pm)$ dans les paramètres, pour simplifier les notations.

      Alors on exige que
      \begin{gather}
        (\Theta^+_\psi - \Theta^-_\psi)(\tilde{\delta}_K) = (\Theta^+_\psi - \Theta^-_\psi)(\tilde{\delta}'_K)
      \end{gather}
      où les caractères $\Theta^\pm_\psi$ sont définis sur $\Mp(W_K)$. Remarquons que ces valeurs de caractère sont bien définies et non nulles car $\delta_K$ et $\delta'_K$ n'ont pas de valeurs propres $\pm 1$. Dans le cas $W_K = \{0\}$, la conjugaison stable équivaut à la conjugaison par $G(F)$.
  \end{enumerate}
\end{definition}

En résumé, chaque classe semi-simple stable dans $\tilde{G}$ est un relèvement analytique $G(F)$-équivariant d'une classe semi-simple stable dans $G(F)$.

\begin{lemma}\label{prop:conjugue-existe}
  Soient $\delta, \delta' \in G(F)_\mathrm{ss}$ stablement conjugués. Étant donné $\tilde{\delta} \in \rev^{-1}(\delta)$, il existe un unique $\tilde{\delta}' \in \rev^{-1}(\delta')$ qui est stablement conjugué à $\tilde{\delta}$.
\end{lemma}
\begin{proof}
  Notons que $(\Theta^+_\psi - \Theta^-_\psi)(\tilde{x}) = \Theta_\psi(-\tilde{x})$. Alors l'assertion découle de \cite[Corollaire 4.5]{Li11} et de \cite[Théorème 4.2 (ii)]{Li11} qui dit que $(\Theta^+_\psi - \Theta^-_\psi)(\tilde{\delta}_K)$ et $(\Theta^+_\psi - \Theta^-_\psi)(\tilde{\delta}'_K)$ diffèrent par un élément dans $\bmu_8$, pour tout $\tilde{\delta}' \in \rev^{-1}(\delta')$.
\end{proof}

Évidemment, si $\tilde{\delta}$ et $\tilde{\delta}'$ sont stablement conjugués, alors $\noyau\tilde{\delta}$ et $\noyau\tilde{\delta}'$ le sont aussi, pour tout $\noyau \in \bmu_8$.

Utilisons maintenant le langage du \S\ref{sec:coho-abel}.

\begin{definition}\label{def:adjoint-stable}
  Soient $\delta \in G(F)_\text{ss}$ et $\tilde{\delta} \in \rev^{-1}(\delta)$. Pour tout $x \in H^0(F, G_\delta \to G) = (G_\delta \backslash G)(F)$, on désigne par $x^{-1} \tilde{\delta}x$ l'unique élément dans $\rev^{-1}(x^{-1}\delta x)$ qui est stablement conjugué à $\tilde{\delta}$, dont l'existence est garantie par le Lemme \ref{prop:conjugue-existe}.

  On observe que pour $x$ provenant de $G(F)$, cela revient à l'action adjointe usuelle de $x^{-1}$ sur $\tilde{G}$ car $\Theta^+_\psi - \Theta^-_\psi$ est une distribution invariante.
\end{definition}

Munissons $H^0(F, G_\delta \to G)$ de la structure naturelle d'une variété $F$-analytique. Alors $x \mapsto x^{-1} \delta x$ fournit une bijection lisse de $H^0(F, G_\delta \to G)$ sur l'ensemble  des $\tilde{\delta}' \in \tilde{G}_\text{ss}$ stablement conjugués à $\tilde{\delta}$. La lissité est facile à vérifier; par exemple, on peut utiliser la régularité du caractère $\Theta_\psi$ hors de $\{\tilde{x} \in \tilde{G} : \det(x^2-1) = 0\}$. On vérifie aussi que l'image de $H^0(F, G_\delta \to G)$ est un sous-ensemble fermé de $\tilde{G}$.

Ensuite, observons la compatibilité suivante entre l'action adjointe ci-dessus et l'action adjointe usuelle par $G(F)$. Pour distinguer les deux, on note momentanément $\Ad(y^{-1}): \tilde{\delta} \to y^{-1} \tilde{\delta} y$ l'action adjointe sur $\tilde{G}$ par $y \in G(F)$.

\begin{lemma}\label{prop:adjoint-stable-compat}
  Soient $\delta, \tilde{\delta}$ comme précédemment. Alors
  \begin{itemize}
    \item pour $y \in G(F)$ et $x \in H^0(F, G_{y^{-1}\delta y} \to G)$, on a
      $$ x^{-1} (\Ad(y^{-1})(\tilde{\delta})) x  = (yx)^{-1} \tilde{\delta} (yx); $$
    \item pour $z \in G(F)$ et $x \in H^0(F, G_\delta \to G)$, on a
      $$ \Ad(z^{-1})(x^{-1}\tilde{\delta} x) = (xz)^{-1} \tilde{\delta} (xz). $$
  \end{itemize}
  Ici, on a utilisé l'isomorphisme canonique $H^0(F, G_{y^{-1}\delta y} \to G) \to H^0(F, G_\delta \to G)$ fourni par $x \mapsto yz$, et l'opération de $G(F)$ à droite sur $H^0(F, G_\delta \to G)$ par multiplication.
\end{lemma}
\begin{proof}
  Tout cela découle de l'invariance de $\Theta^+_\psi - \Theta^-_\psi$ et de la Définition \ref{def:conj-stable-local}.
\end{proof}

\subsubsection{Le cas global}
Soient $F$ un corps de nombres et $\psi = \prod_v \psi_v : \A/F \to \C^\times$ un caractère unitaire non-trivial. On construit le revêtement métaplectique adélique $\rev: \tilde{G} \to G(\A)$ avec $\Ker(\rev) = \bmu_8$. On regarde $G(F)$ comme un sous-groupe discret de $\tilde{G}$. Pour toute place $v$ de $F$, on identifie $\tilde{G}_v$, la fibre de $\rev$ au-dessus de $G(F_v)$, au revêtement local $\Mp(W_v)$. On a une surjection $\Resprod_v \tilde{G}_v \twoheadrightarrow \tilde{G}$.

Rappelons que les cohomologies adéliques que nous considérons sont définies au moyen du produit restreint dans \S\ref{sec:coho-abel}.

\begin{definition}\label{def:conj-stable-global}
  Soient $\delta \in G(F)_\text{ss}$ et $x = (x_v)_v \in H^0(\A, G_\delta \to G)$. Prenons une image réciproque $(\tilde{\delta}_v)_v \in \Resprod_v \tilde{G}_v$ de $\delta \in G(F)$ sous l'application $\Resprod_v \tilde{G}_v \twoheadrightarrow \tilde{G} \supset G(F)$. On définit $x^{-1} \delta x$ comme l'image de $(x^{-1}_v \tilde{\delta}_v x_v)_v$ dans $\tilde{G}$. Cela est visiblement indépendant du choix de l'image réciproque.

  On dit qu'un élément $\delta' \in G(\A)$ est localement stablement conjugué à $\delta$ s'il existe $x \in H^0(\A, G_\delta \to G)$ tel que $x^{-1} \delta x = \delta'$. De même, on dit qu'un élément $\tilde{\delta}' \in \tilde{G}$ est localement stablement conjugué à $\delta$ s'il existe $x \in H^0(\A, G_\delta \to G)$ tel que $x^{-1} \delta x = \tilde{\delta}'$.
\end{definition}

Par conséquence, $x \mapsto x^{-1} \delta x$ définit une bijection lisse de $H^0(\A, G_\delta \to G)$ sur l'ensemble des éléments dans $\tilde{G}$ localement stablement conjugués à $\delta$. L'image est un sous-ensemble fermé de $\tilde{G}$. Cf. le cas local discuté après la Définition \ref{def:adjoint-stable}.

\begin{remark}
  Il faut vérifier que $(x^{-1}_v \tilde{\delta}_v x_v)_v$ appartient à $\Resprod_v \tilde{G}_v$. Pour cela, on applique le Lemme \ref{prop:I-G-adelique} qui affirme que $x_v$ provient du sous-groupe hyperspécial $K_v$ pour presque toute $v$.
\end{remark}

Remarquons aussi qu'étant donnés $\delta$ et $x$, on peut effectuer la conjugaison stable \guillemotleft usuelle\guillemotright, i.e. sans revêtement. Cela donne un élément dans $G(\A)$, noté momentanément $(x^{-1}\delta x)_{G(\A)}$. Nous n'utiliserons que la version sur $\tilde{G}$ lors de la stabilisation. Ceci est justifié par le résultat suivant.

\begin{lemma}\label{prop:obstruction}
  Soient $\delta \in G(F)_\mathrm{ss}$ et $x \in H^0(\A, G_\delta \to G)$. Identifions $\delta$ avec son image dans $\tilde{G}$ via la section $G(F) \hookrightarrow \tilde{G}$ et formons $x^{-1}\delta x \in \tilde{G}$ par la procédure précédente. Formons d'autre part $(x^{-1} \delta x)_{G(\A)} \in G(\A)$. Si $(x^{-1} \delta x)_{G(\A)} \in G(F)$, alors $x^{-1} \delta x$ est égal à l'image de $(x^{-1} \delta x)_{G(\A)}$ par $G(F) \hookrightarrow \tilde{G}$. En particulier, $x^{-1} \delta x$ appartient à $G(F)$.
\end{lemma}
En résumé, la classe de conjugaison stable de $\delta \in G(F) \hookrightarrow \tilde{G}$ contient les mêmes $G(F)$-points que la classe de conjugaison stable de $\delta$ dans $G(\A)$.

\begin{proof}
  Posons $\eta := (x^{-1} \delta x)_{G(\A)}$, alors on a une décomposition
  $$ \eta = (\eta_K, +1, -1) \in \Sp(W_K) \times \Sp(W_+) \times \Sp(W_-), $$
  comme dans la Définition \ref{def:conj-stable-local}, mais toute opération est effectuée sur $F$.

  On regarde $\eta$ comme un élément de $\tilde{G}$ via $G(F) \hookrightarrow \tilde{G}$. La formule du produit \cite[Théorème 4.8]{Li11} dit que
  $$ \left(\prod_v (\Theta^+_{\psi_v} - \Theta^-_{\psi_v}) \right) (\eta_K) = \left( \prod_v \Theta_{\psi_v} \right)(-\eta_K) = 1. $$

  De même, si l'on décompose
  $$ \delta = (\delta_K, +1, -1) \in \Sp(W_K) \times \Sp(W_+) \times \Sp(W_-), $$
  ce qui est loisible car $\delta$ et $\eta$ sont conjugués par $G(\bar{F})$, alors
  $$ \left(\prod_v (\Theta^+_{\psi_v} - \Theta^-_{\psi_v}) \right) (\delta_K) = 1. $$

  D'autre part, en se rappelant la Définition \ref{def:conj-stable-local}, on voit que $x^{-1} \delta x$ se décompose aussi en $x^{-1}\delta x = (\tilde{\xi}_K, +1, -1)$ avec $\tilde{\xi}_K \in \Mp(W_K)$. Selon les Définitions \ref{def:conj-stable-local} et \ref{def:conj-stable-global}, on a
  $$ \left( \prod_v (\Theta^+_{\psi_v} - \Theta^-_{\psi_v}) \right)(\tilde{\xi}_K) = \left(\prod_v (\Theta^+_{\psi_v} - \Theta^-_{\psi_v}) \right) (\delta_K). $$

  Or on sait que $x^{-1} \delta x \in \tilde{G}$ s'envoie sur $\eta = (x^{-1} \delta x)_{G(\A)}$ par $\rev$, donc ils sont forcément égaux.
\end{proof}

\subsection{Correspondance équi-singulière et les facteurs de transfert}
Soit $F$ un corps local ou global de caractéristique nulle. On se donne un $F$-espace symplectique $(W, \angles{\cdot|\cdot})$ de dimension $2n$. Posons $G := \Sp(W)$ et prenons une donnée endoscopique elliptique $(n',n'')$ pour le revêtement métaplectique de $G(F)$ (le cas local) ou de $G(\A)$ (le cas global). Le groupe endoscopique correspondant est
$$ H := \underbrace{\SO(2n'+1)}_{H'} \times \underbrace{\SO(2n''+1)}_{H''}. $$

Soient $\delta \in G(F)_\text{ss}$ et $\gamma = (\gamma', \gamma'') \in H(F)_\text{ss}$. On utilise les paramétrages pour $\delta$ et $\gamma$ dans \S\ref{sec:correspondance}.

\begin{definition}\label{def:equi-sing}
  Supposons que $\gamma$ et $\delta$ se correspondent. On dit que la correspondance $\gamma \leftrightarrow \delta$ est équi-singulière si les paramètres pour $\gamma$ et $\delta$ dans \S\ref{sec:correspondance} satisfont à:
  \begin{itemize}
    \item il n'y a pas de \guillemotleft fusion\guillemotright\,entre les valeurs propres de $a'$ et $-a''$, i.e. leurs images sous les applications dans $\Hom_{F-\text{alg}}(\star, \bar{F})$ (avec $\star = K'$ ou $K''$) n'entrelacent pas ;
    \item $\gamma'$ et $\gamma''$ n'ont pas de valeur propre $-1$, i.e. $V''_- = \{0\}$ et $V'_- = \{0\}$.
  \end{itemize}

  Les conditions ne concernent que $H$ et la classe de conjugaison stable de $\gamma$. Posons en conséquence
  \begin{gather}\label{eqn:equi-sing-ens}
    \Sigma_{\text{équi}}(H) := \left\{ \gamma \in \Sigma_\text{ss}(H) : \exists \delta \in \Gamma_\text{ss}(G) \text{ tel que } \gamma \leftrightarrow \delta \text{ est équi-singulier} \right\}.
  \end{gather}
\end{definition}

Observons que la première condition équivaut à $(K, K^\sharp) = (K', K'^\sharp) \times (K'', K''^\sharp)$ en tant que des $F$-algèbres.

\begin{example}\label{ex:equi-sing-ex}
  Si $n'=n$, $H = \SO(2n+1)$, alors $\delta=1$ et $\gamma=1$ est en correspondance équi-singulière. Si $n''=n$, $H=\SO(2n+1)$, alors $\delta=-1$ et $\gamma=1$ est en correspondance équi-singulière.
\end{example}

\begin{proposition}\label{prop:equi-sing-commutants}
  Si $\gamma \leftrightarrow \delta$ est équi-singulière, alors
  \begin{enumerate}[i)]
    \item on a $H^\gamma = H_\gamma$;
    \item les commutants sont de la forme
      \begin{align*}
        G_\delta & = U(W'_{K'}, h'_{K'}) \times U(W''_{K''}, h''_{K''}) \times \Sp(W_+) \times \Sp(W_-), \\
        H_\gamma & = H'_{\gamma'} \times H''_{\gamma''}, \text{ où } \\
        H'_{\gamma'} & = U(W'_{K'}, r'_{K'}) \times \SO(V'_+, q'_+), \\
        H''_{\gamma''} & = U(W''_{K''}, r''_{K''}) \times \SO(V''_+, q''_+)
      \end{align*}
      où $h'_{K'}, h''_{K''}$ (resp. $r'_{K'}, r''_{K''}$) sont des formes anti-hermitiennes (resp. hermitiennes) sur $W'_{K'}$ et $W''_{K''}$, respectivement, et
      \begin{align*}
        \dim_F V'_+ &= \dim_F W_+ + 1, \\
        \dim_F V''_+ &= \dim_F W_- + 1;
      \end{align*}
    \item le $F$-groupe $U(W'_{K'}, h'_{K'})$ (resp. $U(W''_{K''}, h''_{K''})$) est une forme intérieure de $U(W'_{K'}, r'_{K'})$ (resp. $U(W''_{K''}, r''_{K''})$).
  \end{enumerate}
\end{proposition}
\begin{proof}
  La première assertion résulte de l'absence de valeur propre $-1$ du côté de $H$. Pour la même raison, on déduit les égalités dans la deuxième assertion. La condition \guillemotleft sans fusion\guillemotright\,dans la Définition \ref{def:equi-sing} entraîne que $(K, K^\sharp) \simeq (K', K'^\sharp) \times (K'', K''^\sharp)$, sous lequel le générateur $a \in K$ s'envoie sur $(a', -a'') \in K' \times K''$. On peut donc séparer le facteur $U(W_K, h_K) = U(W'_{K'}, h'_{K'}) \times U(W''_{K''}, h''_{K''})$ dans $G_\delta$ selon la provenance des valeurs propres (rappelons \S\ref{sec:parametrage}). D'où la deuxième assertion.

  Pour la troisième assertion, on utilise le fait qu'étant donnés $m \geq 1$ et une extension quadratique $E/F$, il n'y a qu'un seul groupe unitaire quasi-déployé associé à un $E$-espace hermitien ou anti-hermitien de dimension $m$, à isomorphisme près.
\end{proof}

\begin{remark}
  Le mot \guillemotleft équi-singulier\guillemotright\, est emprunté de \cite[\S 2.4]{LS2}, ce qui provient de la notion de $(G,H)$-régularité dans \cite[\S 3]{Ko86}. Pour l'endoscopie standard des groupes réductifs connexes, lorsque l'on se descend aux commutants connexes des éléments semi-simples en correspondance équi-singulière (cf. \cite{LS2}), on se ramène au cas le plus simple de l'endoscopie, à savoir la torsion intérieure.

  Pour le revêtement métaplectique, notre approche est basée sur le paramétrage explicite. Il y a d'ailleurs une différence cruciale. On a vu que les facteurs unitaires des commutants sont toujours reliées par torsion intérieure. Par contre, les facteurs correspondant aux valeurs propres $\pm 1$ sont reliés par l'endoscopie non-standard \cite[\S 8.2]{Li11} entre systèmes de racines de type $\mathbf{B}_n$ et $\mathbf{C}_n$.
\end{remark}

Dès maintenant, $F$ désigne un corps local de caractéristique nulle, sauf mention expresse du contraire.

\begin{definition}\label{def:facteur-equi-sing}
  Soient $\gamma \in H(F)_\text{ss}$ et $\delta \in G(F)_\text{ss}$ sont en correspondance équi-singulière, et $\tilde{\delta} \in \rev^{-1}(\delta)$. On définit le facteur de transfert
  $$ \Delta(\gamma, \tilde{\delta}) := \lim_{\substack{\gamma_1 \to \gamma \\ \tilde{\delta}_1 \to \tilde{\delta}}} \Delta(\gamma_1, \tilde{\delta}_1) $$
  où les paires $(\gamma_1, \delta_1)$ sont des éléments semi-simples réguliers qui se correspondent, et $\tilde{\delta}_1 \in \rev^{-1}(\delta_1)$. 
\end{definition}

Cette limite existe d'après la descente semi-simple \cite[\S 7]{Li11} et la Définition \ref{def:equi-sing}. De plus, $\Delta$ est localement constant près de $(\gamma, \tilde{\delta}$). En effet, on peut supposer que
\begin{align*}
  \gamma_1 & = \exp(Y)\gamma, \quad Y \in \mathfrak{h}_\gamma(F), \\
  \tilde{\delta}_1 & = \exp(X)\tilde{\delta}, \quad X \in \mathfrak{g}_\delta(F),
\end{align*}
avec $X,Y$ assez proches de $0$. On vérifie que la formule de \cite[Théorème 7.10]{Li11} pour $\Delta(\exp(Y) \gamma, \exp(X) \tilde{\delta})$ tend vers une constante lorsque $(X,Y) \to (0,0)$, par l'équi-singularité.

\begin{example}\label{ex:Delta}
  Revenons à l'Exemple \ref{ex:equi-sing-ex}. Si $n'=n$, alors $\Delta(1,1)=1$. Si $n''=n$, alors $\Delta(1,-1)=1$. Ceci résulte de \cite[Propositions 7.5 et 7.6]{Li11a} et de la formule \cite[Corollaire 4.6]{Li11} corrigée
  \begin{gather}\label{eqn:Theta-valeur}
    \Theta_\psi(-1) = (\Theta^+_\psi - \Theta^-_\psi)(1) = |2|^{-n}_F.
  \end{gather}
\end{example}

\begin{definition}\label{def:kappa}
  Le caractère endoscopique $\kappa \in \mathfrak{R}(G_\delta, G; F)$ associé aux données $(n',n'',\gamma)$ est défini par la procédure suivante. D'après la Définition \ref{def:equi-sing}, on peut écrire les commutants sous la forme
  \begin{align*}
    G_\delta & = U' \times U'' \times \Sp(W_+) \times \Sp(W_-), \\
    H'_{\gamma'} & = U' \times \SO(V'_+, q'_+), \\
    H''_{\gamma''} & = U'' \times \SO(V'_-, q'_-),
  \end{align*}
  où $U'$ (resp. $U''$) est un produit direct de groupes unitaires et de groupes linéaires généraux, à restriction des scalaires près. Notons le nombre des groupes unitaires qui y apparaissent par $s'$ (resp. $s''$). D'après l'Exemple \ref{ex:coho-U}, on a $H^1_\text{ab}(F, U') = \{\pm 1\}^{s'}$, $H^1_\text{ab}(F, U'') = \{\pm 1\}^{s''}$ où chaque groupe unitaire dans la décomposition de $U'$ (resp. de $U''$) contribue un facteur direct $\{\pm 1\}$. Puisque le complexe abélianisé $G_\text{ab}$ est trivial, on a donc
  $$ H^0_\text{ab}(F, G_\delta \to G) = H^1_\text{ab}(F, G_\delta) = \{\pm 1\}^{s'} \times \{\pm 1\}^{s''}. $$

  Alors $\kappa: H^0_\text{ab}(F, G_\delta \to G) \to \bmu_2$ est le caractère qui est trivial sur $\{\pm 1\}^{s'}$ et de la forme $(x_i)_{i=1}^{s''} \mapsto x_1 \cdots x_{s''}$ sur $\{\pm 1\}^{s''}$.

  Si $F$ est un corps de nombres, on définit $\kappa \in \mathfrak{R}(G_\delta, G; F)$ par la même recette en remplaçant $H^\bullet_\text{ab}(F, \cdots)$ par $H^\bullet_\text{ab}(\A/F, \cdots)$ partout. Il se localise en les caractères endoscopiques locaux au moyen de l'homomorphisme $\mathfrak{R}(G_\delta, G; F) \to \mathfrak{R}(G_\delta, G; \A)$.
\end{definition}

\begin{remark}
  Par convention, $\kappa$ est aussi regardé comme une fonction sur $H^0(F, G_\delta \to G)$ à l'aide de l'application $H^0(F, G_\delta \to G) \to H^0_\text{ab}(F, G_\delta \to G)$. On vérifie qu'elle est invariante par multiplication à droite par $G(F)$, donc elle se factorise par $\mathfrak{D}(G_\delta, G; F)$.
\end{remark}

\begin{definition}\label{def:fondamental}
  Rappelons qu'un $F$-tore maximal $T$ dans un $F$-groupe réductif $I$ est dit fondamental \cite[\S 10]{Ko86} si $T$ est $F$-elliptique (i.e. $T/Z_I$ est anisotrope) pour $F$ non-archimédien, et si $T$ est maximalement compact \cite[p.255]{KV95} pour $F$ archimédien.
\end{definition}

Les tores fondamentaux existent pour tout $I$. De plus, $H^1(F,T) \to H^1(F, I)$ est surjectif si $T$ est fondamental, d'après \cite[10.2 Lemma]{Ko86}. Puisque $\text{ab}^1_I$ est surjectif, $H^1(F,T) \to H^1_\text{ab}(F,I)$ l'est aussi.

Prenons un $F$-tore fondamental $T$ dans $G_\delta$. Rappelons que l'on a défini le caractère endoscopique $\kappa_T \in \mathfrak{R}(T, G; F)$ dans \cite[\S 5.3]{Li11}. Selon la décomposition $G_\delta = U \times \Sp(W_+) \times \Sp(W_-)$ (avec $U := U' \times U''$), on décompose $T = T_U \times T_+ \times T_-$ et
$$ \kappa_T = (\kappa_{T, U}, \kappa_{T,+}, \kappa_{T,-}). $$

Enregistrons maintenant une comparaison entre $\kappa$ et $\kappa_T$.

\begin{lemma}\label{prop:kappa-T}
  Sous l'injection $\mathfrak{R}(G_\delta, G; F) \hookrightarrow \mathfrak{R}(T, G; F)$ duale à $H^1(F, T) \twoheadrightarrow H^1_\text{ab}(F, G_\delta)$, on a l'identification
  $$ \kappa = \kappa_{T,U}. $$
\end{lemma}
\begin{proof}
  On note d'abord que la définition de $\kappa_T$ est un cas spécial de notre construction ci-dessus. Les composantes $\kappa_{T, \pm}$ s'envoient sur des caractères de $H^1_\text{ab}(F, \Sp(W_\pm)) = \{1\}$, donc il suffit de comparer $\kappa$ et $\kappa_{T,U}$. Pour cela, on réalise $T$ en termes de groupes unitaires (cf. la description des commutants dans \S\ref{sec:parametrage}), puis on décrit la surjection $H^1(F, T) \to H^1_\text{ab}(F, U)$ au moyen du Lemme \ref{prop:coho-U-mult}.
\end{proof}

\begin{definition}\label{def:paire-nr}
  Supposons que $F$ est non-archimédien de caractéristique résiduelle $p > 2$, et plaçons-nous dans la situation non-ramifiée mentionnée dans le \S\ref{sec:groupes}. On a donc un $\mathfrak{o}_F$-réseau $L$ dans $W$, et $K := \Stab_{G(F)}(L)$ le sous-groupe hyperspécial correspondant qui se plonge dans $\tilde{G}$. On dit qu'une paire $(\gamma, \delta)$ en correspondance équi-singulière est non-ramifiée si
  \begin{itemize}
    \item $\delta \in K$;
    \item le paramètre de $\delta$ est non-ramifié au sens que toutes les données $(W_\pm, \angles{\cdot|\cdot})$, $(K/K^\sharp, x)$, $(W_K, h_K)$ admettent des $\mathfrak{o}_F$-modèles compatibles avec celui de $(W, \angles{\cdot|\cdot})$ fourni par $L$, tels que ces données ont bonne réduction modulo l'idéal maximal $\mathfrak{p}_F$ de $\mathfrak{o}_F$.
  \end{itemize}
  Nous laissons au lecteur le soin de formuler l'analogue du paramétrage dans \S\ref{sec:parametrage} sur un anneau à valuation discrète, ainsi que la notion de \guillemotleft bonne réduction\guillemotright\, pourvu que le corps résiduel  ait pour caractéristique $p>2$. Tout cela doit être standard.

  Remarquons que la situation non-ramifiée définie dans \cite[\S 7.2]{Li11} en est un cas particulier.
\end{definition}

\begin{theorem}\label{prop:facteur}
  Le facteur de transfert dans la Définition \ref{def:facteur-equi-sing} satisfait aux propriétés suivantes.
  \begin{enumerate}[i)]
    \item $\Delta(\gamma, \noyau\tilde{\delta}) = \noyau \Delta(\gamma, \tilde{\delta})$ pour tout $\noyau \in \bmu_8$.
    \item $\Delta(\gamma, \tilde{\delta})$ ne dépend que de la classe de conjugaison stable de $\gamma$ et de la classe de conjugaison de $\tilde{\delta}$.
    \item Soit $\kappa \in \mathfrak{R}(G_\delta, G; F)$ le caractère endoscopique associé à $(n', n'', \gamma)$. On a la propriété du cocycle
      $$ \Delta(\gamma, x^{-1} \tilde{\delta}x) = \kappa(x) \Delta(\gamma, \tilde{\delta}), \quad x \in H^0(F, G_\delta \to G). $$
    \item Si la paire $(\gamma, \delta)$ est non-ramifiée, on a $\Delta(\gamma, \tilde{\delta}) = 1$.
    \item Supposons que $F$ est un corps de nombres, $\tilde{G} := \Mp(W) \to G(\A)$ le revêtement métaplectique adélique par rapport au caractère $\psi = \prod_v \psi_v : \A/F \to \C^\times$. Soient $\gamma \in H(F)_\mathrm{ss}$, $\delta \in G(F)_\mathrm{ss}$ en correspondance équi-singulière. Regardons $\delta$ comme un élément de $\tilde{G}$ et prenons une image réciproque $(\tilde{\delta}_v) \in \Resprod_v \tilde{G}_v$ de $\delta$. Alors on a la formule du produit
      $$ \prod_v \Delta_v(\gamma_v, \tilde{\delta}_v) = 1. $$
  \end{enumerate}
\end{theorem}

\begin{remark}
  Ce théorème est une variante de \cite[\S 5.6 et \S 6.10]{Ko86}. Voir aussi \cite[Lemma 4.1]{Ar01}.
\end{remark}

\begin{proof}
  La première assertion est évidente.

  L'invariance par conjugaison en $\tilde{\delta}$ dans la deuxième assertion est aussi claire. Pour l'invariance par conjugaison stable en $\gamma$, rappelons d'abord que $H^\gamma = H_\gamma$. On suppose que $\gamma' = x^{-1} \gamma x$ pour un $x \in H^0(F, H_\gamma \to H)$. Il suffit d'étudier l'image de $x$ dans $\mathfrak{D}(H_\gamma, H; F)$, ce que l'on note encore par $x$. Prenons un $F$-tore fondamental $T$ dans $H_\gamma$, alors on peut supposer que $x$ provient de $\mathfrak{D}(T, H; F)$ car $H^1(F,T) \to H^1(F, H_\gamma)$ est surjectif. Donc
  \begin{align*}
    \Delta(\gamma, \tilde{\delta}) & = \lim_{\substack{X, Y \\ Y \in \mathfrak{t}_\text{reg}(F) }} \Delta(\exp(Y)\gamma, \exp(X)\tilde{\delta}) \\
    & = \lim_{\substack{X, Y \\ Y \in \mathfrak{t}_\text{reg}(F) }} \Delta(x^{-1} \exp(Y)\gamma x, \exp(X)\tilde{\delta}) \\
    & = \lim_{\substack{X, Y \\ Y \in \mathfrak{t}_\text{reg}(F) }} \Delta(\exp(x^{-1}Yx) \gamma' , \exp(X)\tilde{\delta}) \\
    & = \Delta(\gamma', \tilde{\delta}),
  \end{align*}
  où on a utilisé la propriété correspondante pour le cas semi-simple régulier.

  La troisième assertion est démontrée de la même manière. On suppose que $\tilde{\delta}' = x^{-1}\tilde{\delta}x$ par un $x \in H^0(F, G_\delta \to G)$ (rappelons la Définition \ref{def:adjoint-stable}). Prenons un $F$-tore fondamental $T$ dans $G_\delta$. C'est loisible de supposer que l'image de $x$ dans $\mathfrak{D}(G_\delta, G; F)$ provient de $\mathfrak{D}(T, G; F)$. Alors
  \begin{align*}
    \Delta(\gamma, \tilde{\delta}) & = \lim_{\substack{X, Y \\ X \in \mathfrak{t}_\text{reg}(F) }} \Delta(\exp(Y)\gamma, \exp(X)\tilde{\delta}) \\
    & = \lim_{\substack{X, Y \\ X \in \mathfrak{t}_\text{reg}(F) }} \kappa_T(x)^{-1} \Delta(\exp(Y)\gamma, x^{-1} \exp(X)\tilde{\delta} x) \\
    & = \kappa_T(x)^{-1} \kappa_{T,-}(x) \lim_{\substack{X, Y \\ Y \in \mathfrak{t}_\text{reg}(F) }} \Delta(\exp(Y)\gamma , \exp(x^{-1}Xx)\tilde{\delta}') \\
    & = \kappa_{T,U}(x)^{-1} \Delta(\gamma, \tilde{\delta}') = \kappa(x)^{-1} \Delta(\gamma, \tilde{\delta}').
  \end{align*}
  Là encore, on a utilisé la propriété correspondante pour le cas semi-simple régulier, le Lemme \ref{prop:kappa-T}, ainsi que le fait non-trivial
  \begin{gather}\label{eqn:exp-adjoint}
    x^{-1} (\exp(X)\tilde{\delta}) x  = \underbrace{\kappa_{T,-}(x)}_{\in \bmu_8} \cdot \exp(x^{-1} X x) \cdot x^{-1}\tilde{\delta}x
  \end{gather}
  pour $X \in \mathfrak{t}_\text{reg}(F) \subset \mathfrak{g}_\delta(F)$ assez proche de $0$. Montrons-le.

  Vu la Définition \ref{def:adjoint-stable} pour $x^{-1} (\exp(X)\tilde{\delta}) x$, il suffit de montrer que les valeurs de $(\Theta^+_\psi - \Theta^-_\psi)/|\Theta^+_\psi - \Theta^-_\psi|$ en $\exp(X)\tilde{\delta}$ et en le côté à droite de \eqref{eqn:exp-adjoint} sont pareilles si $X$ est proche de $0$. Comme d'habitude, on suit la recette de Définition \ref{def:conj-stable-local} pour décomposer
  \begin{gather*}
    (\exp(X_K)\tilde{\delta}_K, \exp(X_+), -\exp(X_-)) \mapsto \exp(X) \tilde {\delta}
  \end{gather*}
  selon $G_\delta = U \times \Sp(W_+) \times \Sp(W_-)$. Il y en a un analogue pour $\exp(x^{-1} X x) \cdot x^{-1}\tilde{\delta}x$. Pour la comparaison des valeurs de $(\Theta^+_\psi - \Theta^-_\psi)/|\Theta^+_\psi - \Theta^-_\psi|$, on observe d'abord que $(\Theta^+_\psi - \Theta^-_\psi)/|\cdots|$ se décompose de la même manière \cite[Proposition 4.25]{Li11}. Les composantes dans $U$ et $\Sp(W_+)$ n'importent pas, car $(\Theta^+_\psi - \Theta^-_\psi)/|\cdots|$ est localement constante hors des valeurs propres $-1$, et on a $(\Theta^+_\psi - \Theta^-_\psi)(\tilde{\delta}_K) = (\Theta^+_\psi - \Theta^-_\psi)(x^{-1} \tilde{\delta}_K x)$ par définition.

  Le signe $\kappa_{T,-}(x)$ intervient dans la comparaison en la composante $\Sp(W_-)$. Pour alléger les notations, supposons que $\tilde{\delta} = -1$ et $X = X_-$, donc $\kappa_{T,-} = \kappa_T$. Il faut montrer que
  $$ \Theta_\psi(\exp(X)) = \kappa_T(x) \Theta_\psi(\exp(x^{-1} X x)). $$
  D'après \cite[Corollaire 4.24]{Li11}, on se ramène à montrer
  $$ \gamma_\psi(q[X]) = \kappa_T(x) \gamma_\psi(q[x^{-1}X x]) $$
  où $q[\cdots]$ est la $F$-forme quadratique de Cayley définie dans \cite[\S 4.3]{Li11} et $\gamma_\psi(\cdot)$ est l'indice de Weil. Or ceci est exactement \cite[Lemme 4.14]{Li11} d'après la définition de $\kappa_T$ et de l'équi-singularité. Cela démontre \eqref{eqn:exp-adjoint}.

  Si la paire $(\gamma, \delta)$ est non-ramifiée, les constantes dans $\Delta(\gamma, \delta)$ provenant de $\Theta_\psi \pm \Theta_\psi$ valent $1$ à cause de \cite[Proposition 4.21]{Li11}. Les autres constantes, y compris des symboles de Hilbert quadratiques, valent aussi $1$ car toutes les manipulations dans \cite[\S\S 7.4-7.6]{Li11} peuvent être faites au niveau de $\mathfrak{o}_F$ pour les paires non-ramifiées.

  Enfin, dans le cas global, la paire $(\gamma, \delta)$ est non-ramifiée en presque toute place finie $v$. En particulier le produit dans la cinquième assertion est effectivement fini. En inspectant la descente semi-simple \cite[\S\S 7.5-7.6]{Li11} en détail, $\Delta(\gamma, \tilde{\delta})$ s'exprime en les constantes désignées par $c_\pm, c_u$ et $d_\pm, d_u$ dans \textit{loc. cit.} Ces constantes s'expriment en termes des constantes de Weil $\gamma_\psi(\cdot)$ et des symboles de Hilbert quadratiques associés aux paramètres de $\gamma$ et $\delta$. Chacun admet une formule du produit. D'où la cinquième assertion.
\end{proof}


\subsection{Intégrales orbitales: l'énoncé du transfert}\label{sec:int-orb}
Dans cette sous-section, $F$ désigne toujours un corps local de caractéristique nulle.

Soient momentanément $G$ un $F$-groupe réductif quelconque, et $\delta \in G(F)_\text{ss}$. On définit le discriminant de Weyl par
$$ D^G(\delta) := \det(1-\Ad(\delta)|\mathfrak{g}/\mathfrak{g}_\delta) \in F^\times. $$

Fixons des mesures de Haar sur $G(F)$ et $G_\delta(F)$. Pour $f \in C^\infty_c(G(F))$, l'intégrale orbitale non-normalisée le long de la classe de conjugaison de $\delta$ est
\begin{gather}\label{eqn:int-orb-O}
  \mathscr{O}^G_{\delta}(f) := \int_{G_\delta(F) \backslash G(F)} f(x^{-1} \delta x) \dd x.
\end{gather}

L'intégrale orbitale normalisée est définie par
\begin{gather}\label{eqn:int-orb-J}
  J_G(\delta, f) := |D^G(\delta)|^{\frac{1}{2}} \mathscr{O}^G_{\delta}(f).
\end{gather}

On définit l'intégrale orbitale stable par
\begin{equation}\label{eqn:int-orb-st-O}\begin{split}
  S\mathscr{O}^G_{\delta}(f) & := \int_{H^0(F, G_\delta \to G)} e(G_{x^{-1}\delta x}) f(x^{-1}\delta x) \dd x \\
  & = \sum_{y \in \mathfrak{D}(G_\delta, G; F)} e(G_{y^{-1}\delta y}) \mathscr{O}^G_{y^{-1}\delta y}(f)
\end{split}\end{equation}
où $e(\cdots)$ est le signe de Kottwitz \eqref{eqn:signe-Kottwitz}. Les mesures sont choisies de la façon suivante. D'abord, pour tout $x \in H^0(F, G_\delta \to G)$, on transfère la mesure de Haar de $G_\delta(F)$ vers $G_{x^{-1} \delta x}(F)$ par torsion intérieure. Ensuite, on rappelle l'application canonique surjective $H^0(F, G_\delta \to G) \to \mathfrak{D}(G_\delta, G; F)$. Pour chaque $y \in \mathfrak{D}(G_\delta, G; F)$, sa fibre est $G_{y^{-1}\delta y}(F) \backslash G(F)$ à laquelle une mesure est déjà prescrite. L'ensemble fini $\mathfrak{D}(G_\delta, G; F)$ est muni de la mesure discrète, d'où la mesure sur $H^0(F, G_\delta \to G)$. Signalons que ceci est compatible avec la recette adélique dans \S\ref{sec:mesure-coho}.

L'intégrale orbitale stable normalisée est définie par
\begin{equation}\label{eqn:int-orb-st-S}\begin{split}
  S_G(\delta, f) & := |D^G(\delta)|^{\frac{1}{2}} S\mathscr{O}^G_{\delta}(f) \\
  & = \sum_{y \in \mathfrak{D}(G_\delta, G; F)} e(G_{y^{-1}\delta y}) J_G(y^{-1}\delta y, f).
\end{split}\end{equation}

Puisque le signe de Kottwitz d'un tore est $1$, on voit que ceux-ci généralisent les intégrales orbitales semi-simples régulières. La même définition s'applique à $H$.

Considérons maintenant le revêtement métaplectique $\rev: \tilde{G} \to G(F)$ à huit feuillets, avec $G = \Sp(W)$, $\dim_F W = 2n$. Soient $\tilde{\delta} \in \tilde{G}$ avec $\delta = \rev(\tilde{\delta}) \in G(F)_\text{ss}$.

Rappelons que tout élément de $\tilde{G}$ est bon (voir la Remarque \ref{rem:bonte}). Comme dans \eqref{eqn:int-orb-O} et \eqref{eqn:int-orb-J}, on définit les intégrales orbitales des fonctions anti-spécifiques
\begin{align}
  \mathscr{O}^{\tilde{G}}_{\tilde{\delta}}(f) & := \int_{G_\delta(F) \backslash G(F)} f(x^{-1} \tilde{\delta} x) \dd x, \\
  J_{\tilde{G}}(\tilde{\delta}, f) & := |D^G(\delta)|^{\frac{1}{2}} \mathscr{O}^{\tilde{G}}_{\tilde{\delta}}(f), \quad f \in C^\infty_{c, \asp}(\tilde{G}),
\end{align}
par rapport aux mesures de Haar choisies sur $G_\delta(F)$ et $G(F)$.

Fixons une donnée endoscopique elliptique $(n',n'')$ de $\tilde{G}$. Le groupe endoscopique est noté par $H = H_{n',n''}$. Énonçons le résultat local principal de cet article. Rappelons d'abord que le sous-ensemble $\Sigma_\text{équi}(H)$ de $\Sigma_\text{ss}(H)$ est défini dans \eqref{eqn:equi-sing-ens}.

\begin{theorem}\label{prop:transfert-equi}
  Soit $\gamma \in \Sigma_\text{équi}(H)$. Pour tout $f \in C^\infty_{c, \asp}(\tilde{G})$, on a
  $$ \sum_{\substack{\delta \in \Gamma_{\mathrm{ss}}(G) \\ \delta \leftrightarrow \gamma}} \Delta(\gamma, \tilde{\delta}) e(G_\delta) J_{\tilde{G}}(\tilde{\delta}, f) = |2|^{-\frac{1}{2}(\dim_F V'_+ + \dim_F V''_+ -2)}_F \cdot S_H(\gamma, f^H), $$
  où
  \begin{itemize}
    \item $\tilde{\delta}$ est un élément quelconque de $\rev^{-1}(\delta)$;
    \item on munit $G(F)$, $H(F)$ et les commutants $H_\gamma$, $G_\delta$ des mesures canoniques locales (voir \S\ref{sec:mesures-locales});
    \item $f^H \in C^\infty_c(H)$ est un transfert de $f$ (voir le Théorème \ref{prop:transfert}) par rapport aux mesures choisies,
    \item $V'_+$ et $V''_+$ sont des $F$-espaces vectoriels qui font partie du paramètre de la classe de conjugaison stable de $\gamma$. Plus concrètement, $\dim_F V'_+ + \dim_F V''_+ - 2$ est la somme des multiplicités des valeurs propres $\pm 1$ de $\delta$, pour $\delta \leftrightarrow \gamma$ quelconque.
  \end{itemize}
\end{theorem}

C'est clair que ce théorème généralise le Théorème \ref{prop:transfert}. Sa démonstration est dégagée aux sections suivantes.

\begin{remark}
  Ce résultat est une variante métaplectique de \cite[Conjecture 5.5]{Ko86}. 
\end{remark}

\subsection{Apéritifs}\label{sec:descente}
Cette sous-section est une préparation pour la preuve du Théorème \ref{prop:transfert-equi}. Conservons donc le formalisme de \S\ref{sec:int-orb}. En particulier, $F$ désigne un corps local de caractéristique nulle.

\subsubsection{Descente des intégrales orbitales}
Rappelons brièvement la descente semi-simple des intégrales orbitales normalisées due à Harish-Chandra. Notre référence est \cite[\S 5.2]{Bo94b}. La seule différence est que nous préférons descendre au voisinage de $1$ du commutant, pas sur son l'algèbre de Lie.

Cette théorie s'applique à tout $F$-groupe réductif $G$. On ne fait pas référence à $\Mp(W)$, $\Sp(W)$ dans les discussions suivantes. On utilise la notion des ouverts complètement invariants dans $G(F)$ \cite[\S 2.1]{Bo94b}, ainsi que son avatar stable \cite[\S 6.1]{Bo94b}; pour le cas $F$ non-archimédien, des variantes sont déjà apparues dans \cite[\S 8.1]{Li11}.

Pour simplifier la vie, on ne traitera que les $F$-groupes réductifs connexes; le cas du revêtement métaplectique $\tilde{G} = \Mp(W)$ est tout à fait analogue, grâce à la Remarque \ref{rem:bonte}.

\begin{definition}\label{def:I}
  Introduisons l'espace des intégrales orbitales normalisées $\mathcal{I}(G)$. C'est un espace vectoriel de fonctions lisses $\Gamma_\text{reg}(G) \to \C$. Supposons choisies des mesures de Haar sur $G(F)$ ainsi que sur tous les $F$-tores maximaux, de façon compatible avec les isomorphismes entre ces tores. Alors il y a une surjection $C^\infty_c(G) \to \mathcal{I}(G)$, donnée par $f \mapsto J_G(\cdot, f)$. De plus, si $F$ est archimédien, alors $\mathcal{I}(G)$ est muni d'une topologie telle que ladite surjection est continue; cet espace est étudié dans \cite{Bo94b} pour tous les groupes réels réductifs dans la classe de Harish-Chandra.

  \textit{Grosso modo}, l'espace $\mathcal{I}(G)$ admet une caractérisation
  \begin{itemize}
    \item par les conditions $\mathrm{I}_1(G)$-$\mathrm{I}_4(G)$ dans \cite[\S 3.2]{Bo94b}, si $F$ est archimédien;
    \item par les développements de Shalika près de chaque point de $G(F)_\text{ss}$, si $F$ est non-archimédien.
  \end{itemize}

  Plus généralement, pour un ouvert complètement invariant $\mathcal{U} \subset G(F)$, on introduit aussi l'espace vectoriel $\mathcal{I}(\mathcal{U})$ qui est un espace de fonctions $\Gamma_\text{reg}(\mathcal{U}) \to \C$, muni de la surjection $C^\infty_c(\mathcal{U}) \to \mathcal{I}(\mathcal{U})$ provenant de $J_G(\cdots)$.

  Il existe aussi d'une variante stable $S\mathcal{I}(G)$, muni d'une surjection $C^\infty_c(G) \to S\mathcal{I}(G)$ donnée par $f \mapsto S_G(\cdot, f)$, ce qui se factorise par une application canonique $\mathcal{I}(G) \to S\mathcal{I}(G)$. Ici encore, on regarde $S\mathcal{I}(G)$ comme un espace de fonctions $\Sigma_\text{reg}(G) \to \C$. De même, cela se généralise en $S\mathcal{I}(\mathcal{U})$ pour tout ouvert stablement complètement invariant $\mathcal{U}$ dans $G(F)$. On renvoie à \cite[\S 1]{Ar96} pour des explications.

  L'un des avantages de ce formalisme est qu'il permet de formuler le transfert $f \mapsto f^H$ et la descente semi-simple $f \mapsto f^\natural$ ci-dessous d'une façon plus canonique.
\end{definition}

\begin{remark}\label{rem:trois-aspects}
	L'espace $\mathcal{I}(G)$ interviendra de trois façons dans cet article.
	\begin{enumerate}[(i)]
		\item Un espace de fonctions $\Gamma_\text{reg}(G) \to \C$, comme indiqué plus haut.
		\item L'espace quotient de $C^\infty_c(G)/\angles{J_G(\delta, \cdot) : \delta \in \Gamma_\text{reg}(G)}^\perp$; c'est isomorphe à $\mathcal{I}(G)$ via $f \mapsto J_G(\cdot, f)$.
		\item La restriction des fonctions dans (i) aux tores maximaux fondamentaux de $G$, désormais désignés par la lettre $T$. Plus précisément, on considère les fonctions $t \mapsto \phi(t)$ avec $t \in T_\text{reg}(F)$; on peut même restreindre sur $T_\text{freg}(F)$.
	\end{enumerate}
	
	Les mêmes descriptions s'appliquent à $S\mathcal{I}(G)$. Expliquons l'application $\mathcal{I}(G) \twoheadrightarrow S\mathcal{I}(G)$.
	\begin{compactitem}
		\item Selon (i), pour tout $\phi: \Gamma_\text{reg}(G) \to \C$ dans $\mathcal{I}(G)$, son image dans $S\mathcal{I}(G)$ est la fonction
		$$  \sigma \mapsto \sum_{\substack{\delta \in \Gamma_\text{reg}(G) \\ \delta \mapsto \sigma}} \phi(\delta), \quad \sigma \in \Sigma_\text{reg}(G) $$
		où $\delta \mapsto \sigma$ signifie que $\delta$ appartient à la classe stable $\sigma$. Ceci est clair d'après la définition de $S_G(\cdots)$.
		\item Selon (ii), $\mathcal{I}(G) \twoheadrightarrow S\mathcal{I}(G)$ n'est que l'application quotient naturelle car
		$$ \angles{J_G(\delta, \cdot) : \delta \in \Gamma_\text{reg}(G)}^\perp \subset \angles{S_G(\delta, \cdot) : \delta \in \Gamma_\text{reg}(G)}^\perp. $$
		\item Selon (iii), on doit envoyer une fonction $\phi: T_\text{freg}(F) \to \C$ (on fixe un tore fondamental $T$) à la fonction
		$$ t \mapsto \sum_g \phi(g^{-1}tg), \quad t \in T_\text{freg}(F) $$
		où $g \in G(\bar{F})$ parcourt des représentants de $H^0(F, T \to G)/G(F)$, d'après (i). Observons que $g^{-1}Tg$ est encore un $F$-tore fondamental et $\Ad(g^{-1}): T \rightiso g^{-1}Tg$ est un $F$-isomorphisme, pour chaque $g$ ci-dessus. Il sera utile de remarquer qu'il y a $|\mathfrak{D}(T,G; F)|$ termes dans cette somme-là.
	\end{compactitem}
\end{remark}

Commençons la descente semi-simple. Soit $\delta \in G(F)_\text{ss}$. On considère des ouverts complètement invariants dans $G$ contenant $\delta$ de la forme
$$ \mathcal{U} = \left\{ x^{-1} \exp(X)\delta x : X \in \mathcal{U}^\flat, x \in G(F) \right\} $$ 
où $\mathcal{U}^\flat$ est ouvert $G^\delta(F)$-invariant dans $\mathfrak{g}_\delta(F)$ qui contient $0$ et est $G$-admissible au sens de \cite[\S 2.1]{Bo94b}. On suppose toujours que $\mathcal{U}^\flat$ est suffisamment petit pour que l'exponentielle soit définie et donne un homéomorphisme
$$ \exp: \mathcal{U}^\flat \rightiso \mathcal{U}^\natural := \exp(\mathcal{U}^\flat). $$
En fait, tout voisinage complètement invariant de $\delta$ dans $G$ contient un tel $\mathcal{U}$. La propriété cruciale est que l'application
\begin{align*}
  \pi: G(F) \times \mathcal{U}^\flat & \longrightarrow \mathcal{U}, \\
  (x, X) & \longmapsto x^{-1} \exp(X)\delta x
\end{align*}
est submersive. Elle envoie $\mathcal{U}^\flat_\text{reg}$ sur $\mathcal{U}_\text{reg}$. D'autre part, on a la projection $p: G(F) \times \mathcal{U}^\flat \to \mathcal{U}^\flat$.

On munit $G(F)$ et $G_\delta(F)$ des mesures de Haar. La descente est effectuée à l'aide du diagramme suivant
$$\xymatrix{
  & C^\infty_c(G(F) \times \mathcal{U}^\flat) \ar@{->>}[ld]_{\pi_*} \ar@{->>}[rd]^{p_*} & \\
  C^\infty_c(\mathcal{U}) & & C^\infty_c(\mathcal{U}^\flat)
}$$
avec
\begin{align*}
  \pi_*(\varphi)(\pi(x, X)) & = |\det(1 - \Ad(\exp(X)\delta) | \mathfrak{g}/\mathfrak{g}_\delta )|^{-1/2}_F \cdot \int_{G_\delta(F)} \varphi(yx, yXy^{-1}) \dd y, \\
  p_*(\varphi)(X) & = \int_{G(F)} \varphi(x, X) \dd x.
\end{align*}

\begin{proposition}\label{prop:descent-int-orb}
  Soient $\varphi \in C^\infty_c(G(F) \times \mathcal{U}^\flat)$, $f := \pi_*(\varphi)$, et $f^\natural := p_*(\varphi) \circ \exp^{-1}$. Alors on a
  $$ J_G(\exp(X)\delta, f) = J_{G_\delta}(\exp(X), f^\natural), \quad X \in \mathcal{U}^\flat_{\mathrm{reg}}. $$
\end{proposition}
\begin{proof}
  C'est une variante de \cite[Lemme 5.2.1]{Bo94b}; on le vérifie par un calcul direct. Remarquons une différence: nous intégrons sur $G_\delta(F)$ au lieu de $G^\delta(F)$.
\end{proof}

Soit $f \in C^\infty_c(\mathcal{U})$, on peut trouver $\varphi$ avec $f = \pi_*(\varphi)$ et on pose $f^\natural = p_*(\varphi) \circ \exp^{-1}$. Alors $f \mapsto f^\natural$ est bien définie en tant qu'une application $\mathcal{I}(\mathcal{U}) \to \mathcal{I}(\mathcal{U}^\natural)$; elle est caractérisée par l'égalité dans la Proposition \ref{prop:descent-int-orb}. On dit que $f$ se descend à $f^\natural$.

De plus, cette application est un isomorphisme si $G^\delta = G_\delta$, au sens topologique pour $F = \R$. On renvoie à \cite[Lemme 5.2.1]{Bo94b} pour les détails.

\begin{corollary}\label{prop:remonter}
  Si $f \in C^\infty_c(\mathcal{U})$ se descend en $f^\natural \in C^\infty_c(\mathcal{U}^\natural)$, alors on a
  $$ J_G(\delta, f) = f^\natural(1). $$
\end{corollary}
\begin{proof}
  L'égalité résulte du même calcul que dans la Proposition \ref{prop:descent-int-orb}; en fait c'est plus simple.
\end{proof}

\begin{remark}\label{rem:descente-stable}
  Il y a une version de la descente pour les intégrales orbitales stables; voir \cite[\S 4.8]{Wa13-1}. C'est beaucoup moins élémentaire car il faut utiliser le le transfert entre les formes intérieures. Nous en utilisons une variante qui se trouve aussi dans \textit{loc. cit.} pour l'essentiel.

  Supposons que $\delta \in G(F)_\text{ss}$ satisfait à $G^\delta = G_\delta$. Fixons un voisinage stablement complètement invariant $\mathcal{U}$ de $\delta$, qui provient d'un voisinage ouvert $\mathcal{U}^\flat$ de $0$ dans $\mathfrak{g}_\delta(F)$. Soit $X \in \mathcal{U}^\flat \cap \mathfrak{g}_{\delta, \text{reg}}(F)$. Prenons un système de représentants $\{\delta'\}$ dans $G(F)$ des classes de conjugaison dans la classe de conjugaison stable de $\delta$. Pour chaque $\delta'$, prenons un système de représentants $\{X'\}$ dans $\mathfrak{g}_{\delta'}(F)$ des classes de conjugaison correspondant à $X$ au sens suivant: choisissons $y \in G(\bar{F})$ tel que $\Ad(y)(\delta) = \delta'$, on demande que $\Ad(y)(X)=X'$ en tant que des classes de conjugaison sur $\bar{F}$. L'hypothèse $G^\delta = G_\delta$ entraîne que les éléments $\{\exp(X')\delta'\}$ forment un système de représentants des classes de conjugaison dans la classe de conjugaison stable de $\exp(X)\delta$. Cf.\ la discussion dans \cite[\S 1.5]{LS2}.

  Pour $f \in C^\infty_c(G(F))$, on a donc
  $$ S_G(\exp(X)\delta, f) = \sum_{\delta'} \sum_{X'} J_G(\exp(X')\delta', f). $$

  Effectuons la descente $f \mapsto f^\natural_{\delta'}$ au voisinage de chaque $\delta'$ en prenant un voisinage complètement invariant $\mathcal{U}_{\delta'}$ de $\delta'$, supposé suffisamment petit et provenant d'un voisinage ouvert admissible $\mathcal{U}^\flat_{\delta'}$ de $0$ dans $\mathfrak{g}_{\delta'}(F)$, de sorte que (cf. \cite[\S 3.1]{Lab99})
  $$ \left[ Y \in \mathcal{U}^\flat_{\delta'}, \; x^{-1} \exp(Y)\delta' x \in \mathcal{U}_{\delta'} \right] \Rightarrow x \in G_{\delta'}(\bar{F}). $$

  Quitte à rétrécir les $\mathcal{U}_{\gamma'}$, on peut supposer que leurs adhérences sont disjointes. Le bilan est
  $$ S_G(\exp(X)\delta, f) = \sum_{\delta'} \sum_{X'} J_{G_{\delta'}}(\exp(X'), f^\natural_{\delta'}) = \sum_{\delta'} S_{G_{\delta'}}(\exp(X'), f^\natural_{\delta'}) $$
  où $X' \in \mathfrak{g}_{\delta', \text{reg}}(F)$ est un élément correspondant à $X$.
\end{remark}

\subsubsection{Réduction de la preuve}
Revenons à la situation métaplectique dans \S\ref{sec:int-orb}. C'est-à-dire, on fixe les objets $\tilde{G}$, $G$, $(n', n'')$, $H$, etc.

On se donne $\gamma = (\gamma', \gamma'') \in H(F)_\text{ss}$ et $\delta \in G(F)_\text{ss}$ qui sont en correspondance équi-singulière. Introduisons les paramètres pour $\gamma$ et $\delta$ comme dans la Définition \ref{def:equi-sing}. On écrit les commutants d'après la Proposition \ref{prop:equi-sing-commutants}
\begin{equation}\label{eqn:descente-commutants}
\begin{split}
  G^\delta = G_\delta & = U_G \times \Sp(W_+) \times \Sp(W_-), \\
  H^\gamma = H_\gamma & = U_H \times \SO(V'_+, q'_+) \times \SO(V''_+, q''_+),
\end{split}
\end{equation}
qui satisfont à
\begin{itemize}
  \item $U_H$, $U_G$ sont des produits directs des groupes unitaires ou $\GL$, à restriction des scalaires près;
  \item $U_H$ est une forme intérieure de $U_G$;
  \item $\dim_F V'_+ = \dim_F W_+ + 1$;
  \item $\dim_F V''_+ = \dim_F W_- + 1$.
\end{itemize}

Si $\delta'$ (resp. $\gamma'$) est stablement conjugué à $\delta$ (resp. $\gamma$), alors les commutants sont reliés par torsion intérieure. Pour prouver le Théorème \ref{prop:transfert-equi}, on peut supposer que $H_\gamma$ est quasi-déployé. En effet, $S_H(\gamma, f^H)$ ne dépend que de la classe de conjugaison stable de $\gamma$. D'après \cite[Lemma 3.3]{Ko82}, il existe un conjugué stable $\gamma'$ de $\gamma$ tel que $H_{\gamma'}$ est quasi-déployé. Donc on peut supposer que
\begin{itemize}
  \item $U_H$ est la forme intérieure quasi-déployée de $U_G$;
  \item $\SO(V'_+, q'_+)$ et $\SO(V''_+, q''_+)$ sont déployés.
\end{itemize}

Rappelons aussi que pour $X \in \mathfrak{g}_{\delta,\text{reg}}(F)$ et $Y \in \mathfrak{h}_{\gamma,\text{reg}}(F)$, on écrit $X = (X_u, X_+, X_-)$ et $Y = (Y_u, Y', Y'')$ selon la description des commutants ci-dessus. Si $X$ et $Y$ sont suffisamment proches de $0$, les éléments $\exp(Y)\gamma$ et $\exp(X)\delta$ se correspondent si et seulement si les composantes $X_u$ et $Y_u$ (resp. $X_+$ et $Y'$, $X_-$ et $Y''$) se correspondent par valeurs propres \cite[Définition 7.3]{Li11}, noté
$$ X \CVP Y.$$ 

On fixe un système des représentants $\gamma' \in H(F)_\text{ss}$ des classes de conjugaison stablement conjuguées à $\gamma$. Pour tout $\gamma'$, on fixe ensuite un voisinage ouvert complètement invariant $\mathcal{U}_{\gamma'}$ de $\gamma'$, suffisamment petit comme dans la Remarque \ref{rem:descente-stable}.

De façon similaire, prenons un système des représentants des $\delta \in \Gamma_\text{ss}(G)$ qui correspondent à $\gamma$; pour chaque $\delta$, fixons une image réciproque $\tilde{\delta} \in \tilde{G}$. Pour tout représentant $\delta$, on en fixe un voisinage ouvert complètement invariant $\mathcal{U}_\delta$ dans $G(F)$. On suppose de plus que
\begin{itemize}
  \item les adhérences des $\mathcal{U}_\delta$ (resp. $\mathcal{U}_{\gamma'}$) sont disjointes;
  \item tout élément de $G(F)_\text{ss}$ correspondant à des éléments dans $\bigcup_{\gamma'} \mathcal{U}_{\gamma'}$ appartient à précisément l'un des voisinages $\mathcal{U}_\delta$;
  \item ces voisinages correspondent à des voisinages ouverts admissibles $\mathcal{U}^\flat_\gamma \subset \mathfrak{h}_\gamma(F)$, $\mathcal{U}^\flat_\delta \in \mathfrak{g}_\delta(F)$.
\end{itemize}

On pose $\tilde{\mathcal{U}}_\delta := \bigcup_{x \in G(F)} x^{-1} (\tilde{\delta} \exp(\mathcal{U}^\flat_\delta)) x$ pour tout $\delta$. Fixons de tels choix pour descendre les intégrales orbitales. Pour $f \in C^\infty_{c,\asp}(\tilde{G})$, prenons un transfert $f^H \in C^\infty_c(H)$. Le Théorème \ref{prop:transfert} donne
\begin{gather}\label{eqn:descente-0}
  \sum_{\delta} \sum_{\substack{X \in \mathfrak{g}_{\delta,\text{reg}}(F)/\text{conj} \\ X \CVP Y}} \Delta(\exp(Y)\gamma, \exp(X)\tilde{\delta}) J_{\tilde{G}}(\exp(X)\tilde{\delta}, f) = S_H(\exp(Y)\gamma, f^H).
\end{gather}
pour $Y \in \mathcal{U}^\flat_\gamma$ régulier semi-simple. Pour définir les intégrales orbitales régulières, $G_{\exp(X)\delta}(F)$ et $H_{\exp(Y)\gamma}(F)$ sont munis des mesures de Haar compatibles avec $G_{\exp(X)\delta} \simeq H_{\exp(Y)\gamma}$. Cette formule repose sur la propriété suivante. Pour tout représentant $\delta \in G(F)$ choisi précédemment, on choisit des représentants $X \in \mathfrak{g}_{\delta,\text{reg}}(F)$ qui décrivent la somme ci-dessus. Alors les $\exp(X)\delta$ sont des représentants des classes de conjugaison dans une classe de conjugaison stable. Ceci est justifié par le fait que $G_\delta = G^\delta$; remarquons que le cas général est plus compliqué.

Soit $T$ un $F$-tore maximal de $H_\gamma$. Il se décompose en $T = T_U \times T' \times T''$ selon la décomposition de $H_\gamma$ en $U_H$, $\SO(V'_+, q'_+)$ et $\SO(V''_+, q''_+)$. Rappelons la notion des $F$-tores maximaux fondamentaux dans la Définition \ref{def:fondamental}. Si $T$ est fondamental dans $H_\gamma$, c'est connu que $T$ se transfère vers tous les $H_{\gamma'}$. Cependant, il faut aussi les transférer vers les $G_\delta$.

\begin{lemma}\label{prop:plongements}
  Soit $T$ un $F$-tore maximal fondamental de $H_\gamma$, alors pour tout $\delta$ il existe un plongement $T \hookrightarrow G_\delta$, tel que $T$ est encore fondamental dans $G_\delta$ et, si $Y \in \mathfrak{t}_{\mathrm{reg}}(F)$, alors son image $X$ appartient à $\mathfrak{g}_{\delta,\mathrm{reg}}(F)$ et $X \CVP Y$.

  Réciproquement, soient $Y \in \mathfrak{t}_{\mathrm{reg}}(F)$, $Y' \in \mathfrak{h}_{\gamma', \mathrm{reg}}(F)$ correspondant à $Y$, et $X \in \mathfrak{g}_{\delta,\mathrm{reg}}(F)$ est tel que $X \CVP Y$, alors il existe des plongements du tore fondamental $H_{\gamma'} \hookleftarrow T \hookrightarrow G_\delta$ comme ci-dessus, caractérisés par $Y' \mapsfrom Y \mapsto X$.
\end{lemma}
\begin{proof}
  Il suffit de plonger $T_U$, $T'$ et $T''$ dans $U_G$, $\Sp(W_+)$ et $\Sp(W_-)$, respectivement. Pour $T'$ et $T''$, il résulte de la bijectivité du transfert des classes stables semi-simples réguliers pour l'endoscopie de $\Mp(W_\pm)$, cf. \cite[Remarque 5.2]{Li11}. Quant à $T_U$, il résulte du fait qu'un $F$-tore fondamental se transfère vers toute forme intérieure. La préservation de régularité est claire.

  La réciproque découle des isomorphismes $H_{\gamma', Y'} \simeq H_{\gamma, Y} \simeq G_{\delta, X}$.
\end{proof}

On supposera que l'élément $Y$ dans \eqref{eqn:descente-0} appartient à $\mathfrak{t}_\text{reg}(F)$. D'après ledit lemme, on pourra aussi supposer que chaque élément $X$ dans \eqref{eqn:descente-0} est l'image de $Y$ par un plongement de $T$ dans $G_\delta$, pour chaque $\delta$.

Pour $Y \in \mathfrak{t}_\text{reg}(\R)$ assez proche de $0$, la classe de conjugaison stable de $\exp(Y)\gamma$ est contenue dans $\bigcup_{\gamma'} \mathcal{U}_{\gamma'}$. Afin de prouver le Théorème \ref{prop:transfert-equi}, c'est loisible de supposer que
$$ \Supp(f^H) \subset \bigcup_{\gamma'} \mathcal{U}_{\gamma'} . $$

Soit $\gamma'$ stablement conjugué à $\gamma$. On utilise la Remarque \ref{rem:descente-stable} pour écrire le côté à droite de \eqref{eqn:descente-0} comme
\begin{gather}\label{eqn:descente-1}
  \sum_{\gamma'} S_{H_{\gamma'}}(\exp(Y'), f^{H,\natural}_{\gamma'})
\end{gather}
où $Y'$ correspond à $Y$ et $f^H \mapsto f^{H,\natural}_{\gamma'}$ est la descente semi-simple au voisinage $\mathcal{U}_{\gamma'}$ de $\gamma'$. Ceci est justifié par l'équi-singularité qui entraîne $H^\gamma = H_\gamma$.

D'autre part, on a vu dans la Définition \ref{def:facteur-equi-sing} que $\Delta$ est constante au voisinage de $(\gamma, \tilde{\delta})$ pour chaque $\tilde{\delta}$, donc le côté à gauche de \eqref{eqn:descente-0} est
\begin{gather}\label{eqn:descente-2}
  \sum_{\delta} \Delta(\gamma, \tilde{\delta}) \left( \sum_{X \CVP Y} J_{G_\delta}(\exp(X), f^\natural_\delta) \right) = \sum_{\delta} \Delta(\gamma, \tilde{\delta}) S_{G_\delta}(\exp(X), f^\natural_\delta)
\end{gather}
où $f \mapsto f^\natural_\delta$ est la descente semi-simple au voisinage de $\tilde{\delta}$. Rappelons aussi que l'élément $X$ est pris comme l'image de $Y$ sous le plongement $\mathfrak{t} \hookrightarrow \mathfrak{g}_\delta$.

\begin{remark}\label{rem:T-torsion}
  Étant donnés $Y$ et les $X$, $Y'$, on fixe des plongements du $F$-tore maximal fondamental $T$ de $H_\gamma$ comme dans le Lemme \ref{prop:plongements}, et regarde $T$ comme partagé par tous les commutants $G_\delta$ et $H_{\gamma'}$. Les commutants $G_\delta$ sont reliés par torsion intérieure. Notons $Z$ le centre en commun des $G_\delta$. Comme $H^1(F, T/Z) \twoheadrightarrow H^1(F, G_\delta/Z)$, on peut supposer que les torsions intérieures proviennent de $T$ au sens que leurs cocycles proviennent de $T/Z$. \textit{Idem} pour les $H_{\gamma'}$.
\end{remark}

Dans les \S\S\ref{sec:reel}--\ref{sec:nonarch}, nous commencerons la démonstration du Théorème \ref{prop:transfert-equi} pour $(\gamma, \delta)$ à partir de \eqref{eqn:descente-1}=\eqref{eqn:descente-2}. Les deux côtés sont vus comme des fonctions en $\exp(Y) \in T_\text{reg}(F)$, grâce au choix des plongements de $T$. Une mesure de Haar commode sur $T(F)$ sera aussi choisie dans les démonstrations. Remarquons que le revêtement $\rev: \tilde{G} \to G(F)$ a disparu dans cette égalité.

\section{Stabilisation des termes elliptiques}\label{sec:stabilisation}
Dans cette section,
\begin{itemize}
  \item $F$ désigne un corps de nombres;
  \item on fixe un caractère unitaire non-trivial $\psi = \prod_v \psi_v : \A/F \to \C^\times$;
  \item $(W, \angles{\cdot|\cdot})$ est un $F$-espace symplectique $(W, \angles{\cdot|\cdot})$ de dimension $2n$;
  \item $G := \Sp(W)$;
  \item aux données ci-dessus est associé le revêtement métaplectique adélique
    $$\xymatrix{
      1 \ar[r] & \bmu_8 \ar[r] & \tilde{G} \ar[r]^{\rev} & G(\A) \ar[r] & 1 \\
      & & & G(F) \ar[lu]^{\exists !} \ar@{^{(}->}[u] &
    }$$
    dont la fibre au-dessus de $G(F_v)$ pour une place $v$ de $F$ est notée par $\rev: \tilde{G}_v \to G(F_v)$. C'est le revêtement $\Mp(W_v) \to \Sp(W_v)$ associé à $\psi_v$. On a la surjection $\Resprod_v \tilde{G}_v \to \tilde{G}$.
\end{itemize}

\subsection{Pré-stabilisation}
Commençons par fixer les mesures de Haar: tous les groupes adéliques en vue sont munis de la mesure de Tamagawa. Cela fixe les mesures de Haar sur les revêtements par la convention \eqref{eqn:convention-mes} adaptée au cas adélique.

On aura aussi besoin des accessoires cohomologiques ainsi que des mesures canoniques sur ces objets. Pour cela on renvoie à \S\ref{sec:coho-mes}, notamment la Définition \ref{def:glossaire}.

Soit $I$ un $F$-groupe réductif connexe. Lorsque l'on décompose la mesure de Tamagawa $\mu^I_\text{Tama}$ sur $I(\A)$ en un produit tensoriel en des mesures locales, on suppose que la composante en $v$ est la mesure non-ramifiée pour presque toute place $v$. Il y a un sous-groupe distingué $I(\A)^1 \subset I(\A)$ muni d'une mesure canonique déduite de la mesure de Tamagawa sur $I(\A)$, cf. \cite[\S 1.2]{Lab01}; mais nous n'en avons pas besoin. Dans cet article, il suffit de considérer le cas $Z^\circ_I$ anisotrope, alors $I(\A)^1 = I(\A)$ et on désigne le nombre de Tamagawa de $I$ par
\begin{gather}\label{eqn:tau}
  \tau(I) := \mes(I(F) \backslash I(\A)) \quad \text{par rapport à } \mu^I_\text{Tama}.
\end{gather}

Dans ce qui suit, on suppose toujours que la fonction test $f \in C^\infty_{c,\asp}(\tilde{G})$ est factorisable au sens de \eqref{eqn:factorisation-f}, i.e. $f = \prod_v f_v$ en tant qu'une fonction sur $\Resprod_v \tilde{G}_v$.

\begin{definition}
  Soit $\delta \in G(F)_\text{ss}$. L'intégrale orbitale adélique est définie par
  \begin{align*}
    J_{\tilde{G}}(\delta, f) & = \int_{G_\delta(\A) \backslash G(\A)} f(x^{-1}\delta x) \dd x \\
    & = \prod_v \mathscr{O}^{\tilde{G}_v}_{\tilde{\delta}_v}(f_v) = \prod_v J_{\tilde{G}_v}(\tilde{\delta}_v, f_v)
  \end{align*}
  pour $f \in C^\infty_{c,\asp}(\tilde{G})$ factorisable et $(\tilde{\delta}_v)_v \in \Resprod_v \tilde{G}_v$ qui est une image réciproque de $\delta$; la dernière égalité résulte du fait que $\prod_v |D^G(\delta)|_v = 1$. Il faut aussi justifier la finitude du produit $\prod_v (\cdots)$. En effet, un résultat de Kottwitz \cite[Proposition 7.1]{Ko86} affirme que $J_{\tilde{G}_v}(\tilde{\delta}_v, f_v) = \mathscr{O}^{\tilde{G}_v}_{\tilde{\delta}_v}(f_v) = 1$ pour presque toute place $v$, cf. \cite[Corollary 7.3]{Ko86}.

  Voici une généralisation rapide: si $\tilde{\delta}' = x^{-1} \delta x \in \tilde{G}$ pour un $x \in H^0(\A, G_\delta \to G)$ (cf. la Définition \ref{def:conj-stable-global}), alors les mêmes formules permettent de définir $J_{\tilde{G}}(\tilde{\delta}', f)$. En effet, les discriminants de Weyl $D^{G_v}(\cdots)$ ne changent pas, et on a $\mathscr{O}^{\tilde{G}_v}_{\tilde{\delta}_v}(f_v) = \mathscr{O}^{\tilde{G}_v}_{\tilde{\delta}'_v}(f_v)$ pour presque toute $v$ d'après le Lemme \ref{prop:I-G-adelique}.
\end{definition}

Étant donnés $\delta \in G(F)_\text{ss}$ et $x = (x_v)_v \in H^0(\A, G_\delta \to G)$, le schéma en groupes $G_{x^{-1} \delta x}$ est défini sur $\A$. On définit le signe de Kottwitz adélique $e(G_{x^{-1}\delta x})$ comme le produit des signes de Kottwitz locaux \eqref{eqn:signe-Kottwitz}, i.e.
\begin{gather}\label{eqn:e-adelique}
  e(G_{x^{-1}\delta x}) := \prod_v e(G_{x^{-1}_v \delta x_v}).
\end{gather}
C'est effectivement un produit fini. De plus, si $G_{x^{-1} \delta x}$ est défini sur $F$ (tel est le cas si $x^{-1} \delta x \in G(F)$) alors on a la formule du produit \cite[p.297]{Ko83}
\begin{gather}\label{eqn:e-reciprocite}
  e(G_{x^{-1}\delta x}) = 1.
\end{gather}

\begin{definition}
  Introduisons maintenant les $\kappa$-intégrales orbitales. Soient $\delta \in G(F)_\text{ss}$ et $\kappa \in \mathfrak{R}(G_\delta, G; \A)$. Pour toute $f \in C^\infty_{c,\asp}(\tilde{G})$ factorisable, posons
  \begin{equation}\begin{split}\label{eqn:int-orb-kappa}
    J^\kappa_{\tilde{G}}(\delta, f) & := \int_{H^0(\A, G_\delta \to G)} e(G_{x^{-1}\delta x}) \kappa(x) f(x^{-1}\delta x) \dd x \\
    & = \int_{\mathfrak{D}(G_\delta, G; \A)} e(G_{y^{-1}\delta y}) \kappa(y) J_{\tilde{G}}(y^{-1}\delta y, f) \dd y,
  \end{split}\end{equation}
  où $\kappa$ est regardé comme une fonction sur $H^0(\A, G_\delta \to G)$, qui se factorise par $H^0(\A, G_\delta \to G) \to \mathfrak{D}(G_\delta, G; \A)$. Pour le choix des mesures sur $H^0(\A, G_\delta \to G)$ et $\mathfrak{D}(G_\delta, G; \A)$, on renvoie à \S\ref{sec:mesure-coho}.
\end{definition}

La partie elliptique semi-simple de la formule des traces pour $\tilde{G}$ est
\begin{gather}\label{eqn:T_ell}
  T^{\tilde{G}}_{\text{ell}}(f) := \sum_{\delta \in \Gamma_\text{ell}(G)} \tau(G_\delta) J_{\tilde{G}}(\delta, f), \quad f \in C^\infty_{c,\asp}(\tilde{G}) \text{ factorisable}.
\end{gather}
Cf. \cite[Théorème 6.7, part 4]{Li12b}. Rappelons que $G^\delta = G_\delta$ pour tout $\delta \in G(F)_\text{ss}$, et $Z^\circ_{G_\delta}$ est anisotrope si $\delta$ est $F$-elliptique. La somme portant sur $\delta$ est effectivement finie d'une façon uniforme en $\Supp(f)$, ce qu'assure \cite[Lemma 9.1]{Ar86} ou \cite[Proposition 8.2]{Ko86}.

Le résultat suivant est de nature cohomologique. C'est ce que Labesse appelle la pré-stabilisation. Cf. \cite[Theorem V.2.1]{Lab04}. Tout d'abord, observons qu'étant donné $\kappa \in \mathfrak{R}(G_\delta, G; F)$, la fonctionnelle $J^\kappa_{\tilde{G}}(\delta, \cdot)$  ne dépend que de la classe de conjugaison stable de $\delta \in G(F)_\text{ss}$.

\begin{theorem}\label{prop:pre-st}
  On a
  $$ T^{\tilde{G}}_\mathrm{ell}(f) = \sum_{\delta \in \Sigma_\mathrm{ell}(G)} \sum_{\kappa \in \mathfrak{R}(G_\delta, G; F)} J^\kappa_{\tilde{G}}(\delta, f) $$
  pour toute fonction test factorisable $f \in C^\infty_{c,\asp}(\tilde{G})$.
\end{theorem}
\begin{proof}
  Pour commencer, on regroupe la somme \eqref{eqn:T_ell} selon la conjugaison stable, ce qui donne
  $$ T^{\tilde{G}}_{\text{ell}}(f) = \sum_{\delta \in \Sigma_\text{ell}(G)} \sum_{x \in \mathfrak{D}(G_\delta, G; F)} \tau(G_{x^{-1}\delta x}) e(G_{x^{-1} \delta x}) J_{\tilde{G}}(x^{-1} \delta x, f), $$
  ce qui est loisible car $e(G_{x^{-1} \delta x}) = 1$ d'après \eqref{eqn:e-reciprocite}. Il n'y a aucun souci avec la notation $x^{-1} \delta x$ grâce au Lemme \ref{prop:obstruction}.

  Pour $\delta \in G(F)_\text{ss}$ donné, définissons une fonction $\Phi$ sur $H^0(\A, G_\delta \to G)$ par le diagramme commutatif
  $$\xymatrix{
    H^0(\A, G_\delta \to G) \ar[rd]_{x \mapsto x^{-1}\delta x} \ar[rr]^{\Phi} & & \C \\
    & \tilde{G} \ar[ru]_{f} &
  }$$
  alors $\Phi$ est lisse à support compact, car c'est déjà remarqué après la Définition \ref{def:conj-stable-global} que la flèche $x \mapsto x^{-1}\delta x$ est une immersion fermée lisse.

  Le terme indexé par une classe $\delta \in \Sigma_\text{ell}(G)$, identifiée à un représentant dans $G(F)_\text{ss}$, s'écrit donc
  \begin{gather}\label{eqn:int-Phi}
    \sum_{x \in \mathfrak{D}(G_\delta, G; F)} \tau(G_{x^{-1}\delta x}) \int_{G_{x^{-1}\delta x}(\A) \backslash G(\A)} e(G_{x^{-1}\delta x}) \Phi(x g) \dd g
  \end{gather}
  où on utilise l'action à droite de $G(\A)$ sur $H^0(\A, G_\delta \to G)$ ainsi que la compatibilité du Lemme \ref{prop:adjoint-stable-compat}. Il y a aussi des choix de représentants, mais peu importe.

  Ceux qui concernent le revêtement sont tous empaquetés dans la fonction lisse $\Phi$. La formule de sommation à la Poisson de \cite[Théorème 3.9]{Lab01} est donc applicable à \eqref{eqn:int-Phi}. Elle dit que cette expression est égale au produit de
  $$ \tau(G) \cdot |\Coker[H^1_\text{ab}(\A/F, G_\delta) \to H^1_\text{ab}(\A/F, G)]|^{-1} $$
  avec
  $$ \sum_{\kappa \in \mathfrak{R}(G_\delta, G; F)_1} \; \int_{H^0(\A, G_\delta \to G)} e(G_{x^{-1}\delta x}) \kappa(x) \Phi(x) \dd x . $$

  Rappelons que le complexe abélianisé $G_\text{ab}$ de $G$ est trivial. La première expression ci-dessus vaut $1$ car $H^1_\text{ab}(\A/F, G) = \{1\}$ et $\tau(G)=1$ \cite[Theorem 3.5.1]{Weil82}. La deuxième expression est égale à $\sum_{\kappa \in \mathfrak{R}(G_\delta, G; F)_1} J^\kappa_{\tilde{G}}(\delta, f)$.

  D'autre part, $H^0_\text{ab}(\A, G) = \{1\}$, donc $\mathfrak{E}(G_\delta, G; \A/F) = H^0_\text{ab}(\A/F, G_\delta \to G)$, d'où
  $$ \mathfrak{R}(G_\delta, G; F)_1 = \mathfrak{R}(G_\delta, G; F). $$

  Prenons la somme sur tout $\delta \in \Sigma_\text{ell}(G)$. La formule cherchée en découle.
\end{proof}

\subsection{Application du transfert endoscopique}
Pour tout groupe endoscopique elliptique $H$ de $\tilde{G}$, on note
$$ \Sigma_\text{équi,ell}(H) := \Sigma_\text{équi}(H) \cap \Sigma_\text{ell}(H). $$

\begin{lemma}\label{prop:bijection}
  Il existe une bijection canonique entre les ensembles
  $$ \left\{ (\delta, \kappa) : \delta \in \Sigma_{\mathrm{ell}}(G) , \kappa \in \mathfrak{R}(G_\delta, G; F) \right\} $$
  et
  $$ \left\{ (n', n'', \gamma) : (n',n'') \in \mathcal{E}_{\mathrm{ell}}(\tilde{G}), \; \gamma \in \Sigma_{\text{équi,ell}}(H) \right\} $$
  où $\mathcal{E}_{\mathrm{ell}}(\tilde{G})$ signifie l'ensemble des données endoscopiques elliptiques de $\tilde{G}$ et $H = H_{n',n''}$ est le groupe endoscopique associé. Cette bijection est caractérisée par les conditions suivantes
  \begin{enumerate}[i)]
    \item $\gamma$ et $\delta$ sont en correspondance équi-singulière par rapport à $(n',n'')$;
    \item $\kappa: H^0_{\mathrm{ab}}(\A_F/F, G_\delta \to G) \to \C^\times$ est le caractère endoscopique associé au triplet $(n',n'',\gamma)$ au sens de la Définition \ref{def:kappa}.
  \end{enumerate}
\end{lemma}
\begin{proof}
  Étant donné $(\delta, \kappa)$, prenons un paramètre $K/K^\sharp, a, (W_K, h_K), (W_\pm, \angles{\cdot|\cdot}_\pm)$ pour $\delta$. Le commutant de $\delta$ se décompose comme dans la Définition \ref{def:kappa}
  $$ G_\delta = U \times \Sp(W_+) \times \Sp(W_-). $$

  Puisque $\delta$ est $F$-elliptique, on a $U = \prod_{i \in I} U_i$ où chaque $U_i$ est un groupe unitaire par rapport à une extension quadratique $L_i/L^\sharp_i$ de corps, avec $[L^\sharp_i:F] < +\infty$; on regarde $U_i$ comme un $F$-groupe réductif par la restriction des scalaires. D'après l'Exemple \ref{ex:coho-U}, on a
  $$ H^0_\text{ab}(\A/F, G_\delta \to G) = H^1_\text{ab}(\A/F, U) = \prod_{i \in I} \{\pm 1\}. $$

  Par conséquent, on peut écrire $I = I' \sqcup I''$ de sorte que $\kappa: \prod_{i \in I} \{\pm 1\} \to \C^\times$ est de la forme
  $$ (x_i)_{i \in I} \mapsto \prod_{i \in I''} x_i. $$

  La décomposition $U = \prod_{i \in I} U_i$ résulte d'une décomposition similaire de la paire $(K, K^\sharp)$. Inversement, on arrive à une décomposition $(K, K^\sharp) = (K', K'^\sharp) \times (K'', K''^\sharp)$ selon $I = I' \sqcup I''$. De plus, on a aussi $a=(a', -a'')$ et $ W_K = W'_{K'} \times W''_{K''}$. Pour définir la donnée endoscopique $(n', n'')$, on prend
  \begin{align*}
    n' & := \frac{1}{2}(\dim_F W'_{K'} + \dim_F W_+), \\
    n'' & := \frac{1}{2}(\dim_F W''_{K''} + \dim_F W_-).
  \end{align*}

  Afin de construire l'élément $\gamma = (\gamma', \gamma'')$ dont les paramètres sont de la forme $(K'/K'^\sharp, a', \cdots)$ et $(K''/K''^\sharp, a'', \cdots)$, il faut choisir d'abord une forme hermitienne $h_{K'}$ sur $W'_{K'}$. Montrons que l'on peut s'arranger de sorte que $\Tr_{K'/F} \circ h_{K'}$ est une $F$-forme quadratique sur $W'_{K'}$ à noyau anisotrope de dimension $\leq 2$.

  On note $\angles{a_1, \ldots, a_m}$ la $F$-forme quadratique sur $F^m$ donnée par $(x_1, \ldots, x_m) \mapsto \sum_{i=1}^m a_i x^2_i$. Si deux $F$-formes quadratiques $q_1$ et $q_2$ ont la même classe dans le groupe de Witt de $F$, on écrit $q_1 \Wittequiv q_2$.

  On se ramène d'abord au cas où $K'$ est un corps. En effet, décomposons $K' = \prod_{i \in I} K'_i$ et définissons $I^* \subset I$ comme dans \eqref{eqn:decomp-K}. Pour tout $i \notin I^*$ et toute forme hermitienne $h_{K'_i}$ sur $W'_{K'_i}$, on a $\Tr_{K'_i/F} \circ h_{K'_i} \Wittequiv 0$. Supposons donc que $i \in I^*$ et $h_{K'_i}$ est une forme hermitienne sur $W'_{K'_i}$ telle que $\Tr_{K'_i/F} \circ h_{K'_i} \Wittequiv 0$ ou $\angles{a_i, b_i}$ avec $a_i, b_i \in F^\times$. En utilisant le fait que $\angles{a,-a} \Wittequiv 0$ pour tout $a \in F^\times$, on raisonne par récurrence sur $|I^*|$ pour obtenir des coefficients $c_i \in F^\times$ tels que si l'on prend la forme hermitienne
  $$ h_{K'} := \bigoplus_{i \in I} c_i \cdot h_{K'_i} $$
  sur $W'_{K'}$, alors le noyau anisotrope de $\Tr_{K'/F} \circ h_{K'}$ est de dimension $\leq 2$.

  Supposons donc que $K'$ est un corps. La construction dans la démonstration de \cite[Lemme 4.16]{Li11} fournit une telle forme hermitienne lorsque $W'_{K'} \simeq K'$; en effet, sa composition avec $\Tr_{K'/F}$ est désignée par le symbole $q[a]$ dans \textit{loc. cit.} Choisissons une telle forme hermitienne $h'$ sur $K'$. En général, on écrit $W'_{K'} = (K')^m$. En raisonnant par récurrence sur $m$, l'argument précédent dit qu'il existe des coefficients $\{c_j\}_{j=1}^m$ dans $F^\times$ de sorte que si l'on prend la forme hermitienne $h_{K'} := \bigoplus_{j=1}^m c_j \cdot h'$ sur $W'_{K'}$, alors le noyau anisotrope $\Tr_{K'/F} \circ h_{K'}$ est de dimension $\leq 2$.

  Étant choisi $h_{K'}$, on peut toujours choisir un $F$-espace quadratique $(V'_+, q'_+)$ de dimension impaire telle que
  $$ (W'_{K'}, \Tr_{K'/F} \circ h_{K'}) \oplus (V'_+, q'_+) \simeq (V_{n'}, t \cdot q_{n'}) \Wittequiv \angles{t}, \quad t \in F^\times $$
  où $(V'_{n'}, q_{n'})$ est la $F$-forme quadratique définissant $\SO(2n'+1)$. Quitter à modifier $h'_{K'}$ et $q'_+$ par une homothétie, on peut supposer $t=1$. Alors les données $(K'/K'^\sharp, a', (W'_{K'}, h_{K'}))$ et $(V'_+, q'_+)$ paramètrent une classe semi-simple $\gamma'$ dans $\SO(2n'+1)$.

  De même, il existe des données $(K''/K''^\sharp, a'', (W''_{K''}, h_{K''}))$ et $(V''_+, q''_+)$ paramétrant $\gamma'' \in \SO(2n''+1)$.

  Prenons $\gamma := (\gamma', \gamma'') \in \Sigma_\text{ss}(H)$ où $H = H_{n',n''}$. On voit que $\gamma$ est $F$-elliptique et $\gamma \leftrightarrow \delta$ est équi-singulier, et $\kappa$ est le caractère endoscopique associé à $(n',n'',\gamma)$. En réfléchissant à la construction, on voit que ces conditions caractérisent $(n',n'',\gamma)$.

  La réciproque est démontrée en renversant les arguments. En fait c'est plus simple. Par exemple, l'existence de $\delta \in G(F)_\text{ss}$ correspondant à $\gamma$ résulte du fait que $G = G_\text{SC}$ est quasi-déployé.
\end{proof}

\begin{remark}
  Ce théorème est une variante de \cite[9.7 Lemma]{Ko86} et \cite[Proposition IV.3.4]{Lab04}. Notre énoncé est plus simple car les automorphismes des données endoscopiques n'y interviennent pas. Le lecteur est invité à réfléchir au rôle de l'asymétrie entre $n'$ et $n''$ ici.
\end{remark}

\begin{definition}
  Soit $H$ un groupe endoscopique de $\tilde{G}$. On pose
  $$ ST^H_\text{équi,ell}(f^H) = \tau(H) \sum_{\gamma \in \Sigma_\text{équi,ell}(H)} S_H(\gamma, f^H), \quad f^H \in C^\infty_c(H(\A)), $$
  où $S_H(\gamma, f^H) := \prod_v S_{H_v}(\gamma, f^H_v)$ est l'intégrale orbitale stable adélique, pour $f^H = \prod_v f^H_v$. Rappelons que l'on a mis la mesure de Tamagawa partout.
\end{definition}

Comme la version non-stable, le produit définissant $S_H(\gamma, f^H)$ est effectivement fini d'après \cite[Corollary 7.3]{Ko86}. De même, la somme définissant $ST^H_\text{équi}(f^H)$ l'est aussi d'après \cite[Lemma 9.1]{Ar86}.

\begin{remark}
  Cette définition coïncide avec celles dans \cite[V.4]{Lab04} et \cite[\S 9.2]{Ko86}. En effet, on a vu que $H^\gamma = H_\gamma$ pour tout $\gamma \in \Sigma_\text{équi}(H)$.
\end{remark}

\begin{definition}
  Soit $(n',n'')$ une donnée endoscopique elliptique de $\tilde{G}$ avec groupe endoscopique $H = H_{n',n''}$. On pose
  $$ \iota(\tilde{G}, H) := \tau(G)/\tau(H). $$

  Plus précisément,
  \begin{gather}\label{eqn:coeff}
    \iota(\tilde{G}, H) = \tau(H)^{-1} = \begin{cases}
       \frac{1}{4}, & \text{ si } n',n'' \geq 1, \\
       \frac{1}{2}, & \text{ sinon,} 
    \end{cases}
  \end{gather}
  sauf dans le cas trivial $n=0$ où on a bien sûr $\iota(\tilde{G},H)=1$. En effet, cela découle des faits $\tau(G)=1$ et $\tau(\SO(2k+1))=2$ si $k \geq 1$ \cite[Theorem 3.5.1, 4.5.1]{Weil82}.
\end{definition}

\begin{theorem}\label{prop:FT-stable}
  Soit $f = \prod_v f_v \in C^\infty_{c,\asp}(\tilde{G})$ dont un transfert adélique $f^H = \prod_v f^H_v \in C^\infty_c(H(\A))$ est choisi pour toute donnée endoscopique elliptique $(n',n'')$ avec groupe endoscopique $H = H_{n',n''}$. Alors on a
  $$ T^{\tilde{G}}_{\mathrm{ell}}(f) = \sum_{\substack{(n',n'') \in \mathcal{E}_{\mathrm{ell}}(\tilde{G}) \\ H := H_{n',n''}}} \iota(\tilde{G}, H) ST^H_\text{équi,ell}(f^H). $$
\end{theorem}
\begin{proof}
  Le point de départ est le Théorème \ref{prop:pre-st} qui dit que
  $$ T^{\tilde{G}}_{\mathrm{ell}}(f) = \sum_{(\delta, \kappa)} J^\kappa_{\tilde{G}}(\delta, f) $$
  où les paires $(\delta, \kappa)$ sont comme dans le Lemme \ref{prop:bijection}.

  Notons $(n',n'',\gamma)$ le triplet correspondant à $(\delta, \kappa)$ d'après le Lemme \ref{prop:bijection} et posons $H := H_{n',n''}$. Il faut montrer que
  \begin{gather}\label{eqn:desideratum}
    J^\kappa_{\tilde{G}}(\delta, f) = S_H(\gamma, f^H).
  \end{gather}

  Avant d'attaquer cette égalité, montrons que les mesures de Tamagawa utilisées pour définir $J^\kappa_{\tilde{G}}(\delta, f)$ et $S_H(\gamma, f^H)$ peuvent être remplacées par les mesures canoniques globales de Gross \eqref{eqn:Gross-global}. En effet, on a vu dans la Proposition \ref{prop:equi-sing-commutants} qu'il y a des décompositions de la forme
  \begin{align*}
    G_\delta & = U_G \times \Sp(2a) \times \Sp(2b), \\
    H_\gamma & = U_H \times \SO(2a+1) \times \SO(2b+1),
  \end{align*}
  telles que $U_H$ est une forme intérieure de $U_G$. Les résultats dans \S\ref{sec:motifs}, notamment l'Exemple \ref{ex:motif}, montrent qu'à $G_\delta$ et $H_\gamma$ est associé le même motif d'Artin-Tate $M$. Pour $G_\delta(\A)$ et $H_\gamma(\A)$, la mesure de Tamagawa et celle canonique de Gross ne diffèrent que par le facteur $\varepsilon(M)$, d'après la Proposition \ref{prop:epsilon-M}. Cela justifie le changement de mesure sur $G_\delta(\A)$ et $H_\gamma(\A)$. Il en est de même pour les formes intérieures de $G_\delta$, $H_\gamma$, ainsi que les groupes $G$, $H$ eux-mêmes. C'est donc loisible d'établir \eqref{eqn:desideratum} avec les mesures canoniques globales.

  On a remarqué dans la Définition \ref{def:kappa} que $\kappa$ se localise en les caractères endoscopiques locaux $(\kappa_v)_v$ au moyen d'un homomorphisme $\mathfrak{R}(G_\delta, G; F) \to \mathfrak{R}(G_\delta, G; \A)$. Prenons une image réciproque $(\tilde{\delta}_v)_v \in \Resprod_v \tilde{G}_v$ de $\delta$. En se rappelant \eqref{eqn:int-orb-kappa}, on a
  $$ J^\kappa_{\tilde{G}}(\delta, f) = \prod_v  \left( \; \int_{H^0(F_v, G_\delta \to G)} e(G_{x^{-1}_v \delta x_v}) \kappa_v(x_v) f_v(x^{-1}_v \tilde{\delta}_v x_v) \dd x_v \right) . $$

  Vu les propriétés 3-5 du Théorème \ref{prop:facteur} pour les facteurs de transfert, $J^\kappa_{\tilde{G}}(\delta, f)$ s'écrit comme le produit sur toute place $v$ des expressions
  \begin{gather}\label{eqn:int-endoscopique}
    \sum_{\substack{\eta_v \in \Gamma_{\mathrm{ss}}(G_v) \\ \eta_v \leftrightarrow \gamma}} \Delta(\gamma, \tilde{\eta}_v) e(G_{\eta_v}) J_{\tilde{G}_v}(\tilde{\eta}_v, f_v).
  \end{gather}
  C'est sous-entendu que cette expression vaut $1$ pour presque tout $v$, ce qui découle toujours des résultats de Kottwitz dans \cite[\S 7]{Ko86}.

  On peut supposer que les mesures locales du côté de $G$ comme du côté de $H$ sont toutes les mesures canoniques locales de Gross définies dans \S\ref{sec:mesures-locales}; vu \eqref{eqn:Gross-global}, cela est justifié par l'identification du motif comme précédemment. Appliquons maintenant le Théorème \ref{prop:transfert-equi} en chaque place $v$. Il dit que le produit sur tout $v$ de l'expression \eqref{eqn:int-endoscopique} est égal à
  $$ \prod_v S_{H_v}(\gamma, f^H_v) = S_H(\gamma, f^H) $$
  avec les mesures canoniques locales de Gross, car
  $$ \prod_v |2|^{-\frac{1}{2}(\dim_F V'_+ + \dim_F V''_+ -2)}_{F_v} = 1 $$
  où $V'_+$ et $V''_+$ sont les $F$-espaces dans le paramètre pour la classe stable de $\gamma$, qui sont communs à toutes les places de $F$. On a obtenu \eqref{eqn:desideratum}.

  En insérant le facteur $\tau(H) \tau(H)^{-1}$ dans \eqref{eqn:desideratum} puis en prenant la somme sur tout $(n',n'',\gamma)$ à l'aide du Lemme \ref{prop:bijection}, on arrive à $\sum_{n'+n'' = n} \iota(\tilde{G}, H) ST^H_\text{équi,ell}(f^H)$, ce qu'il fallait démontrer.
\end{proof}

\section{Transfert équi-singulier réel}\label{sec:reel}
L'objectif des \S\S \ref{sec:reel}--\ref{sec:nonarch} est de prouver le Théorème \ref{prop:transfert-equi} à partir de l'égalité \eqref{eqn:descente-1}=\eqref{eqn:descente-2}. Commençons par le cas $F=\R$.

\subsection{Mesures chez Harish-Chandra}\label{sec:mesure-HC}
Rappelons le formalisme de base pour les groupes de Lie réductifs comme exposé dans \cite{HC75}. Notre but est de récapituler la recette de Harish-Chandra pour définir des mesures invariantes sur les orbites $T \backslash G$ qui apparaissent dans la formule de limite (Théorème \ref{prop:limite}). Par conséquent, nous utiliserons le langage des groupes de Lie réductifs dans cette sous-section.

Pour les généralités sur les groupes dans la classe de Harish-Chandra et les involutions de Cartan, on renvoie à \cite[IV.3]{KV95}.

Dans cette sous-section, les symboles $\alpha$, etc., désignent toujours la version sur l'algèbre de Lie des racines. Les racines au niveau du groupe sont désignées par $\xi_\alpha$, etc.

Soient $G$ un groupe réductif de Lie dans la classe de Harish-Chandra, eg. le groupe des $\R$-points d'un $\R$-groupe réductif connexe. On note $\mathfrak{g}$ son algèbre de Lie. On fixe une involution de Cartan globale
$$ \theta: G \to G $$
qui induit une involution $\mathfrak{g} \to \mathfrak{g}$ ce que l'on note encore par $\theta$. Alors il y a une décomposition canonique, avec les notations standard
$$ \mathfrak{g} = \underbrace{\mathfrak{k}}_{\theta = \identity} \oplus \underbrace{\mathfrak{p}}_{\theta = -\identity}. $$
Si $K$ est le sous-groupe compact maximal correspondant à $\theta$, alors $\mathfrak{k}$ est l'algèbre de Lie de $K$. Remarquons que de telles paires $(\mathfrak{g}, \theta)$ forment un cas spécial du formalisme de la Définition \ref{def:alg-Lie-sym-ortho} (voir plus loin).

On fixe une $\R$-forme bilinéaire symétrique $B: \mathfrak{g} \times \mathfrak{g} \to \R$ telle que
\begin{itemize}
  \item $B$ est invariante au sens que $B([X,Y],Z) + B(Y, [X,Z])=0$ pour tout $X,Y,Z$;
  \item $B_\theta(X,Y) := -B(X, \theta Y)$ est positive définie sur $\mathfrak{g} \times \mathfrak{g}$;
  \item $B(\theta X, \theta Y) = B(X,Y)$ pour tout $X,Y$.
\end{itemize}
De telles formes symétriques existent. Par exemple, si $\mathfrak{g}$ est semi-simple alors on peut prendre la forme de Killing. Fixons un tel quadruple $(G,K,\theta,B)$.

Soit $Q$ un sous-groupe parabolique $\theta$-stable. On a la décomposition de Langlands $Q = M A U$, où $U \subset G$ est le sous-groupe de Lie analytique correspondant au radical nilpotent de $\mathfrak{q}$. On définit $\rho_Q \in \mathfrak{a}^*$ par
$$ \angles{\rho_Q, H} = \frac{1}{2} \Tr(\ad(H)|\mathfrak{u}), \quad H \in \mathfrak{a}. $$

Supposons que $Q = P_0$ est un sous-groupe parabolique minimal $\theta$-stable. Alors on peut normaliser la mesure sur $G$ à l'aide de la décomposition d'Iwasawa $G=KAU$, disons
$$ \dd x = e^{2\angles{\rho_{P_0}, \log a}} \dd k \dd a \dd u, \quad x = kau, $$
où on met la mesure de masse totale $1$ sur $K$, et les mesures sur $A$ et $U$ sont induites de la forme positive définie $B_\theta$ restreinte sur $\mathfrak{a}$ et $\mathfrak{u}$, respectivement.

Observons aussi que $\mathfrak{a} \oplus \mathfrak{u} = \mathfrak{p} \subset \mathfrak{g}$.

\begin{remark}
  Autrement dit, la mesure sur $G$ est normalisée en utilisant le fibré $G \twoheadrightarrow K \backslash G$, dont chaque fibre $\simeq K$ est munie de la mesure de Haar de masse totale $1$. La base $K \backslash G \simeq AU$, un espace symétrique riemannien de type non-compact, est munie de la mesure invariante à droite provenant de $B_\theta$ restreinte sur $\mathfrak{p}$. En termes de formes différentielles, la mesure invariante à droite sur $AU$ provient d'une densité invariante dont la valeur en $1$ est
  $$ |v^*_1 \wedge \cdots \wedge v^*_r| $$
  où $v^*_1, \ldots, v^*_r$ est la base duale à une base orthonormée $v_1, \ldots, v_r$ de $\mathfrak{p}$ pour $B_\theta$.
\end{remark}

Pour $G$, les intégrales orbitales semi-simples régulières sont définies comme dans \S\ref{sec:int-orb}. Plus précisément, soit $T$ un sous-groupe de Cartan $\theta$-stable de $G$, on décompose $\mathfrak{t} = \mathfrak{t}_I \oplus \mathfrak{t}_R$ (les composantes \guillemotleft imaginaire\guillemotright\,et \guillemotleft réelle\guillemotright\,de $T$) selon $\mathfrak{g} = \mathfrak{k} \oplus \mathfrak{p}$. Cela donne
$$T = T_I T_R$$
où $\mathfrak{t}_R \rightiso T_R$ par l'exponentielle. La mesure de Haar sur $T$ est normalisée en exigeant que
\begin{itemize}
  \item le groupe compact $T_I$ est de masse totale $1$;
  \item la mesure sur $T_R$ est déduite de la forme positive définie $B_\theta$ restreinte sur $\mathfrak{t}_R$.
\end{itemize}

Puisque la mesure de Haar sur $G$ est déjà fixée, on peut bien définir les intégrales orbitales
\begin{align*}
  \mathscr{O}^G_t(f) & := \int_{T \backslash G} f(x^{-1} t x) \dd x = \int_{T_R \backslash G} f(x^{-1} t x) \dd x, \quad f \in C^\infty_c(G), \\
  J_G(t, f) & := |D^G(t)|^{\frac{1}{2}} \mathscr{O}^G_t(f).
\end{align*}
pour $t \in T$ qui est régulier. Le discriminant de Weyl ici est défini comme dans le cas linéaire
$$ D^G(t) = \Ad(1-\Ad(t)|\mathfrak{g}/\mathfrak{t}), \quad t \in T. $$

Fixons toujours un sous-groupe de Cartan $\theta$-stable $T=T_I T_R$. Rappelons que les racines de $(\mathfrak{g}_\C, \mathfrak{t}_\C)$ sont classifiées en trois genres: réelle, imaginaire et complexe; les racines imaginaires se divisent encore en celles compactes et non-compactes. Cf. \cite[pp.248-249]{KV95}. Pour toute racine $\alpha \in \mathfrak{t}^*$, on désigne par $\xi_\alpha: T \to \C^\times$ le caractère de $T$ associé.

Suivant \cite[\S 8 et \S 17]{HC75}, on introduit $M := Z_G(\mathfrak{t}_R)$, qui est un sous-groupe de Levi $\theta$-stable de $G$. On fixe un système de racines positives $\Psi$ pour $(\mathfrak{m}_\C, \mathfrak{t}_\C)$, c'est-à-dire pour les racines imaginaires pour $(\mathfrak{g}_\C, \mathfrak{t}_\C)$ (cf. \cite{KV95}).

\begin{definition}\label{def:Delta}
  On pose
  \begin{align*}
    \Delta_+(t) & := |\det(1 - \Ad(t^{-1}) | \mathfrak{g}/\mathfrak{m})|^{\frac{1}{2}}, \\
    '\Delta_I(t) & := \prod_{\alpha \in \Psi} (1 - \xi_\alpha(t^{-1})), \quad t \in T \\
    \Delta_I(H) & := \prod_{\alpha \in \Psi} \left( e^{\alpha(H)/2} - e^{-\alpha(H)/2} \right), \quad H \in \mathfrak{t}, \\
    \rho_I & := \frac{1}{2} \sum_{\alpha \in \Psi} \alpha \in \mathfrak{t}^*.
  \end{align*}

  Avec le choix de $T$ et $\Psi$, on pose pour tout $t \in T$ régulier dans $G$ (cf. \cite[p.145]{HC75})
  \begin{gather}
    'F_f(t) := \; '\Delta_I(t) \Delta_+(t) \mathscr{O}^G_t(f).
  \end{gather}

  C'est la variante de l'intégrale orbitale que Harish-Chandra utilise dans sa formule de limite.
\end{definition}

\begin{lemma}\label{prop:F-vs-J}
  Posons
  $$ b_\Psi(t) := \prod_{\alpha \in \Psi} \dfrac{1 - \xi_\alpha(t^{-1})}{|1 - \xi_\alpha(t^{-1})|}. $$
  \begin{enumerate}[i)]
    \item On a
    $$ 'F_f(t) = b_\Psi(t) J_G(t,f), \quad t \in T_{\mathrm{reg}} $$
    pour tout $f \in C^\infty_c(G)$.
    \item Soit $t = \exp(H)$ régulier avec $H \in \mathfrak{t}$, alors on a
      $$ 'F_f(t) = e^{-\rho_I(H)} \cdot \dfrac{\Delta_I(H)}{|\Delta_I(H)|} \cdot J_G(t, f). $$
  \end{enumerate}
\end{lemma}

\begin{remark}
  Puisque $e^{\alpha(H)/2} - e^{-\alpha(H)/2}$ est imaginaire si $\alpha$ l'est, le facteur $\Delta_I(H)|\Delta_I(H)|^{-1}$ est toujours une racine quatrième de l'unité.
\end{remark}

\begin{proof}
  Rappelons que $|\xi_\alpha(t)|=1$ pour toute racine imaginaire $\alpha$. Donc
  \begin{align*}
    \prod_{\substack{\alpha \\ \text{non-imaginaire}}} |1-\xi_\alpha(t^{-1})|^{\frac{1}{2}} & = \prod_{\substack{\alpha \\ \text{imaginaire}}} |1-\xi_\alpha(t^{-1})|^{-\frac{1}{2}} \cdot |D^G(t)|^{\frac{1}{2}} \\
    & = \prod_{\alpha \in \Psi} |1-\xi_\alpha(t^{-1})|^{-1} |D^G(t)|^{\frac{1}{2}};
  \end{align*}
  le côté à gauche étant $\Delta_+(t)$, d'où
  $$ '\Delta_I(t) \Delta_+(t) = \prod_{\alpha \in \Psi} \dfrac{1 - \xi_\alpha(t^{-1})}{|1 - \xi_\alpha(t^{-1})|} \cdot |D^G(t)|^{\frac{1}{2}} = b_\Psi(t) |D^G(t)|^{\frac{1}{2}} . $$

  Pour montrer la deuxième assertion, on écrit $1 - \xi_\alpha(t^{-1}) = e^{-\alpha(H)/2} (e^{\alpha(H)/2} - e^{-\alpha(H)/2})$ pour toute racine $\alpha \in \Psi$, et rappelle le fait que $|e^{-\rho_I(H)}|=1$. Il en résulte que
  $$ b_\Psi(\exp H) = e^{-\rho_I(H)} \dfrac{\Delta_I(H)}{|\Delta_I(H)|}. $$
\end{proof}

\subsection{Formule de limite}
Fixons toujours le quadruple $(G,K,\theta,B)$ et un sous-groupe de Cartan $\theta$-stable $T$ comme précédemment. Remarquons que tout sous-groupe de Cartan de $G$ est conjugué à un qui est $\theta$-stable.

On a un homomorphisme canonique $\partial$ d'algèbres de $\text{Sym}(\mathfrak{t}_\C)$ sur l'algèbre des opérateurs différentiels sur $T$.

Étant choisi un système de racines imaginaires positives $\Psi$, on définit un automorphisme de l'algèbre symétrique $\text{Sym}(\mathfrak{t}_\C)$, noté $\varpi \mapsto \varpi'$, qui est déterminé par $H' = H + \rho_I(H)$ pour tout $H \in \mathfrak{t}_\C$. Cet automorphisme est caractérisé par la propriété
\begin{gather}\label{eqn:varpi-prime}
  \partial(\varpi') \circ e^{-\rho_I} = e^{-\rho_I} \circ \partial(\varpi), \quad \varpi \in \text{Sym}(\mathfrak{t}_\C).
\end{gather}

Pour toute forme linéaire $\lambda \in \mathfrak{t}^*_\C$, on définit $H_\lambda \in \mathfrak{t}_\C$ par
\begin{gather}\label{eqn:B-H}
  B(H_\lambda, \cdot) = \lambda \text{ sur } \mathfrak{t}_\C.
\end{gather}
Cette application dépend du choix de $B$. On définit maintenant
\begin{gather}
  \varpi^G  := \prod_{\alpha > 0} H_\alpha \in \text{Sym}(\mathfrak{t}_\C)
\end{gather}
où un système de racines positives pour $(\mathfrak{g}_\C, \mathfrak{t}_\C)$ est choisi de façon compatible avec $\Psi$. La même construction s'applique au quadruple $(K,K,\identity, B|_{\mathfrak{k} \times \mathfrak{k}})$ et $T_I$, qui permet de définir
\begin{gather}
  \varpi^K := \prod_{\substack{\alpha > 0 \\ \text{pour} K}} H_\alpha \in \text{Sym}(\mathfrak{t}_{I,\C})
\end{gather}

Rappelons que le rang $\rank(G/K)$ de l'espace symétrique $G/K$ est la dimension d'une sous-algèbre commutative maximale de $\mathfrak{p}$. Le rang d'un groupe de Lie réductif $G$ est noté $\rank(G)$. On est prêt à énoncer la formule de limite de Harish-Chandra.

\begin{theorem}[Harish-Chandra {\cite[\S 17, Lemma 5 et \S 37]{HC75}}]\label{prop:limite}
  La fonction $t \mapsto \left[ \partial((\varpi^G)') \; 'F_f \right](t)$ se prolonge en une fonction lisse sur $T$. Il existe une constante explicite $c(G)$ dépendant de $B$ telle que si $T$ est maximalement compact, alors
  $$ \left[ \partial((\varpi^G)') \; 'F_f \right](1) = c(G)f(1), \quad f \in C^\infty_c(G). $$

  La constante $c(G)$ est déterminée comme suit. On pose
  \begin{align*}
    \nu & := \dim(G/K) - \rank(G/K), \\
    q & := \frac{1}{2}( \dim(G/K) - \rank(G) + \rank(K)), \\
    W(G,T) & := N_G(T)/T, \\
    \rho^K & := \frac{1}{2} \sum_{\substack{\alpha > 0 \\ \text{racine pour } (\mathfrak{k}_\C, \mathfrak{t}_{I,\C})}} \alpha
  \end{align*}
  Alors
  $$ c(G) = (-1)^q (2\pi)^q 2^{\nu/2} \cdot \varpi^K(\rho^K) \cdot |W(G,T)|. $$
\end{theorem}

On enregistre une égalité qui sera utile plus tard. Soient $G_1, G_2$ des $\R$-groupes réductifs connexes reliés par torsion intérieure. Définissons les entiers $q_1, q_2$ comme précédemment. D'autre part on a aussi les signes de Kottwitz $e(G_1), e(G_2)$ dans \eqref{eqn:signe-Kottwitz}. Alors
\begin{gather}\label{eqn:q-e}
  (-1)^{q_1 - q_2} = e(G_1)e(G_2)^{-1}.
\end{gather}
En effet, on a $(-1)^{q_1 - q_2} e(G_1)^{-1} e(G_2) = (-1)^{\frac{1}{2}(\rank(K_1) - \rank(K_2))}$, où $K_1$ et $K_2$ sont des sous-groupes compacts maximaux de $G_1$ et $G_2$, respectivement. Soit $T$ un tore maximal fondamental de $G_1$ avec la décomposition $T(\R) = T_I T_R$ en les parties compactes et vectorielles, alors $\rank(K_1) = \dim T_I$; il en est de même pour $K_2$. Comme les tores maximaux fondamentaux se transfèrent vers toute forme intérieure, on a donc $\rank(K_1) = \rank(K_2)$ et l'égalité s'en suit.

\begin{remark}\label{rem:positivite-cpt}
  Le nombre $\varpi^K(\rho^K)$ est positif. En fait, dans le cas $K$ connexe, on a la formule explicite suivante due à Harish-Chandra \cite[p.203]{HC75}
  \begin{gather}\label{eqn:HC}
    \mes(T_I)/\mes(K) = (2\pi)^{-\frac{1}{2} (\dim K - \rank K) } \varpi^K(\rho^K)
  \end{gather}
  où les mesures de Haar sur $K$ et $T_I$ sont déterminées par la forme symétrique $B_\theta$. Une autre preuve se trouve dans \cite{Mc80}.
\end{remark}

\begin{remark}\label{rem:L^G}
  Supposons toujours que $T$ est un sous-groupe de Cartan $\theta$-stable et maximalement compact. Vu le Lemme \ref{prop:F-vs-J} et \eqref{eqn:varpi-prime}, on peut réinterpréter le Théorème \ref{prop:limite} comme
  \begin{align*}
    f(1) & = c(G)^{-1} \left[(\varpi^G)' \; 'F_f \right](1) \\
    & = c(G)^{-1} \lim_{t \to 1} \left[ \partial((\varpi^G)') e^{-\rho_I} \frac{\Delta_I}{|\Delta_I|} J_G(\cdot, f) \right](t) \\
    & = c(G)^{-1} \lim_{t \to 1} \left[ \partial(\varpi^G) \frac{\Delta_I}{|\Delta_I|} J_G(\cdot, f) \right](t)
  \end{align*}
  où $\Delta_I$ et $e^{-\rho_I}$ sont vus comme des fonctions en les $t \in T$ réguliers et proches de $1$.

  Introduisons l'espace vectoriel topologique $\mathcal{I}(G)$ dans la Définition \ref{def:I}. Définissons une forme linéaire $\mathfrak{L}^G: \mathcal{I}(G) \to \C$ par
  \begin{gather}\label{eqn:L^G}
    \mathfrak{L}^G: \phi \longmapsto c(G)^{-1} \lim_{\substack{t \in T \\ t \to 1}} \left[ \partial(\varpi^G) \frac{\Delta_I}{|\Delta_I|} \phi \right](t).
  \end{gather}

  Alors on a $\mathfrak{L}^G(J_G(\cdot,f)) = f(1)$. En particulier $\mathfrak{L}^G$ ne dépend pas du choix de $\Psi$. De plus, $\mathfrak{L}^G$ est localisé en $1$ au sens qu'il se factorise à travers de l'espace $\varinjlim_{\mathcal{U}} \mathcal{I}(\mathcal{U})$, où $\mathcal{U}$ parcourt les ouverts complètement $G$-invariants contenant $1$. Indiquons que $\mathfrak{L}^G$ se factorise par l'aspect (iii) de la Remarque \ref{rem:trois-aspects}, car on ne regarde que les tores fondamentaux.

  Supposons maintenant que $G$ est le groupe des $\R$-points d'un $\R$-groupe réductif connexe. On a introduit l'avatar stable $S\mathcal{I}(G)$ de $\mathcal{I}(G)$ dans la Définition \ref{def:I}. C'est vérifié dans \cite[2.22]{Sh83} que la forme linéaire $\mathfrak{L}^G$ est stable au sens qu'elle se factorise en une forme linéaire
  $$ \mathfrak{L}^G: S\mathcal{I}(G) \to \C .$$
  Ici encore, $\mathfrak{L}^G$ se factorise à travers de l'espace $\varinjlim_{\mathcal{U}} S\mathcal{I}(\mathcal{U})$, où $\mathcal{U}$ parcourt les ouverts stablement complètement $G$-invariants contenant $1$.

  Enfin, vu la définition de $\mathfrak{L}^G$, on peut même se limiter aux restrictions des fonctions dans $S\mathcal{I}(\mathcal{U})$ à $T \cap \mathcal{U}_\text{reg}$.
\end{remark}

Supposons maintenant que $G$ admet des représentations de carré intégrable (i.e. les séries discrètes). Ceci équivaut à $\rank(G) = \rank(K)$, ou encore: les sous-groupes de Cartan maximalement compacts sont compacts. Fixons un sous-groupe de Cartan $\theta$-stable maximalement compact $T$. Pour toute représentation $\pi$ de carré intégrable de $G$, on note
\begin{itemize}
  \item $d(\pi)$: le degré formel, ce qui dépend du choix de la mesure de Haar;
  \item $\lambda_\pi$: le caractère infinitésimal de $\pi$, vu comme une $W(\mathfrak{g}_\C, \mathfrak{t}_\C)$-orbite dans $\mathfrak{t}^*_\C$.
\end{itemize}

\begin{theorem}[Harish-Chandra {\cite[p.96]{HC66}}]\label{prop:deg-formel-HC}
  Supposons que $\pi$ est une représentation de carré intégrable de $G$. On a
  $$ d(\pi) = (-1)^q c(G)^{-1} |W(G,T)| \cdot |\varpi^G(\lambda_\pi)|. $$
\end{theorem}

On prend un système de racines positives pour $(\mathfrak{g}_\C, \mathfrak{t}_\C)$ compatible avec $\Psi$. On note alors $\rho^G := \frac{1}{2} \sum_{\alpha > 0} \alpha$. Le résultat suivant sera utilisé pour la comparaison des opérateurs $\mathfrak{L}^G$.

\begin{corollary}\label{prop:deg-formel}
  Soit $\pi_0$ une représentation de carré intégrable de $G$ avec $\lambda_{\pi_0} = \rho^G$ modulo l'action de $W(\mathfrak{g}_\C, \mathfrak{t}_\C)$.

  Alors on a
  $$ c(G)^{-1} \varpi^G(\rho^G) = (-1)^q |W(G,T)|^{-1} d(\pi_0). $$
\end{corollary}
De telles représentations existent. Si $G$ est compact, alors $\pi_0$ n'est autre que la représentation triviale.

\begin{proof} 
  Vu le Théorème \ref{prop:deg-formel-HC}, il suffit de montrer que $\varpi^G(\rho^G) > 0$. C'est bien connu que $\angles{H_\alpha, \rho^G} > 0$ pour toute racine simple $\alpha$, et $\varpi^G(\rho^G) > 0$ s'en déduit.
\end{proof}

\subsection{Preuve du transfert}
\subsubsection{Un calcul}
Considérons pour un temps la situation suivante
\begin{itemize}
  \item $k$ est un corps de caractéristique autre que $2$;
  \item $G=\Sp(W)$ où $(W, \angles{\cdot|\cdot})$ est un $k$-espace symplectique de dimension $2n$;
  \item $H=\SO(V,q)$ où $(V,q)$ est un $k$-espace quadratique de dimension $2n+1$ comme dans \S\ref{sec:donnees-endo}, de sorte que $H$ est déployé.
\end{itemize}

On fixe une base $e_n, \ldots, e_1, e_{-1}, \ldots, e_{-n}$ de $W$ telle que $\angles{e_i|e_j}=0$ et $\angles{e_i|e_{-j}} = \delta_{i,j}$ pour tout $1 \leq i,j \leq n$. Cela fournit un plongement $G \hookrightarrow \GL(2n)$.

De façon analogue, on fixe une base $e_n, \ldots, e_1, e_0, e_{-1}, \ldots, e_n$ de $V$ telle que $q(e_i|e_{-j})=\delta_{i,j}$ pour $-n \leq i,j \leq n$. Cela fournit un plongement $H \hookrightarrow \GL(2n+1)$.

En termes de ces plongements, les éléments diagonaux (resp. triangulaires supérieurs) forment un tore maximal (resp. un sous-groupe de Borel) dans $G$ et $H$. Le tore maximal diagonal pour $G$ est de la forme
$$ T^G := \left\{ \text{diag}(x_n, \ldots, x_1, x^{-1}_1, \ldots, x^{-1}_n) : x_1, \ldots, x_n \in \Gm \right\}, $$
où $\text{diag}(\cdots)$ signifie la matrice diagonale avec les éléments donnés. Le cas pour $H$ est similaire: c'est de la forme
$$ T^H := \left\{ \text{diag}(x_n, \ldots, x_1, 1, x^{-1}_1, \ldots, x^{-1}_n) : x_1, \ldots, x_n \in \Gm \right\}. $$

Pour $T=T^G$ ou $T=T^H$, et $1 \leq i \leq n$, on définit le caractère $\epsilon_i \in X^*(T)$ en envoyant un élément décrit ci-dessus sur $x_i$. On définit le co-caractère $\eta_i \in X_*(T)$ en envoyant $x$ sur l'élément décrit par $x_i = x$ et $x_j = 1$ si $j \neq i$. Ils induisent des éléments dans $\mathfrak{t}^*$ et $\mathfrak{t}$, ce que l'on note toujours par $\epsilon_i$ et $\eta_i$. La base $\eta_1, \ldots, \eta_n$ est duale à $\epsilon_1, \ldots, \epsilon_n$.

Pour tout $k$-espace vectoriel $S$ de dimension finie, on définit une forme symétrique $B_\text{tr}$ sur $\End_k(S)$, donnée par
$$ B_\text{tr}(X,Y) = \frac{1}{2} \Tr(XY). $$

À l'aide des plongements précédents de $G$ et $H$, on obtient des formes symétriques non-dégénérées invariantes sur $\mathfrak{g}$ et $\mathfrak{h}$, ce que l'on note par le même symbole $B_\text{tr}$. On vérifie que
$$ B_\text{tr}(\eta_i, \eta_j) = \delta_{i,j}, \quad 1 \leq i, j \leq n $$
pour $G$ ou $H$.

Ces $k$-groupes sont tous déployés. On décrit les racines positives de $(G, T^G)$ par rapport au sous-groupe de Borel choisi ci-dessus. Elles sont
\begin{gather*}
  \epsilon_i \pm \epsilon_j, \quad 1 \leq i <j \leq n, \\
  2\epsilon_i, \quad 1 \leq i \leq n;
\end{gather*}
tandis que celles de $(H, T^H)$ sont
\begin{gather*}
  \epsilon_i \pm \epsilon_j, \quad 1 \leq i <j \leq n, \\
  \epsilon_i, \quad 1 \leq i \leq n.
\end{gather*}

Les sous-groupes de Borel dans $G$ et $H$ étant fixés, on note la demi-somme des racines positives de $G$ (resp. $H$) par $\rho^G_k$ (resp. $\rho^H_k$).

Pour toute forme symétrique non-dégénérée invariante $B^G$ sur $\mathfrak{g}$, on définit un isomorphisme $\lambda \mapsto H_\lambda$ par $\mathfrak{t}^* \rightiso \mathfrak{t}$ par $B^G(H_\lambda, \cdot) = \lambda(\cdot)$. Pour $B$ donné, on note
$$ \varpi^G_k := \prod_{\substack{\alpha > 0 \\ \text{racine}}} H_\alpha \in \text{Sym}(\mathfrak{t}). $$

Si $B^G=B_\text{tr}$, alors $H_{\epsilon_i} = \eta_i$ pour tout $1 \leq i \leq n$. Dans ce cas-là
$$ \varpi^G_k = \prod_{1 \leq i <j \leq n} \left((\eta_i + \eta_j)(\eta_i - \eta_j)\right) \cdot \prod_{i=1}^n 2\eta_i. $$

De façon similaire, on a la version pour $H$ par rapport à la forme symétrique non-dégénérée invariante $B^H=B_\text{tr}$,
$$ \varpi^H_k = \prod_{1 \leq i <j \leq n} \left((\eta_i + \eta_j)(\eta_i - \eta_j)\right) \cdot \prod_{i=1}^n \eta_i. $$

\begin{lemma}\label{prop:2n}
  On a
  \begin{enumerate}[i)]
    \item pour tout choix de $B^G$, on a $2^{-n} = \varpi^G_k(\rho^H_k) / \varpi^G_k(\rho^G_k)$;
    \item pour $B^G = B_{\mathrm{tr}}$ et $B^H = B_{\mathrm{tr}}$, on a $2^{-2n} = \varpi^H_k(\rho^H_k) / \varpi^G_k(\rho^G_k)$.
  \end{enumerate}
\end{lemma}

Observons que dans la deuxième assertion, on n'exige pas que les sous-groupes de Borel dans $G$ et $H$ sont choisis de façon compatible.

\begin{proof}
  Le quotient dans la première assertion est bien indépendant du choix de $B^G$, car de telles formes sont toutes proportionnelles à $B_\text{tr}$. Utilisons donc le choix $B^G = B_\text{tr}$ et prenons d'ailleurs $B^H = B_\text{tr}$ pour $H$. Selon la description précédente de $\varpi^G$ et $\varpi^H$, on a $\varpi^G = 2^n \varpi^H$. Donc il suffit de prouver la deuxième assertion.

  C'est loisible de supposer $k=\C$. Notons $G^c$ et $H^c$ les formes compactes réelles de $G$ et $H$, respectivement. Alors $G = G^c_\C$, $H = H^c_\C$. On peut aussi supposer que $T$ est la complexifiée d'un $\R$-tore maximal $T^c$ partagé par $G^c$ et $H^c$. On définit $\varpi^{G^c}, \varpi^{H^c} \in \text{Sym}(\mathfrak{t})$ et $\rho^{G^c}, \rho^{H^c} \in \mathfrak{t}^*$ comme dans \S\ref{sec:mesure-HC}. Alors
  $$ \varpi^G_k(\rho^G_k) = \varpi^{G^c}(\rho^{G^c}), \quad \varpi^H_k(\rho^H_k) = \varpi^{H^c}(\rho^{H^c}). $$

  Un calcul par Steinberg (voir \cite[p.95]{Mc80} ou \cite[Appendix]{Ha71}) nous donne
  $$ \varpi^{G^c}(\rho^{G^c}) = 2^{-\frac{1}{2}(\dim G^c - \rank G^c)} \cdot \prod_m m! \cdot \prod_{\alpha > 0} B_\text{tr}(H_\alpha, H_\alpha) $$
  où $m$ parcourt les exposants de $G^c$, et les racines $\alpha$ sont celles complexifiées, i.e. de $(G,T)$. Une formule analogue vaut pour $H^c$. C'est connu que $G^c$ et $H^c$ ont le même nombre de racines et les mêmes exposants
  $$ 1, 3, \ldots, 2n-1 \quad (\text{sans multiplicités}). $$

  Donc
  $$ \dfrac{\varpi^H_k(\rho^H_k)}{\varpi^G_k(\rho^G_k)} = \dfrac{\varpi^{H^c}(\rho^{H^c})}{\varpi^{G^c}(\rho^{G^c})} = \dfrac{ \prod_{\substack{\alpha > 0 \\ \text{pour } H^c}} B_\text{tr}(H_\alpha, H_\alpha) }{ \prod_{\substack{\alpha > 0 \\ \text{pour } G^c}} B_\text{tr}(H_\alpha, H_\alpha) } = 2^{-2n} $$
  d'après la description des racines ci-dessus.
\end{proof}
\begin{remark}
  C'est aussi possible de prouver ces assertions par un calcul non-trivial mais élémentaire. Nous l'omettons.
\end{remark}

\subsubsection{Application de la formule de limite}
Reprenons le formalisme de \S\ref{sec:descente}, notamment l'égalité \eqref{eqn:descente-1}=\eqref{eqn:descente-2}, avec $F=\R$. En particulier, on suppose que $\gamma \in H(\R)_\text{ss}$ est en correspondance équi-singulière avec des éléments de $G(\R)_\text{ss}$.

Soient $\delta \in G(\R)_\text{ss}$ correspondant à $\gamma$ et $\gamma' \in H(\R)_\text{ss}$ un conjugué stable de $\gamma$. On décrit $G_\delta$, $H_{\gamma'}$ par \eqref{eqn:descente-commutants}. On munit $G$, $H$, ainsi que les composantes $U_G$, $U_H$, $\Sp(W_\pm)$, $\SO(V'_+, q'_+)$ et $\SO(V''_+, q''_+)$ des mesures canoniques locales de Gross dans \S\ref{sec:mesures-locales}. Ces mesures sont compatibles avec torsion intérieure d'après la Proposition \ref{prop:mes-can-prop}. La bonne mesure de Haar sur le tore fondamental $T(\R)$ sera explicitée dans \eqref{eqn:bonne-mesure-T}.

Nous nous proposons de comparer les formules de limite (Théorème \ref{prop:limite}) pour $G_\delta$ et $H_{\gamma'}$. Pour ce faire, choisissons un plongement $T \hookrightarrow G_\delta$ et supposons que les données auxiliaires $K$, $\theta$, $B$ comme dans \S\ref{sec:mesure-HC} sont choisies pour $G_\delta$ pour un certain $\delta$ correspondant à $\gamma$, telles que l'on a l'égalité dans l'espace de mesures invariantes $T(\R) \backslash G_\delta(\R)$:
\begin{gather}\label{eqn:mesure-G/T}
  \dfrac{\text{mesure canonique locale de } G_\delta(\R)}{\text{la mesure choisie sur } T(\R)} = \text{la mesure induite par } (G_\delta, K, \theta, B).
\end{gather}
Il faut montrer que ceci peut être simultanément accompli pour tout $\delta$ si c'est accompli pour un. Il faut également comparer les formes linéaires $\mathfrak{L}^{G_\delta}$ pour les divers quadruples $(G_\delta, K, \theta, B)$. \textit{Idem} avec les $H_{\gamma'}$ au lieu de $G_\delta$.

Précisons d'abord une bonne mesure de Haar sur $T(\R)$ qui satisfait à \eqref{eqn:mesure-G/T}, pour un élément $\delta$ (resp. $\gamma'$).

\begin{lemma}\label{prop:v(M_I)}
  Soient $I$ un $\R$-groupe réductif connexe, $\theta: I \to I$ une involution de Cartan, et $T \subset I$ un $\R$-tore maximal fondamental $\theta$-stable. Notons $A_I$ le plus grand sous-tore déployé de $Z^\circ_I$. Alors ils font partie d'un quadruple $(I(\R),K,\theta,B)$ comme dans \S\ref{sec:mesure-HC}, telle que la mesure invariante sur $T(\R) \backslash I(\R)$ prescrite par la recette de Harish-Chandra dans \S\ref{sec:mesure-HC} coïncide avec le quotient
  $$ \frac{ \text{la mesure canonique locale sur } I(\R)} { v(M_{I'}) \cdot (2\pi)^{-\dim T/A_I} \cdot (\text{la mesure canonique locale sur } T(\R))}. $$

  Ici, $M_{I'}$ signifie le motif d'Artin-Tate associé à $I' := I/A_I$ et $v(M_{I'}) := \mes((I')^c(\R))$ par rapport à la mesure canonique locale, où $(I')^c$ désigne la forme compacte de $I'$.

  De plus, si $I/A_I$ est anisotrope, alors toute mesure invariante sur $T(\R) \backslash I(\R)$ prescrite par la recette de Harish-Chandra est de la forme ci-dessus.
\end{lemma}
D'après \cite[(7.4)]{Gr97}, $v(M_{I'})$ est déterminé par $M_{I'}$. Cela justifie la notation.

\begin{proof}
  Vu la Proposition \ref{prop:mes-can-prop}, c'est loisible de supposer que $Z^\circ_I$ est anisotrope. Choisissons un quadruple $(I(\R),K,\theta,B)$ et modifions $B$. Si $B$ est remplacé par $tB$, où $t>0$, alors la mesure invariante sur $T(\R) \backslash I(\R)$ est multipliée par
  $$ t^{\dim(I(\R)/K) - \rank(I(\R)/K)}. $$
  L'exposant est nul si et seulement si l'espace symétrique riemannien $I(\R)/K$ est de type euclidien, ce qui est possible seulement si $I(\R)=K$. Mais dans ce cas-là, les mesures prescrites dans \S\ref{sec:mesure-HC} satisfont à $\mes(I(\R)) = \mes(T(\R)) = 1$. Cela coïncide avec la mesure quotient ci-dessus, car la mesure de $T(\R)$ par rapport à la mesure canonique locale vaut $(2\pi)^{\dim T}$, d'après \cite[(7.4)]{Gr97}.
\end{proof}

\begin{proposition}\label{prop:torsion-int}
  Soient $I_1$, $I_2$ deux $\R$-groupes réductifs connexes reliés par une torsion intérieure $\Psi: I_1 \times_\R \C \rightiso I_2 \times_\R \C$. Supposons que $T$ est un $\R$-tore maximal fondamental partagé par $I_1$ et $I_2$; on peut donc supposer que $\Psi$ respecte $T$ au sens de la Remarque \ref{rem:T-torsion}.

  Alors on peut choisir des données $K_i, \theta_i, B_i$ pour $I_i$ ($i=1, 2$) comme dans \S\ref{sec:mesure-HC}, telles que
  \begin{enumerate}[i)]
    \item $T$ est $\theta_i$-stable pour $i=1,2$ et $\theta_1|_T = \theta_2|_T$;
    \item les mesures invariantes induites sur $T(\R) \backslash I_1(\R)$ et $T(\R) \backslash I_2(\R)$ par $B_1$ et $B_2$ sont compatibles avec la torsion intérieure $\Psi$;
    \item pour tout $\phi \in S\mathcal{I}(I_1) = S\mathcal{I}(I_2)$ en tant qu'une fonction sur $T_{\mathrm{reg}}(\R)$ (rappelons Remarque \ref{rem:trois-aspects} (iii)), on a
      $$ |\mathfrak{D}(T, I_1; \R)|^{-1} e(I_1) \mathfrak{L}^{I_1}(\phi) = |\mathfrak{D}(T, I_2; \R)|^{-1} e(I_2) \mathfrak{L}^{I_2}(\phi) $$
    où les formes linéaires $\mathfrak{L}^{I_i}$ sont définies par rapport au quadruple $(I_i, K_i, \theta_i, B_i)$ pour $i=1,2$.
  \end{enumerate}
\end{proposition}

On observe que l'identification entre $S\mathcal{I}(I_1)$ et $S\mathcal{I}(I_2)$ est possible d'après le transfert entre les formes intérieures en l'endoscopie standard pour les groupes réels.

\begin{proof}
  Les assertions sont implicites dans la preuve du transfert équi-singulier pour l'endoscopie standard \cite[Lemma 2.4A]{LS2}; voir aussi \cite[\S 2.9]{Sh83}. Pour convaincre le lecteur, donnons-en une démonstration au cas où $T$ est anisotrope. Remarquons que cela suffira pour démontrer le Lemme \ref{prop:egalite-L} ci-dessous.

  Les mesures sur $T(\R) \backslash I_1(\R)$ et $T(\R) \backslash I_2(\R)$ données dans le Lemme \ref{prop:v(M_I)} sont compatible avec la torsion intérieure $\Psi$ qui respecte $T$, car les motifs $M_{I_1}$ et $M_{I_2}$ sont pareils. Choisissons donc les données $K_i, \theta_i, B_i$ pour $i=1,2$ comme dans ledit lemme; on peut supposer de plus que $T(\R)$ est stable par $\theta_1$ et $\theta_2$. Les assertions 1 et 2 sont donc satisfaites.


  Montrons la troisième assertion. Prenons les systèmes de racines positives $\Psi_1, \Psi_2$ qui sont compatibles avec $\Psi$. Il faut identifier
  \begin{itemize}
    \item les fonctions $\Delta^{(1)}_I/|\Delta^{(1)}_I|$ et $\Delta^{(2)}_I/|\Delta^{(2)}_I|$ sur $\mathfrak{t}_\text{reg}(\R)$, définies par rapport à $I_1$ et $I_2$, respectivement;
    \item les éléments
      $$ |\mathfrak{D}(T, I_1; \R)|^{-1} e(I_1) c(I_1)^{-1} \varpi^{I_1} $$
      et
      $$ |\mathfrak{D}(T, I_2; \R)|^{-1} e(I_2) c(I_2)^{-1} \varpi^{I_2} $$
      de $\text{Sym}(\mathfrak{t}_\C)$.
  \end{itemize}

  L'égalité $\Delta^{(1)}_I/|\Delta^{(1)}_I| = \Delta^{(2)}_I/|\Delta^{(2)}_I|$ découle immédiatement de la compatibilité entre $\Psi_1$ et $\Psi_2$. Quant aux éléments $|\mathfrak{D}(T, I_1; \R)|^{-1} e(I_1) c(I_1)^{-1} \varpi^{I_1}$ et $|\mathfrak{D}(T, I_2; \R)|^{-1} e(I_2) c(I_2)^{-1} \varpi^{I_2}$, c'est clair qu'ils ne diffèrent que par une constante dans $\C^\times$. Pour l'épingler, on évalue ces deux polynômes en la demi-somme des racines positives $\rho$ en commun. Notons $I_c$ la forme intérieure compacte en commun de $I_1$ et $I_2$. Pour $i=1,2$, on note $\pi^{I_i}_0$ la représentation de carré intégrable de $I_i(\R)$ qui est un transfert de la représentation triviale de $I_c(\R)$. On définit l'entier $q_i$ qui apparaît dans le Théorème \ref{prop:limite} pour $I_i(\R)$. Vu le Corollaire \ref{prop:deg-formel}, on a
  \begin{align*}
    |\mathfrak{D}(T, I_i; \R)|^{-1} e(I_i) c(I_i)^{-1} \varpi^{I_i}(\rho) & = |W(I_i,T)|^{-1} |\mathfrak{D}(T, I_i; \R)|^{-1} (-1)^{q_i} e(I_i) d(\pi^{I_i}_0), \\
    & = |W(I_{i,\C}, T_\C)|^{-1} (-1)^{q_i} e(I_i) d(\pi^{I_i}_0)
  \end{align*}
  où on a utilisé la bijection suivante due à Shelstad \cite[Lemma 2.4.1]{Sh83}
  $$ W(I_i, T) \backslash W(I_{i,\C}, T_\C) \simeq \mathfrak{D}(T, I_i; \R). $$

  Puisque $(-1)^{q_1-q_2} = e(I_1)e(I_2)^{-1}$ d'après \eqref{eqn:q-e}, il reste à montrer que $d(\pi^{I_1}_0) = d(\pi^{I_2}_0)$. Or cette égalité est bien connue: les degrés formels sont invariants par torsion intérieures \cite[p.2.23]{Sh83}.
\end{proof}

Les centres connexes de $H_{\gamma'}$ et $G_\delta$ s'identifient, dont le plus grand sous-tore déployé est désigné par $A$. Les groupes $H_{\gamma'}/A$ et $G_\delta/A$ ont le même motif $M$. Vu lesdits lemmes, munissons $T(\R)$ de la mesure
\begin{gather}\label{eqn:bonne-mesure-T}
  v(M) (2\pi)^{-\dim T/A} \cdot (\text{la mesure canonique locale sur } T(\R)).
\end{gather}
Par conséquent, la condition \eqref{eqn:mesure-G/T} est donc satisfaite pour les groupes $G_\delta$, $H_{\gamma'}$ avec des données $(K,\theta, B)$ convenables, ce que l'on fixe.

\begin{lemma}\label{prop:egalite-L}
  Soit $\delta \in G(\R)_{\mathrm{ss}}$ qui correspond à $\gamma$. Pour $\phi \in S\mathcal{I}(\mathcal{U}_\delta) \cap S\mathcal{I}(\mathcal{U}_\gamma)$, vu comme une fonction sur l'ensemble des éléments de $T_{\mathrm{reg}}(\R)$ proches de $1$, on a
  $$ |\mathfrak{D}(T, G_\delta; \R)|^{-1} e(G_\delta) \mathfrak{L}^{G_\delta}(\phi) = 2^{-\frac{1}{2} (\dim_\R V'_+ + \dim_\R V''_+ - 2)} \cdot |\mathfrak{D}(T, H_\gamma; \R)|^{-1} e(H_\gamma) \mathfrak{L}^{H_\gamma}(\phi). $$
\end{lemma}
\begin{proof}
  On décompose ces formes linéaires selon la décomposition dans \eqref{eqn:descente-commutants}. La Proposition \ref{prop:torsion-int} permet de traiter la comparaison entre les composantes $U_G$ et $U_H$ (plus précisément, il suffit de traiter la torsion intérieure entre des groupes unitaires). De plus, on peut traiter séparément la comparaison entre $\Sp(W_+)$ et $\SO(V'_+, q'_+)$ (resp. $\Sp(W_-)$ et $\SO(V''_-, q''_+)$). Les cas de $\Sp(W_\pm)$ étant en symétrie, on se ramène au cas $\delta=1$, $\gamma=1$. On observe que $e(G)=e(H)=1$ car ils sont quasi-déployés. Posons $n := \frac{1}{2} \dim_\R W$. Reste à montrer que
  $$ |\mathfrak{D}(T, G; \R)|^{-1} \mathfrak{L}^G(\phi) = 2^{-n} \cdot |\mathfrak{D}(T, H; \R)|^{-1} \mathfrak{L}^H(\phi) $$
  pour $\phi \in S\mathcal{I}(\mathcal{U}_\delta) \cap S\mathcal{I}(\mathcal{U}_\gamma)$, où $\mathcal{U}_\delta$ et $\mathcal{U}_\gamma$ sont des voisinages de $1$ stablement complètement invariants et suffisamment petits dans $G(\R)$ et $H(\R)$, respectivement.

  Choisissons un système de racines positives $\Psi^G$ pour $(\mathfrak{g}_\C, \mathfrak{t}_\C)$ (resp. $\Psi^H$ pour $(\mathfrak{h}_\C, \mathfrak{t}_\C)$) de façon standard, c'est-à-dire qu'il existe une base $\epsilon_1, \ldots, \epsilon_n$ de $\mathfrak{t}^*_\C$ telle que les éléments dans $\Psi^G$ sont de la forme $\epsilon_i \pm \epsilon_j$ (pour $1 \leq i < j \leq n$) ou $2\epsilon_i$ (pour $1 \leq i \leq n$). Pour $\Psi^H$, on remplace $2\epsilon_i$ par $\epsilon_i$. Une définition détaillée sera donnée avant le Lemme \ref{prop:2n}. Remarquons que ces racines sont toutes imaginaires puisque $T$ est anisotrope.

  La première observation est que $\Delta^G_I(Y)/|\Delta^G_I(Y)| = \Delta^H_I(Y)/|\Delta^H_I(Y)|$ pour $Y \in \mathfrak{t}_\text{reg}(\R)$ suffisamment proche de $0$, où $\Delta^G_I, \Delta^H_I$ sont les fonctions dans la Définition \ref{def:Delta} pour $G$ et $H$. En effet, vu la description des racines ci-dessus, on a
  $$ \dfrac{\Delta^G_I(Y)}{\Delta^H_I(Y)} = \prod_{i=1}^n \dfrac{e^{\epsilon_i(Y)} - e^{-\epsilon_i(Y)}}{e^{\epsilon_i(Y)/2} - e^{-\epsilon_i(Y)/2}} = \prod_{i=1}^n \left( e^{\epsilon_i(Y)/2} + e^{-\epsilon_i(Y)/2} \right). $$

  Comme $\epsilon_i(H) \in i\R$, ce quotient est réel. Il tend vers $2^n$ lorsque $H$ tend vers $0$, donc $\Delta^G_I(Y)/\Delta^H_I(Y)$ est positif si $Y$ est proche de $0$.

  Il reste à établir l'égalité suivante dans $\text{Sym}(\mathfrak{t}_\C)$
  $$ |\mathfrak{D}(T, G; \R)|^{-1} c(G)^{-1} \varpi^G = |\mathfrak{D}(T, H; \R)|^{-1} 2^{-n} c(H)^{-1} \varpi^H . $$

  Les deux côtés ne peuvent différer que par une constante multiplicative dans $\C^\times$. En effet, les restrictions sur $\mathfrak{t}$ des formes bilinéaires $B^G$ et $B^H$ sont proportionnelles à une forme symétrique invariante standard $B_\text{tr}$ telle que $B_\text{tr}(\eta_i, \eta_j) = \delta_{i,j}$ où $\{\eta_1, \ldots, \eta_n \}$ est la base duale à $\epsilon_1, \ldots, \epsilon_n$. Cf. la preuve du Lemme \ref{prop:2n}.

  Donc il suffit de montrer que
  \begin{gather}\label{eqn:comparaison-L}
    |\mathfrak{D}(T, G; \R)|^{-1} c(G)^{-1}\varpi^G(\rho^H) = |\mathfrak{D}(T, H; \R)|^{-1} 2^{-n} c(H)^{-1} \varpi^H(\rho^H) \neq 0
  \end{gather}
  où $\rho^G$, $\rho^H$ sont les demi-sommes des racines positives.

  Prenons $\pi^H_0$ une représentation de carré intégrable de $H(\R)$, dont le caractère infinitésimal est $\rho^H \in \mathfrak{t}^*_\C/W(H_\C, T_\C)$. Une telle représentation est un transfert de la représentation triviale de la forme intérieure compacte de $H(\R)$. De même, prenons $\pi^G_0$ pour $G(\R)$ de caractère infinitésimal $\rho^G$; c'est une représentation de carré intégrable.

  Rappelons les entiers $q_G$, $q_H$ qui apparaissent dans le Théorème \ref{prop:limite}. D'après le Corollaire \ref{prop:deg-formel}, le côté à droite de \eqref{eqn:comparaison-L} est égal à $ |\mathfrak{D}(T, H; \R)|^{-1} (-1)^{q_H} |W(H,T)|^{-1}$ multiplié par $2^{-n} d(\pi^H_0)$, qui sont tous non nuls. La théorie de Shelstad fournit une bijection entre $W(H, T) \backslash W(H_\C, T_\C)$ et $\mathfrak{D}(T, H; \R)$, donc l'expression ci-dessus s'écrit comme
  $$ |W(H_\C, T_\C)|^{-1} (-1)^{q_H} 2^{-n} d(\pi^H_0). $$

  D'autre part, le côté à gauche de \eqref{eqn:comparaison-L} est
  $$ |\mathfrak{D}(T, G; \R)|^{-1} c(G)^{-1}\varpi^G(\rho^G) \cdot \dfrac{\varpi^G(\rho^H)}{\varpi^G(\rho^G)} = |W(G_\C, T_\C)|^{-1} (-1)^{q_G} d(\pi^G_0) \cdot \dfrac{\varpi^G(\rho^H)}{\varpi^G(\rho^G)}, $$
  toujours d'après le Corollaire \ref{prop:deg-formel} appliqué à $G(\R)$ et la théorie de Shelstad. On vérifie que
  $$ q_G = q_H = \dfrac{n(n+1)}{2} $$
  d'après la table dans \cite[X.6.2]{Hel01}, et $W(G_\C, T_\C) \simeq W(H_\C, T_\C)$ abstraitement. On s'est ramené à montrer que
  $$ 2^{-n} \cdot \dfrac{d(\pi^H_0)}{d(\pi^G_0)} = \dfrac{\varpi^G(\rho^H)}{\varpi^G(\rho^G)} $$
  pourvu que $G(\R)$ (resp. $H(\R)$) soit muni de la mesure déterminées par $B^G$ (resp. $B^H$).

  Pour le côté à droite, on a $2^{-n} = \varpi^G(\rho^H) / \varpi^G(\rho^G)$, ce qu'assure le Lemme \ref{prop:2n} appliqué sur le corps $k=\C$.

  Il reste à établir l'égalité $d(\pi^H_0) = d(\pi^G_0)$. Rappelons que $G$ et $H$ ont le même motif $M$. Vu la normalisation des mesures du Lemme \ref{prop:v(M_I)}, l'équation fonctionnelle locale du Théorème \ref{prop:eq-fonc-locale}, et le fait que $G$ et $H$ ont le $\R$-tore maximal anisotrope $T$ en commun, on peut multiplier les mesures de Haar par la même constante non nulle s'exprimant en $L(M)$, $L(M^\vee(1))$ et $v(M)$, de sorte que les degrés formels sont définis par rapport aux mesures
  $$ |\mathfrak{D}(T, G; \R)|^{-1} (-1)^{q(G)} \mu^G_\text{EP}, \quad |\mathfrak{D}(T, H; \R)|^{-1} (-1)^{q(H)} \mu^H_\text{EP} , $$
  respectivement; on a aussi utilisé les faits simples
  \begin{itemize}
    \item $e(G) = e(H) = 1$;
    \item $q(G) = q(H) = \frac{n(n+1)}{2}$, ce qui est déjà utilisé.
  \end{itemize}

  On conclut que $d(\pi^H_0) = d(\pi^G_0)$ à l'aide du Lemme \ref{prop:deg-EP}.
\end{proof}

\begin{proof}[Démonstration du Théorème \ref{prop:transfert-equi} pour $F=\R$]
  Posons $t := \frac{1}{2} (\dim_\R V'_+ + \dim_\R V''_+ - 2)$.

  On part de l'égalité entre \eqref{eqn:descente-1} et \eqref{eqn:descente-2}, regardées comme des fonctions en $\exp(Y) \in T_\text{reg}(\R)$. On applique à \eqref{eqn:descente-1} la forme linéaire $2^{-t} |\mathfrak{D}(T, H_\gamma; \R)|^{-1} e(H_\gamma) \mathfrak{L}^{H_\gamma}$ de la Remarque \ref{rem:L^G}. On veut traiter tous les $\gamma'$ stablement conjugués à $\gamma$, donc il convient de rappeler la Proposition \ref{prop:torsion-int} pour les $H_{\gamma'}$. D'après la Remarque \ref{rem:L^G}, le Corollaire \ref{prop:remonter} et ladite Proposition, on obtient la somme sur toute classe de conjugaison $\gamma'$ dans la classe stable de $\gamma$, des expressions
  $$ 2^{-t} e(H_{\gamma'}) f^{H,\natural}_{\gamma'}(1) = 2^{-t} e(H_{\gamma'}) J_H(\gamma', f^H). $$
  Autrement dit, on obtient $2^{-t} S_H(\gamma, f^H)$. Signalons que le facteur $|\mathfrak{D}(T, H_\gamma; \R)|^{-1}$ a disparu dans le résultat car on applique $\mathfrak{L}^{H_\gamma}$ à des intégrales orbitales stables: voir Remarque \ref{rem:trois-aspects} (iii).

  Vu le Lemme \ref{prop:egalite-L}, la même forme linéaire appliquée à \eqref{eqn:descente-2} nous donne la somme sur toute classe de conjugaison $\delta$ correspondant à $\gamma$, des expressions
  $$ \Delta(\gamma, \tilde{\delta}) e(G_\delta) f^\natural_\delta(1) = \Delta(\gamma, \tilde{\delta}) e(G_\delta) J_{\tilde{G}}(\tilde{\delta}, f). $$
  On obtient ainsi le côté à gauche de l'égalité du Théorème \ref{prop:transfert-equi}, ce qu'il fallait démontrer.
\end{proof}

\section{Transfert équi-singulier complexe}\label{sec:cplx}
\subsection{Formule de limite pour les groupes complexes}\label{sec:limite-cplx}
Commençons avec un interlude sur la dualité des algèbres de Lie symétriques orthogonales. La référence est \cite[Chapter V]{Hel01}.

\begin{definition}\label{def:alg-Lie-sym-ortho}
  Une algèbre de Lie symétrique orthogonale est une paire $(\mathfrak{g}, \theta)$ où $(\mathfrak{g}, [,])$ est une algèbre de Lie réelle, et $\theta \in \Aut_\R(\mathfrak{g})$ est une involution non-triviale, telles que si l'on écrit
  $$ \mathfrak{g} = \underbrace{\mathfrak{k}}_{\theta=\identity} \oplus \underbrace{\mathfrak{p}}_{\theta=-\identity} $$
  alors $\mathfrak{k}$ est une sous-algèbre de Lie compacte.
\end{definition}

Soit $(\mathfrak{g}, \theta)$ une algèbre de Lie symétrique orthogonale. On plonge $\mathfrak{g}$ dans sa complexifiée $\mathfrak{g}_\C = \mathfrak{g} \oplus i\mathfrak{g}$ en tant qu'algèbre de Lie réelle. L'involution sur $\C$ induite sur $\mathfrak{g}_\C$ est encore notée $\theta$. Posons
$$ \mathfrak{g}' := \mathfrak{k} \oplus i\mathfrak{p}. $$
On vérifie que $\mathfrak{g}'$ est une sous-algèbre de Lie réelle de $\mathfrak{g}_\C$. De plus, $\theta(\mathfrak{g}') = \mathfrak{g}'$.

\begin{definition}
  La paire $(\mathfrak{g}', \theta)$ est dite l'algèbre de Lie symétrique orthogonale duale à $(\mathfrak{g}, \theta)$.
\end{definition}

Évidemment, $\mathfrak{g}'$ et $\mathfrak{g}$ sont des formes réelles dans $\mathfrak{g}_\C$. Elles sont isomorphes sur $\C$ par l'isomorphisme
$$ \Psi: \mathfrak{g} \otimes_\R \C \rightiso \mathfrak{g}_\C \rightiso \mathfrak{g}' \otimes_\R \C. $$

Si $\omega \in \topwedge \mathfrak{g}^*$, alors $\omega' := i^{\dim_\R \mathfrak{p}} \omega$ est définie sur $\R$ par rapport à la forme réelle $\mathfrak{g}'$. Cf. la définition de la mesure canonique locale au cas réel dans \S\ref{sec:mesures-locales}.

Supposons que $B: \mathfrak{g} \times \mathfrak{g} \to \R$ est une $\R$-forme symétrique non-dégénérée telle que
\begin{itemize}
  \item $B$ est invariante au sens que $B([X,Y],Z) + B(Y, [X,Z])=0$ pour tout $X,Y,Z$;
  \item $B$ est négative définie sur $\mathfrak{k} \times \mathfrak{k}$;
  \item $B$ est $\theta$-invariante.
\end{itemize}
Cela entraîne que $\mathfrak{p} \perp \mathfrak{g}$ par rapport à $B$. Remarquons que l'on n'exige pas que $B$ soit définie positive sur $\mathfrak{p}$ ici. Alors $B$ induit une $\C$-forme symétrique $B_\C$ sur $\mathfrak{g}_\C \times \mathfrak{g}_\C$. On peut définir une $\R$-forme symétrique non-dégénérée sur $\mathfrak{g}' \times \mathfrak{g}'$ comme la restriction de $B_\C$. Plus précisément
$$ B' := B_\C \circ (\Psi \times \Psi)^{-1} : \mathfrak{g}' \times \mathfrak{g}' \to \R. $$
En effet, $B' = B$ sur $\mathfrak{k}$; si l'on identifie momentanément $i\mathfrak{p}$ et $\mathfrak{p}$ en divisant par $i$, alors $B' = -B$ sur $i\mathfrak{p}$.

\begin{remark}\label{rem:forme-compacte}
  Indiquons que $B$ définit une densité $|\omega_B| \in |\topwedge \mathfrak{g}|$: on prend une base orthogonale $v_1, \ldots, v_r$ de $\mathfrak{g}$ telle que $B(v_i, v_i) = \pm 1$, alors sa base duale définit $|\omega_B| = |v^*_1 \wedge \cdots \wedge v^*_r|$. \textit{Idem} pour $B'$ et $\mathfrak{g}'$. C'est compatible avec le transfert $\omega \mapsto \omega'$ ci-dessus au sens que $|\omega_{B'}| = |i^{\dim_\R \mathfrak{p}} \omega_B|$. Cette construction est liée à la définition de la mesure canonique locale sur $\R$ au moyen de la forme compacte. En effet, soient $G$ un $\R$-groupe réductif et $\theta: G \to G$ une involution de Cartan. On construit le dual $(\mathfrak{g}', \theta)$ de $(\mathfrak{g}, \theta)$. Alors $\mathfrak{g}'$ est la forme compacte de $\mathfrak{g}$. Si l'on fixe $B$ pour $(\mathfrak{g}, \theta)$, alors $B'$ est négative définie.
\end{remark}

Supposons maintenant que $\mathfrak{g}$ est réductif. Soient
\begin{align*}
  \mathfrak{t} &= \mathfrak{c} \oplus \mathfrak{a}  \subset \mathfrak{g}, \\
  \mathfrak{t}' &= \mathfrak{c} \oplus i\mathfrak{a}  \subset \mathfrak{g}'
\end{align*}
des sous-algèbres de Cartan $\theta$-stables, décomposées selon les valeurs propres $+1, -1$ de $\theta$. On dit qu'elles sont en dualité. Pour tout $\lambda \in \mathfrak{t}^*$, on note $H_\lambda \in \mathfrak{t}$ l'élément caractérisé par
$$ B(H_\lambda, \cdot) = \lambda(\cdot). $$
\textit{Idem} pour $\mathfrak{t}' \subset \mathfrak{g}'$ et pour leurs complexifiées.

Choisissons une sous-algèbre de Borel dans $\mathfrak{g}_\C$ contenant $\mathfrak{t}_\C = \mathfrak{t}'_\C$. On définit donc les demi-sommes des racines positives $\rho$, $\rho'$ pour $(\mathfrak{g}_\C, \mathfrak{t}_\C)$ et $(\mathfrak{g}'_\C, \mathfrak{t}'_\C)$. D'autre part, on pose $\varpi := \prod_{\alpha > 0} H_\alpha$ et $\varpi' := \prod_{\alpha' > 0} H_{\alpha'}$ pour $(\mathfrak{g}_\C, \mathfrak{t}_\C)$ et $(\mathfrak{g}'_\C, \mathfrak{t}'_\C)$, respectivement. Selon notre construction, les égalités suivantes sont des tautologies.
\begin{gather}\label{eqn:varpi-rho-eq}
  \rho = \rho', \quad \varpi(\rho) = \varpi'(\rho').
\end{gather}

On est prêt à considérer les groupes complexes. Prenons
\begin{itemize}
  \item $G^c$: un $\R$-groupe réductif connexe anisotrope,
  \item $G_\C := G^c \times_\R \C$, vu comme un $\R$-groupe par restriction des scalaires.
\end{itemize}

C'est connu que $G^c$ est un sous-groupe compact maximal de $G_\C$. La conjugaison par rapport à la structure réelle $G^c$ donne une involution de Cartan $\theta \in \Aut_\R(G_\C)$. Ceci fournit une algèbre de Lie symétrique orthogonale $(\mathfrak{g}_\C, \theta)$. Le sous-espace $\mathfrak{k}$ sur lequel $\theta=\identity$ est $\mathfrak{g}^c$, plongé dans $\mathfrak{g}_\C = \mathfrak{g}^c \oplus i\mathfrak{g}^c$.

D'autre part, considérons le plongement diagonal $G^c \hookrightarrow G^c \times G^c$. L'image est l'ensemble des points fixes par $\theta: (x,y) \mapsto (y,x)$ (oui, le même symbole $\theta$). Au niveau d'algèbres de Lie, cela fournit une autre algèbre de Lie symétrique orthogonale $(\mathfrak{g}^c \times \mathfrak{g}^c, \theta)$. Le sous-espace $\mathfrak{k}$ sur lequel $\theta=\identity$ est le diagonal $\{ (X,X) : X \in \mathfrak{g}^c \}$, tandis que celui sur lequel $\theta=-\identity$ est l'anti-diagonal $\{(X,-X) : X \in \mathfrak{g}^c \}$.

Le résultat suivant est la dualité entre espaces symétriques riemanniens de type II et IV, pour l'essentiel.
\begin{proposition}\label{prop:dualite-2-4}
  Il existe un isomorphisme $\Psi: (G_\C) \times_\R \C \rightiso (G^c \times G^c) \times_\R \C$ qui respecte les involutions $\theta \otimes_\R \C$ et induit $\identity: G^c \rightiso G^c$ sur les points fixes par $\theta$.

  De plus, $(\mathfrak{g}^c \times \mathfrak{g}^c, \theta)$ peut s'identifier avec le dual $((\mathfrak{g}_\C)', \theta)$ de $(\mathfrak{g}_\C, \theta)$ de telle sorte que l'application induite par $\Psi$ sur l'algèbre de Lie devient l'isomorphisme $\Psi: \mathfrak{g}_\C \otimes_\R \C \rightiso (\mathfrak{g}_\C)' \otimes_\R \C$ dans la dualité des algèbres de Lie symétriques orthogonales.
\end{proposition}
\begin{proof}[Esquisse de la démonstration]
  L'identification repose sur l'isomorphisme $\C$-linéaire à droite
  \begin{align*}
    \C \otimes_\R \C & \stackrel{\sim}{\longrightarrow} \C \oplus \C \\
    z \otimes w & \longmapsto (zw, \bar{z}w)
  \end{align*}
  où $z \mapsto \bar{z}$ est la conjugaison complexe, et $\C$ opère sur $\C \otimes_\R \C$ (resp. $\C \oplus \C$) par la deuxième composante (resp. la multiplication diagonale). Ceci induit un isomorphisme canonique entre les foncteurs
  $$ \text{Res}_{\C/\R}(\cdot \times_\R \C) \times_\R \C \rightiso (\cdot \times_\R \C) \times (\cdot \times_\R \C) = ((\cdot) \times (\cdot)) \times_\R \C $$
  dans des catégories convenables des groupes algébriques, algèbres de Lie, etc. La conjugaison complexe opérant sur le premier facteur $\C$ à gauche correspond au morphisme d'échange à droite.
\end{proof}

Prenons une $\R$-forme symétrique invariante $B'_0: \mathfrak{g}^c \times \mathfrak{g}^c \to \R$ qui est définie négative. Cette forme définit une densité invariante, donc une mesure de Haar $\mu[B'_0]$ sur $G^c$ de manière familière: soient $u_1, \ldots, u_r$ une base orthonormée de $\mathfrak{g}^c$ par rapport à $-B'_0$ et $u^*_1, \ldots, u^*_r$ la base duale, alors $u^*_1 \wedge \cdots \wedge u^*_r$ donne la densité cherchée. On note
$$ \mes(G^c; B'_0) := \mes(G^c) \text{ par rapport à } \mu[B'_0] . $$

Pour $\mathfrak{g}^c \times \mathfrak{g}^c$, on pose $B' := \frac{1}{2} (B'_0 \times B'_0)$, de sorte que sa restriction via le plongement diagonal ou l'anti-diagonale $\mathfrak{g}^c \hookrightarrow \mathfrak{g}^c \times \mathfrak{g}^c$ est $B'_0$. De plus, $B'$ est invariante et $\theta$-invariante. Le bilan: on obtient une mesure $\mu[B']$ sur le groupe $G^c \times G^c$.

On en déduit des mesures de Haar sur $G_\C$ en deux façons.
\begin{itemize}
  \item[\textbf{Recette de Gross}.] Remarquons que $G^c \times G^c$ est la forme compacte du $\R$-groupe réductif connexe $G_\C$. On a vu dans \S\ref{sec:mesures-locales} que l'on peut transférer la mesure $\mu[B']$ vers $G_\C$ via $\Psi$. La mesure ainsi obtenue est notée $\mu^{G_\C}[B'_0]$.

  \item[\textbf{Recette de Harish-Chandra}.] La dualité des algèbres de Lie symétriques orthogonales et la Proposition \ref{prop:dualite-2-4} nous donne une $\R$-forme symétrique invariante $B: \mathfrak{g}_\C \times \mathfrak{g}_\C$ duale à $B'$. Elle est $\theta$-invariante. De plus
    \begin{itemize}
      \item $B=B'$ sur le sous-espace $\mathfrak{k} = \mathfrak{g}^c$ sur lequel $\theta=\identity$;
      \item $B$ est positive définie sur le sous-espace $\mathfrak{p}$ sur lequel $\theta=-\identity$.
    \end{itemize}
    La recette dans \S\ref{sec:mesure-HC} est donc applicable au quadruple $(G_\C, G^c, \theta, B)$. La mesure de Haar de $G_\C$ ainsi obtenue  est notée $\mu^{G_\C}[B]$.
\end{itemize}

\begin{lemma}\label{prop:changement-K}
  On a
  $$ \mes(G^c; B'_0) \cdot \mu^{G_\C}[B] = \mu^{G_\C}[B'_0]. $$
\end{lemma}
\begin{proof}
  Regardons le fibré $G_\C \twoheadrightarrow G^c \backslash G_\C$. La mesure $\mu^{G_\C}[B'_0]$ dans la recette de Gross est obtenue en mettant la mesure $\mu[B'_0]$ sur les fibres $\simeq G^c$, et la mesure sur la base $G^c \backslash G_\C$ est celle correspondant à la restriction de $B$ sur $\mathfrak{p}$. Ceci résulte de la Remarque \ref{rem:forme-compacte}.

  La mesure $\mu^{G_\C}[B]$ dans la recette de Harish-Chandra est obtenue en mettant la même mesure sur la base $G^c \backslash G_\C$, mais sur les fibres $\simeq G^c$ on met la mesure de masse totale $1$. L'assertion en découle.
\end{proof}

Choisissons un $\R$-tore maximal quelconque $T$ de $G^c$ et un sous-groupe de Borel $B_\C$ de $G_\C$ contenant $T_\C := T \times_\R \C$. Cela permet de définir les demi-sommes des racines $\rho^{G_c}$ et $\rho^{G_\C}$. Il faut prendre garde que $G_\C$ est vu comme un $\R$-groupe ici; les racines en question vivent donc dans $(\mathfrak{t}_\C \otimes_\R \C)^* = \mathfrak{t}^*_\C \oplus \mathfrak{t}^*_\C$. Remarquons aussi que $T_\C$ est un tore fondamental dans $G_\C$; en fait, il n'y a qu'une seule classe de conjugaison de tores maximaux dans $G_\C$.

\begin{lemma}\label{prop:L^G-cplx-general}
  Pour le quadruple $(G_\C, G^c, \theta, B)$ associé au choix de $B'_0$, la forme linéaire $\mathfrak{L}^{G_\C}$ dans la Remarque \ref{rem:L^G} est donnée par
  \begin{gather*}
    \mathfrak{L}^{G_\C}(\phi) = c(G_\C)^{-1} \cdot \lim_{\substack{t \in T_{\C, \mathrm{reg}} \\ t \to 1}} \left[ \partial(\varpi^{G_\C})\phi \right](t)
  \end{gather*}
  pour tout $\phi \in \mathcal{I}(G_\C)$, avec
  \begin{gather*}
    c(G_\C) := (-4\pi)^{\frac{1}{2} (\dim G^c - \rank(G^c))} |W(G_\C, T_\C)| \varpi^{G^c}(\rho^{G^c}).
  \end{gather*}
  Ici, les éléments $\varpi^{G^c} \in \mathrm{Sym}(\mathfrak{t}_\C)$ et $\varpi^{G_\C} \in \mathrm{Sym}(\mathfrak{t}_\C \otimes_\R \C)$ sont définis par rapport à $B'_0$ et $B$, respectivement.
\end{lemma}
\begin{proof}
  La constante de Harish-Chandra $c(G_\C)$ dans le Théorème \ref{prop:limite} est facile à évaluer. Tout ce qui reste à montrer est que $\Delta_I \equiv 1$. En effet, toute racine pour $(\mathfrak{g}_\C, \mathfrak{t}_\C) \otimes_\R \C$ est complexe: cela résulte de la démonstration de la Proposition \ref{prop:dualite-2-4} et de \cite[p.249]{KV95}.
\end{proof}

\begin{lemma}\label{prop:cplx-epinglage}
  Conservons les notations précédentes. On a
  \begin{align*}
    c(G_\C)^{-1} \varpi^{G_\C}(\rho^{G_\C}) & = (-\pi)^{-\frac{1}{2} (\dim G^c - \rank(G^c))} |W(G_\C, T_\C)|^{-1} \varpi^{G^c}(\rho^{G^c}) \\
    & \neq 0.
  \end{align*}
\end{lemma}
\begin{proof}
  On sait que $\varpi^{G^c}(\rho^{G^c}) > 0$ d'après la Remarque \ref{rem:positivite-cpt}. Vu ledit Lemme, il suffit de montrer que
  $$ \varpi^{G_\C}(\rho^{G_\C}) = 2^{\dim G^c - \rank(G^c)} \cdot \varpi^{G^c}(\rho^{G^c})^2. $$
  En effet, $B: \mathfrak{g}_\C \times \mathfrak{g}_\C \to \R$ correspond à $B' := 2^{-1} (B'_0 \times B'_0)$ via la dualité des algèbres de Lie symétriques orthogonales.  On se ramène à prouver $\varpi^{G_\C}(\rho^{G_\C}) = \varpi^{G^c}(\rho^{G^c})^2$ avec $\varpi^{G_\C}$ associé $2B$ au lieu de $B$, ce qu'assure \eqref{eqn:varpi-rho-eq}.
\end{proof}

\subsection{Preuve du transfert}
Tout d'abord, on se place dans la situation suivante.
\begin{itemize}
  \item $G = \Sp(2n)_{/\R}$, et $G^c$ sa forme intérieure compacte;
  \item $H = \SO(2n+1)_{/\R}$ déployé, et $H^c$ sa forme intérieure compacte;
  \item $G_\C := G^c \times_\R \C$, $H_\C := H^c \times_\R \C$, tous vus comme des $\R$-groupes réductifs connexes;
  \item $T$: un $\R$-tore maximal partagé par $G^c$ et $H^c$, ce qui existe;
  \item $T_\C := T \times_\R \C$: un $\R$-tore maximal partagé par $G_\C$ et $H_\C$.
\end{itemize}

L'involution de Cartan $\theta$ sur $G_\C$ (resp. $H_\C$) que nous considérons est la conjugaison complexe par rapport à la forme réelle $G^c$ (resp. $H^c$).

Selon les discussions avant le Lemme \ref{prop:2n} appliquées à $G$ et $H$, on dispose donc d'une base $\epsilon_1, \ldots, \epsilon_n$ de $\mathfrak{t}^*_\C$, la base duale $\eta_1, \ldots, \eta_n$ de $\mathfrak{t}_\C$, et des formes symétriques invariantes non-dégénérées $B^G_\text{tr}: \mathfrak{g} \times \mathfrak{g} \to \R$, $B^H_\text{tr}: \mathfrak{h} \times \mathfrak{h} \to \C$. Lesdites formes symétriques se transfèrent aux formes intérieures compactes $\mathfrak{g}^c$ et $\mathfrak{h}^c$ car elles sont invariantes. On les note par le même symbole $B_\text{tr}$. On voit que les deux $B_\text{tr}$ sont toutes négatives définies, car $B^G_\text{tr}$ et $B^H_\text{tr}$ sont des multiples positifs des formes de Killing.

De plus, leurs restrictions sur $\mathfrak{t} \times \mathfrak{t}$ est la même forme symétrique non-dégénérée, ce qui justifie la notation
\begin{align*}
  B_\text{tr}: & \mathfrak{t} \times \mathfrak{t} \longrightarrow \R \\
  B_\text{tr}(\eta_i, \eta_j) & = \delta_{i,j}, \quad 1 \leq i,j \leq n.
\end{align*}

Du point de vue du \S\ref{sec:limite-cplx}, les formes $B_\text{tr}$ (pour $G^c$ et $H^c$) jouent le rôle de $B'_0$. On en déduit une forme symétrique $B^{G_\C}: \mathfrak{g}_\C \times \mathfrak{g}_\C \to \R$ (resp. $B^{H_\C}: \mathfrak{h}_\C \times \mathfrak{h}_\C \to \R$) d'après la recette dans \S\ref{sec:limite-cplx}, ce qui définit des mesures de Haar sur $G_\C$ et $T_\C \subset G_\C$ (resp. $H_\C$ et $T_\C \subset T_\C$) selon la recette de Harish-Chandra.

\begin{lemma}\label{prop:T-coincide}
  Les mesures de Haar sur $T_\C$ (resp. $T$) induites par $B^{G_\C}$ et $B^{H_\C}$ (resp. $B_{\mathrm{tr}}$) coïncident.
\end{lemma}
\begin{proof}
  On décompose $\mathfrak{t}_\C$ dans $\mathfrak{g}_\C$ en $\mathfrak{t}_\C = \mathfrak{t} \oplus i\mathfrak{t}$. Vu la normalisation des mesures dans \S\ref{sec:mesure-HC}, il suffit de regarder la restriction de $B_\text{tr}$ sur la partie non-compacte $i\mathfrak{t}$. Il en est de même dans $\mathfrak{h}_\C$. On conclut car c'est la même forme $B_\text{tr}: \mathfrak{t} \times \mathfrak{t} \to \R$ qui intervient pour $\mathfrak{g}^c$ et $\mathfrak{h}^c$.

  Le cas pour $T$ et $B_\text{tr}$ est pareil.
\end{proof}

On fixe les systèmes des racines positives pour $(\mathfrak{g}_\C, \mathfrak{t}_\C)$ et $(\mathfrak{h}_\C, \mathfrak{t}_\C)$ de façon standard. D'après les discussions dans \S\ref{sec:limite-cplx}, cela permet de définir les demi-sommes de racines positives $\rho^{G^c}$, $\rho^{H^c}$, $\rho^{G_\C}$ et $\rho^{H_\C}$. C'est facile de décrire les racines pour $G_\C$ et $H_\C$ à l'aide des discussions avant le Lemme \ref{prop:2n}.

\begin{proposition}\label{prop:comparaison-cplx-0}
  Utilisons les notations du Lemme \ref{prop:L^G-cplx-general}.  Avec les formes symétriques $B_\mathrm{tr}$, on a l'égalité suivante dans $\mathrm{Sym}(\mathfrak{t}_\C)$
  $$ c(G_\C)^{-1} \varpi^{G_\C} = 2^{-2n} \cdot \dfrac{\mes(H^c; B_\mathrm{tr})}{\mes(G^c; B_\mathrm{tr})} \cdot c(H_\C)^{-1} \varpi^{H_\C}. $$
\end{proposition}
\begin{proof}
  A priori, les deux côtés ne diffèrent que par une constante multiplicative dans $\C^\times$ d'après la description des racines. Afin de l'épingler, on applique d'abord le Lemme \ref{prop:cplx-epinglage}. Puisque $W(G_\C, T_\C) = W(H_\C, T_\C)$ et $\dim G^c - \rank(G^c) = \dim H^c - \rank(H^c)$, on a
  \begin{align*}
    \dfrac{c(G_\C)^{-1} \varpi^{G_\C}(\rho^{H_\C})}{c(H_\C)^{-1} \varpi^{H_\C}(\rho^{H_\C})} & = \dfrac{\varpi^{G_\C}(\rho^{H_\C})}{\varpi^{G_\C}(\rho^{G_\C})} \cdot \dfrac{c(G_\C)^{-1} \varpi^{G_\C}(\rho^{G_\C})}{c(H_\C)^{-1} \varpi^{H_\C}(\rho^{H_\C})} \\
    & = \dfrac{\varpi^{G_\C}(\rho^{H_\C})}{\varpi^{G_\C}(\rho^{G_\C})} \cdot \frac{\varpi^{G^c}(\rho^{G^c})}{\varpi^{H^c}(\rho^{H^c})} \\
    & = \left( \dfrac{\varpi^{G^c}(\rho^{H^c})}{\varpi^{G^c}(\rho^{G^c})} \right)^2 \cdot \dfrac{\varpi^{G^c}(\rho^{G^c})}{\varpi^{H^c}(\rho^{H^c})}
  \end{align*}
  où on a aussi utilisé \eqref{eqn:varpi-rho-eq}; voir aussi la preuve du Lemme \ref{prop:cplx-epinglage}.

  Le premier quotient dans la dernière expression est déjà calculé dans le Lemme \ref{prop:2n} (avec $k=\C$): on a
  $$ \left( \dfrac{\varpi^{G^c}(\rho^{H^c})}{\varpi^{G^c}(\rho^{G^c})} \right)^2 = 2^{-2n}. $$

  Traitons l'autre terme. Observons que $G^c$ et $H^c$ ont le même nombre de racines; de plus, les mesures sur $T$ induites par $B_\text{tr}$ pour $G^c$ et $H^c$ sont pareilles d'après le Lemme \ref{prop:T-coincide}. D'après \eqref{eqn:HC}, on a
  $$ \dfrac{\varpi^{G^c}(\rho^{G^c})}{\varpi^{H^c}(\rho^{H^c})} = \dfrac{\mes(H^c; B_\mathrm{tr})}{\mes(G^c; B_\mathrm{tr})}. $$

  Cela achève la démonstration.
\end{proof}
\begin{remark}
  C'est peut-être utile de savoir que $\mes(H^c; B_\text{tr})/\mes(G^c; B_\text{tr})=2^{2n}$ bien que nous ne l'utilisons pas. Cela résulte par un calcul directe à l'aide de \cite{Mc80} et du Lemme \ref{prop:T-coincide}.
\end{remark}

Considérons maintenant les mesures canoniques locales dans \S\ref{sec:mesures-locales}. Fixons une mesure de Haar $\mu^{T_\C}$ sur $T_\C$ et munissons $T_\C \backslash G_\C$ (resp. $T_\C \backslash H_\C$) de la mesure quotient $\mu^{T_\C} \backslash \mu^{G_\C}$ (resp. $\mu^{T_\C} \backslash \mu^{H_\C}$), où $\mu^{G_\C}$ et $\mu^{H_\C}$ désignent les mesures canoniques locales.

Choisissons des formes symétriques convenables $B^{G_\C}_\text{can}$, $B^{H_\C}_\text{can}$ donnant lesdites mesures invariantes sur $T_\C \backslash G_\C$ et $T_\C \backslash H_\C$ par la recette de Harish-Chandra (rappelons \S\ref{sec:mesure-HC}). Ceci est toujours possible car $G_\C$ et $H_\C$ ne sont pas compacts, cf. la preuve du Lemme \ref{prop:v(M_I)}.

Une autre observation: la forme linéaire $\mathfrak{L}^{G_\C}$ (resp. $\mathfrak{L}^{H_\C}$) est déterminée par la mesure invariante sur $T_\C \backslash G_\C$ (resp. $T_\C \backslash H_\C$) par la recette de Harish-Chandra. De plus, $\mathfrak{L}^{G_\C}$ est inversement proportionnelle à la mesure invariante sur $T_\C \backslash G_\C$. C'est possible de le vérifier à la main par la définition de $\mathfrak{L}^{G_\C}$, mais c'est aussi une conséquence directe du Théorème \ref{prop:limite}.

\begin{corollary}\label{prop:egalite-L-cplx}
  Définissons $\mathfrak{L}^{G_\C}$ et $\mathfrak{L}^{H_\C}$ par les formes $B^{G_\C}_{\mathrm{can}}$, $B^{H_\C}_{\mathrm{can}}$ choisies ci-dessus. Alors on a l'égalité entre des opérateurs différentiels sur $T_\C$
  $$ \mathfrak{L}^{G_\C} = 2^{-2n} \mathfrak{L}^{H_\C}. $$
\end{corollary}
\begin{proof}
  Commençons par les mesures dans la Proposition \ref{prop:comparaison-cplx-0}. En termes des notations dans \S\ref{sec:limite-cplx}, on met les mesures $\mu^{G_\C}[B^{G_\C}]$ et $\mu^{H_\C}[B^{H_\C}]$ sur $G_\C$ et $H_\C$, respectivement. Les formes $B^{H_\C}$ et $B^{G_\C}$ induisent la même de Haar sur $T_\C$, ce que l'on fixe dans la preuve.

  Passons ensuite aux mesures $\mu^{G_\C}[B_\text{tr}]$ et $\mu^{H_\C}[B_\text{tr}]$. Grâce au Lemme \ref{prop:changement-K}, on a
  \begin{align*}
    \dfrac{\mu^{G_\C}[B_\text{tr}]}{\mu^{G_\C}[B^{G_\C}]} & = \mes(G^c; B_\text{tr}), \\
    \dfrac{\mu^{H_\C}[B_\text{tr}]}{\mu^{H_\C}[B^{H_\C}]} & = \mes(H^c; B_\text{tr}).
  \end{align*}

  Passons aux mesures canoniques locales $\mu^{G_\C}$ et $\mu^{H_\C}$. Traitons d'abord $G_\C$. Rappelons que $\mu^{G_\C}$ provient de la mesure canonique $\mu^{G_c} \otimes \mu^{G_c}$ sur $G^c \times G^c$. En le comparant avec la définition de $\mu^{G_\C}[B_\text{tr}]$, on déduit
  \begin{align*}
    \dfrac{\mu^{G_\C}}{\mu^{G_\C}[B_\text{tr}]} & =  2^{\dim G^c} \cdot \dfrac{\mu^{G^c} \otimes \mu^{G^c}}{\mu^{G^c}[B_\text{tr}] \otimes \mu^{G^c}[B_\text{tr}]} \\
    & = 2^{\dim G^c} \left( \dfrac{\mes(G^c; \mu^{G^c})}{\mes(G^c; B_\text{tr})} \right)^2
  \end{align*}

  On a besoin du fait que $\mes(G^c; \mu^{G^c}) = \mes(H^c; \mu^{H^c})$, ce qui découle de \cite[(7.4)]{Gr97} et du fait que $G^c$ et $H^c$ ont les mêmes exposants
  $$ 1, 3, \ldots, 2n-1 \quad (\text{sans multiplicité}). $$
  Comme $\dim G^c = \dim H^c$, on en déduit que
  $$ \dfrac{ \mu^{G_\C} / \mu^{G_\C}[B_\text{tr}] }{ \mu^{H_\C} / \mu^{H_\C}[H_\text{tr}]} =  \left( \dfrac{\mes(H^c; B_\text{tr})}{\mes(G^c; B_\text{tr})} \right)^2 . $$

  D'où
  $$ \dfrac{ \mu^{G_\C} / \mu^{G_\C}[B^{G_\C}] }{ \mu^{H_\C} / \mu^{H_\C}[B^{H_\C}] } = \dfrac{\mes(H^c; B_\text{tr})}{\mes(G^c; B_\text{tr})}. $$

  Rappelons que l'on n'a jamais touché la mesure de Haar sur $T_\C$. Vu la proportionnalité inverse entre $\mathfrak{L}^{G_\C}$, $\mathfrak{L}^{H_\C}$ et les mesures invariantes, la Proposition \ref{prop:comparaison-cplx-0} donne
  \begin{align*}
  \mathfrak{L}^{G_\C}  & = 2^{-2n} \cdot \dfrac{\mes(H^c; B_\mathrm{tr})}{\mes(G^c; B_\mathrm{tr})} \cdot \dfrac{ \mu^{H_\C} / \mu^{H_\C}[B^{H_\C}] }{ \mu^{G_\C} / \mu^{G_\C}[B^{G_\C}] } \cdot \mathfrak{L}^{H_\C} \\
  & = 2^{-2n} \mathfrak{L}^{H_\C},
  \end{align*}
  ce qu'il fallait démontrer.
\end{proof}

Reprenons maintenant le formalisme de \S\ref{sec:descente}: on a $\gamma \in H(\C)_\text{ss}$ en correspondance équi-singulière avec $\delta \in G(\C)_\text{ss}$, etc. Il faut changer les notations quelque peu: on travaille avec les groupes
\begin{align*}
  G & := \Sp(W, \C), \quad \dim_\C W = 2n, \\
  \tilde{G} & := \Mp(W, \C) = \bmu_8 \times G, \\
  H = H_{n',n''} & :=  \SO(2n'+1, \C) \times \SO(2n''+1, \C).
\end{align*}
On fixe un tore fondamental $T$ partagé par $G$ et $H$ muni d'une mesure de Haar. On peut aussi supposer que $T$ est stable par les involutions de Cartan pour $G$ et $H$ par rapport à leurs formes réelles compactes. Les groupes $G, H, G_\delta, H_\gamma$ sont tous munis des mesures canoniques locales. Il n'y a plus de différences entre conjugaison et conjugaison stable, et les signes de Kottwitz $e(\cdot)=1$. L'égalité \eqref{eqn:descente-1}=\eqref{eqn:descente-2} s'écrit donc comme
\begin{gather}\label{eqn:descente-cplx}
  J_{H_\gamma}(\exp(Y), f^{H,\natural}_\gamma) = \Delta(\gamma, \tilde{\delta}) J_{G_\delta}(\exp(X) f^\natural_\delta)
\end{gather}
où $X$ et $Y$ sont réguliers et proviennent du même élément dans $\mathfrak{t}$.

\begin{proof}[Démonstration du Théorème \ref{prop:transfert-equi} pour $F=\C$]
  Montrons d'abord l'égalité des formes linéaires
  \begin{gather}\label{eqn:egalite-L-cplx}
    \mathfrak{L}^{G_\delta} = 2^{-(\dim_\C V'_+ + \dim_\C V''_+ - 2)} \mathfrak{L}^{H_\gamma}
  \end{gather}
  où on prend des formes symétriques invariantes $B^{G_\delta}_\text{can}$ et $B^{H_\gamma}_\text{can}$ induisant les mesures prescrites invariantes sur $T \backslash G_\delta$ et $T \backslash H_\gamma$, respectivement.

  Décomposons $G_\delta$ et $H_\gamma$ selon \eqref{eqn:descente-commutants}. Il y a des décompositions parallèles pour $B^{G_\delta}_\text{can}$, $B^{H_\gamma}_\text{can}$ et $T$. Puisque $U_G = U_H$, l'identification $\mathfrak{L}^{U_G} = \mathfrak{L}^{U_H}$ est une tautologie. Les situations de $\Sp(W_+)$ et $\Sp(W_-)$ sont en symétrie. On se ramène ainsi au cas $\delta=1$, $\gamma=1$. L'égalité à montrer devient $\mathfrak{L}^G = 2^{-2n} \mathfrak{L}^H$, ce qui est précisément le Corollaire \ref{prop:egalite-L-cplx}.

  Pour conclure, appliquons la forme linéaire \eqref{eqn:egalite-L-cplx} à \eqref{eqn:descente-cplx}. Selon la Remarque \ref{rem:L^G} et le Corollaire \ref{prop:remonter}, on arrive à
  \begin{multline*}
    \Delta(\gamma, \tilde{\delta}) J_{\tilde{G}}(\tilde{\delta}, f) = \Delta(\gamma, \tilde{\delta}) f^\natural_\delta(1) \\
    = 2^{-(\dim_\C V'_+ + \dim_\C V''_+ - 2)} f^{H,\natural}_\gamma(1) = 2^{-(\dim_\C V'_+ + \dim_\C V''_+ - 2)} J_H(\gamma, f^H)
  \end{multline*}
  où on a utilisé le fait que $\Delta \equiv 1$ dans le cas complexe. C'est exactement l'assertion du Théorème \ref{prop:transfert-equi} si l'on se rappelle que $e(\cdots) = 0$, $J_H(\gamma, f^H) = S_H(\gamma, f^H)$ et $|z|_\C := z\bar{z}$ pour tout $z \in \C$.
\end{proof}

\section{Transfert équi-singulier non-archimédien}\label{sec:nonarch}
Dans cette section, $F$ désigne un corps local non-archimédien de caractéristique nulle.

\subsection{Formule de Rogawski}\label{sec:Rogawski}
Soit $G$ un $F$-groupe réductif connexe dans cette sous-section. Le plus grand sous-tore déployé de $Z^\circ_G$ est désigné par $A_G$. Soit $T$ un $F$-tore maximal fondamental de $G$; c'est-à-dire que $T/A_G$ est anisotrope (cf. la Définition \ref{def:fondamental}). De tels tores existent.

Fixons une mesure de Haar $\nu^{G/A_G}$ de $(G/A_G)(F)$, et prenons la mesure de Haar sur $(T/A_G)(F)$ de masse totale $1$ si $T$ est fondamental.

Observons que $T(F) \backslash G(F) = (T/A_G)(F) \backslash (G/A_G)(F)$ car $A_G$ est déployé. On a donc fixé les mesures pour définir les intégrales orbitales non-normalisées $\mathscr{O}^G_t(\cdot)$ et celles normalisées $J_G(t, \cdot)$, pour $t \in T_\text{reg}(F)$.

Rappelons le développement de Shalika le long de $T$
\begin{gather}\label{eqn:dev-Shalika}
  \mathscr{O}^G_t(f) = \sum_{u} \Gamma^T_u(t) \mathscr{O}^G_u(f), \quad t \in \mathcal{U} \cap T_\text{reg}(F), \; f \in C^\infty_c(G),
\end{gather}
Ici, $\mathcal{U}$ est un voisinage de $1$ dépendant de $f$ et $u$ parcourt les classes de conjugaison unipotentes dans $G(F)$. On fixe une mesure de Haar sur chaque commutant $(G_u/A_G)(F)$ pour définir les germes de Shalika $\Gamma^T_u(\cdot)$ et les intégrales orbitales unipotentes non-normalisées $\mathscr{O}^G_u(f)$. Ce choix est insignifiant car nous ne considérerons que le cas $u=1$, pour lequel $\mathscr{O}^G_u(f) = f(1)$ et le germe $\Gamma^T_1(\cdot)$ est bien déterminé par $\nu^{G/A_G}$.

On extraira la partie avec $u=1$ à l'aide de l'homogénéité des germes suivante.
\begin{proposition}\label{prop:homogeneite}
  Les germes $\Gamma^T_u(\cdot)$ satisfait à
  $$ \Gamma^T_u(\exp(t^2 X)) = |t|^{-\dim G/G_u}_F \Gamma^T_u(\exp X) $$
  pour tout $X \in \mathfrak{t}_{\mathrm{reg}}(F)$ proche de $0$ et tout $t \in \mathfrak{o}_F$.
\end{proposition}

\begin{theorem}[Rogawski {\cite{Ro81}}]\label{prop:Rogawski}
  Si $\nu^{G/A_G} = (-1)^{q(G)} \mu^{G/A_G}_{\mathrm{EP}}$, on a
  $$ \Gamma^T_1(t) = (-1)^{q(G)} $$
  pour $t \in T_{\mathrm{reg}}(F)$ suffisamment proche de $1$. En général on peut écrire $\Gamma^T_1(t) = (-1)^{q(G)}d(\pi_0)^{-1}$ où $\pi_0$ est la représentation de Steinberg de $(G/A_G)(F)$ est $d(\cdot)$ signifie le degré formel.
\end{theorem}
Cf. Lemme \ref{prop:deg-EP}.

\begin{remark}\label{rem:L^G-nonarch}
  Comme pour la formule de limite de Harish-Chandra (Théorème \ref{prop:limite}), introduisons l'espace des intégrales orbitales normalisées $\mathcal{I}(G)$ et définissons une forme linéaire
  $$ \mathfrak{L}^G : \mathcal{I}(G)  \to \C $$
  de la façon suivante. Soit $\phi = J_G(\cdot, f) \in \mathcal{I}(G)$. On considère sa restriction sur $T_\text{reg}(F)$. Le comportement local de $\phi$ au voisinage de $1$ est décrit par le développement de Shalika. La Proposition \ref{prop:homogeneite} donne, pour $X \in \mathfrak{t}_\text{reg}(F)$ fixé,
  \begin{gather}\label{eqn:homogeneite} 
    \phi(\exp(t^2 X)) = |D^G(\exp(t^2 X))|^{1/2}_F \cdot \sum_{j \geq 0} |t|^{-j}_F \phi^T_j(X), \quad t \in F^\times, \; |t|_F \ll 1,
  \end{gather}
  pour une famille unique des coefficients $\phi^T_j(X)$. De plus, le Théorème \ref{prop:Rogawski} dit que $\phi^T_0(X)$ ne dépend pas de $X$. On pose
  \begin{align*}
    \mathfrak{L}^G(\phi) = (-1)^{q(G)} d(\pi_0) \phi^T_0(X),
  \end{align*}
  ce qui est bien défini et ne dépend que du comportement de $\phi$ près de $1$. Alors le Théorème \ref{prop:Rogawski} entraîne
  $$ \mathfrak{L}^G(J_G(\cdot, f)) = f(1). $$

  Il convient de rappeler l'aspect (iii) de la Remarque \ref{rem:trois-aspects}: ici, on regarde que la restriction de $\phi \in \mathcal{I}(G)$ aux tores maximaux fondamentaux. Cela permettra de comparer ces formes linéaires associées à des groupes différents.
  
  De plus, la forme linéaire est stable au sens qu'elle se factorise par $S\mathcal{I}(G)$; c'est démontré dans \cite[\S 3]{Ko88}.
\end{remark}

\subsection{Preuve du transfert}\label{sec:preuve-nonarch}
Commençons par une variante de la Proposition \ref{prop:torsion-int}.

\begin{proposition}\label{prop:torsion-int-nonarch}
  Soient $I_1$, $I_2$ deux $F$-groupes réductifs connexes reliés par une torsion intérieure $\Psi: I_1 \times_F \bar{F} \rightiso I_2 \times_F \bar{F}$. Supposons que $T$ est un $F$-tore maximal fondamental partagé par $I_1$ et $I_2$ via $\Psi$; on peut donc supposer que la classe associée à $\Psi$ provient de $H^1(F, T/Z)$ où $Z$ est le centre de $I_1$ et $I_2$. On note $A$ la partie déployée de $Z^\circ$.

  Munissons $(I_1/A)(F)$ et $(I_2/A)(F)$ des mesures de Haar
  \begin{gather*}
    |\mathfrak{D}(T, I_i; F)|^{-1} (-1)^{q(I_i)} \mu^{I_i/A}_{\mathrm{EP}}, \quad i=1,2,
  \end{gather*}
  et munissons $(T/A)(F)$ de la mesure de Haar de masse totale $1$, alors
  \begin{enumerate}[i)]
    \item les mesures de Haar sur $I_1(F)$ et $I_2(F)$ sont compatibles avec la torsion intérieure $\Psi$;
    \item pour tout $\phi \in S\mathcal{I}(I_1) = S\mathcal{I}(I_2)$, en tant qu'une fonction sur $T_{\mathrm{reg}}(F)$, on a
      $$ |\mathfrak{D}(T, I_1; F)|^{-1} e(I_1) \mathfrak{L}^{I_1}(\phi) = |\mathfrak{D}(T, I_2; F)|^{-1} e(I_2) \mathfrak{L}^{I_2}(\phi) $$
    avec les mesures de Haar prescrites ci-dessus.
  \end{enumerate}
\end{proposition}

On note que les ensembles $\mathfrak{D}(T, I_i; F)$ et $\mathfrak{D}(T, I_2; F)$ sont en bijection dans le cas non-archimédien \cite[p.631]{Ko88}. D'autre part, comme dans le cas archimédien, l'identification de $S\mathcal{I}(I_1)$ et $S\mathcal{I}(I_2)$ résulte du transfert des intégrales orbitales entre formes intérieures.

\begin{proof}
  Les assertions sont implicites dans \cite[Lemma 2.4A]{LS2}. La première assertion est \cite[Theorem 1]{Ko88}. Traitons la deuxième assertion. On peut supprimer les $|\mathfrak{D}(T, I_i; F)|$ dans l'assertion ($i=1,2$); on peut aussi remplacer $e(I_i)$ par $(-1)^{q(I_i)}$ pour $i=1,2$. Alors l'égalité cherchée est immédiate en vue de l'expression de $\mathfrak{L}^{I_1}$, $\mathfrak{L}^{I_2}$ dans la Remarque \ref{rem:L^G-nonarch} et du fait simple que $D^{I_1}(\exp(X)) = D^{I_2}(\exp(X))$ si $X \in \mathfrak{t}(F)$.
\end{proof}

Reprenons maintenant le formalisme de \S\ref{sec:descente}, notamment l'égalité \eqref{eqn:descente-1}=\eqref{eqn:descente-2}, avec $F$ non-archimédien. On a $\gamma \leftrightarrow \delta$ en correspondance équi-singulière. Cependant, il faut réconcilier les choix des mesures de Haar. 

\begin{itemize}
  \item[\textbf{Recette A}.] Dans le Théorème \ref{prop:transfert-equi}, on a choisi les mesures canoniques locales sur les groupes $G(F)$, $G_\delta(F)$ et $H_{\gamma'}(F)$. De plus, les groupes $G_\delta$ et $H_{\gamma'}$ partagent la même partie déployée $A$ de leurs centres connexes; en effet, les centres connexes sont contenus dans les composantes $U_G$, $U_H$ dans \eqref{eqn:descente-commutants} qui sont reliées par torsion intérieure.

    On munit $A(F)$ de la mesure canonique locale et prenons les mesures quotients par $A(F)$ partout. Vu la Proposition \ref{prop:mes-can-prop}, la procédure donne des mesures canoniques locales sur
    $$ (G_\delta/A)(F), (H_{\gamma'}/A)(F). $$
    Enfin, on prend la mesure sur $T(F)$ (le tore $T$ est celui pris dans Lemme \ref{prop:plongements}) de sorte que $\mes((T/A)(F))=1$. Ces choix des mesures sont compatibles avec la torsion intérieure.

  \item[\textbf{Recette B.}] Dans la preuve du Lemme \ref{prop:egalite-L-nonarch} suivant, qui est crucial, on va utiliser la mesure de Haar prescrite dans la Proposition \ref{prop:torsion-int-nonarch} sur un groupe $\star = G_\delta/A$, $H_{\gamma'}/A$ ou $T/A$, à savoir
    $$ |\mathfrak{D}(T/A, \star; F)|^{-1} (-1)^{q(\star)} \mu^{\star}_{\mathrm{EP}}. $$
    En particulier, on a toujours $\mes((T/A)(F))=1$ pour cette recette.

  \item[\textbf{Réconciliation}.] Ces prescriptions de mesures sont différentes. Heureusement, les groupes $H_{\gamma'}/A$ et $G_\delta/A$ pour tous $\delta$ et $\gamma'$ ont le même motif d'Artin-Tate, noté $M$, d'après les résultats dans \S\ref{sec:motifs}. Il résulte de l'équation fonctionnelle locale (Théorème \ref{prop:eq-fonc-locale}) que, pour passer de la recette A à B, il suffit de le multiplier par une constante non nulle $c(M, T)$ qui ne dépend que du motif $M$ et du tore fondamental $T$. Ici on a utilisé le fait simple $q(\Sp(2k))=q(\SO(2k+1))=k$ pour tout $k$.
\end{itemize}

\begin{lemma}\label{prop:egalite-L-nonarch}
  Soit $\delta \in G(F)_{\mathrm{ss}}$ qui correspond à $\gamma$. Notons $A$ la partie déployée en commun de $G_\delta$ et $H_\gamma$. Munissons $(G_\delta/A)(F)$, $(H_\gamma/A)(F)$ et $(T/A)(F)$ des mesures de Haar selon l'une des recettes A et B dans la discussion précédente. Pour tout $\phi \in S\mathcal{I}(\mathcal{U}_\delta) \cap S\mathcal{I}(\mathcal{U}_\gamma)$, vu comme une fonction sur les éléments de $T_{\mathrm{reg}}(F)$ proche de $1$, on a
  $$ |\mathfrak{D}(T, G_\delta; F)|^{-1} e(G_\delta) \mathfrak{L}^{G_\delta}(\phi) = |2|^{-\frac{1}{2}_F (\dim_F V'_+ + \dim_F V''_+ - 2)} \cdot |\mathfrak{D}(T, H_\gamma; F)|^{-1} e(H_\gamma) \mathfrak{L}^{H_\gamma}(\phi). $$
\end{lemma}
\begin{proof}
  Vu la conclusion des discussions précédentes, il suffit de prouver l'égalité avec la recette B.

  Comme dans la démonstration, on décompose les formes $\mathfrak{L}^{G_\delta}$ et $\mathfrak{L}^{H_\gamma}$ selon la décomposition \eqref{eqn:descente-commutants}. D'après le Lemme \ref{prop:torsion-int-nonarch}, on peut oublier les composantes $U_G$, $U_H$ qui sont reliées par torsion intérieure, et puis traiter séparément la comparaison entre $\Sp(W_+)$ et $\SO(V'_+, q'_+)$ (resp. $\Sp(W_-)$ et $\SO(V''_-, q''_+)$). Par symétrie, on se ramène au cas $\delta=1$, $\gamma=1$. Notons que $e(G)=e(H)=1$.

  Il reste à montrer que
  $$ |\mathfrak{D}(T, G; F)|^{-1} \mathfrak{L}^G(\phi) = |2|^{-n}_F \cdot |\mathfrak{D}(T, H; F)|^{-1} \mathfrak{L}^H(\phi) $$
  pour tout $\phi \in S\mathcal{I}(\mathcal{U}_\delta) \cap S\mathcal{I}(\mathcal{U}_\gamma)$.

  Rappelons l'expression pour les formes linéaires $\mathfrak{L}^G$ et $\mathfrak{L}^H$ dans la Remarque \ref{rem:L^G-nonarch} avec nos choix de mesures: on a
  \begin{equation}\label{eqn:L-comparaison-nonarch}
  \begin{split}
    |\mathfrak{D}(T, G; F)|^{-1} \mathfrak{L}^G(\phi) & = (-1)^{q(G)} \phi^{T,G}_0(X), \\
    |\mathfrak{D}(T, H; F)|^{-1} \mathfrak{L}^H(\phi) & = (-1)^{q(H)} \phi^{T,H}_0(X),
  \end{split}\end{equation}
  où $\phi^{T,G}_0(X)$ (resp. $\phi^{T,H}_0(X)$) est le germe de degré $0$ dans le développement de Shalika de $|D^G(\exp X)|^{-1/2} \cdot \phi$ (resp. $|D^H(\exp X)|^{-1/2} \cdot \phi$); cf. \eqref{eqn:homogeneite}. On a donc
  $$ \phi^{T,G}_0(X) = \left| \dfrac{D^G(\exp X)}{D^H(\exp X)} \right|^{-1/2}_F \phi^{T,H}_0(X) $$
  où $X \in \mathfrak{t}_\text{reg}(F)$ est suffisamment proche de $0$. D'après la description standard des systèmes de racines pour $G$ et $H$ que nous avons rappelé avant le Lemme \ref{prop:2n}, on a
  \begin{align*}
    \dfrac{D^G(\exp X)}{D^H(\exp X)} & = \prod_{i=1}^n \dfrac{1 - e^{2\epsilon_i(X)}}{1 - e^{\epsilon_i(X)}} \cdot \dfrac{1 - e^{-2\epsilon_i(X)}}{1 - e^{-\epsilon_i(X)}} \\
    & = \prod_{i=1}^n \left( 1+e^{\epsilon_i(X)} \right) \left( 1+e^{-\epsilon_i(X)} \right) \\
    & \xrightarrow{\text{dans } F} 2^{2n} \qquad \text{ lorsque } X \to 0.
  \end{align*}

  D'où $\phi^{T,G}_0(X) = |2|^{-n}_F \phi^{T,H}_0(X)$ pour $X$ suffisamment proche de $0$. D'autre part, $q(G)=q(H)=n$. L'assertion résulte donc de \eqref{eqn:L-comparaison-nonarch}.
\end{proof}

On peut finir la démonstration du Théorème \ref{prop:transfert-equi} de façon familière.

\begin{proof}[Démonstration du Théorème \ref{prop:transfert-equi} pour $F$ non-archimédien]
  Posons
  $$ t := \frac{1}{2} (\dim_F V'_+ + \dim_F V''_+ - 2).$$
  Utilisons la recette A dans les discussions avant le Lemme \ref{prop:egalite-L-nonarch}. C'est-à-dire nous travaillons avec les mesures canoniques locales.

  On applique à \eqref{eqn:descente-1} la forme linéaire $2^{-t} |\mathfrak{D}(T, H_\gamma; F)|^{-1} e(H_\gamma) \mathfrak{L}^{H_\gamma}$ de la Remarque \ref{rem:L^G-nonarch}. D'après la Remarque \ref{rem:L^G-nonarch} et la Proposition \ref{prop:torsion-int-nonarch}, on obtient la somme sur toute classe de conjugaison $\gamma'$ dans la classe stable de $\gamma$, des expressions
  $$ 2^{-t} e(H_{\gamma'}) f^{H,\natural}_{\gamma'}(1) = 2^{-t} e(H_{\gamma'}) J_H(\gamma', f^H). $$
  Autrement dit, on obtient $2^{-t} S_H(\gamma, f^H)$. Le facteur $|\mathfrak{D}(T, H_\gamma; F)|^{-1}$ a disparu car on applique $\mathfrak{L}^{H_\gamma}$ à des intégrales orbitales stables: voir Remarque \ref{rem:trois-aspects} (iii).

  Vu le Lemme \ref{prop:egalite-L-nonarch}, la même forme linéaire appliquée à \eqref{eqn:descente-2} nous donne la somme sur toute classe de conjugaison $\delta$ correspondant à $\gamma$, des expressions
  $$ \Delta(\gamma, \tilde{\delta}) e(G_\delta) f^\natural_\delta(1) = \Delta(\gamma, \tilde{\delta}) e(G_\delta) J_{\tilde{G}}(\tilde{\delta}, f). $$
  On obtient ainsi le côté à gauche de l'égalité du Théorème \ref{prop:transfert-equi}, ce qu'il fallait démontrer.
\end{proof}

\bibliographystyle{abbrv-fr}
\bibliography{stable}

\begin{flushleft}
  Wen-Wei Li \\
  Insitute of Mathematics, \\
  Academy of Mathematics and Systems Science, Chinese Academy of Sciences, \\
  55, Zhongguancun East Road, \\
  100190 Beijing, China. \\
  Adresse électronique: \href{mailto:wwli@math.ac.cn}{\texttt{wwli@math.ac.cn}}
\end{flushleft}

\end{document}